%% file: BBandBBS.tex
\documentclass{GJA}

\usepackage[T1]{fontenc}
\usepackage{lmodern}
\usepackage{pdfpages}
\usepackage{academicons}

\usepackage[dvipsnames,HTML]{xcolor} 
\definecolor{orcidlogocol}{HTML}{A6CE39}
\definecolor{darkred}{cmyk}{0.0, 0.7, 1.0, 0.0}
\definecolor{darkblue}{rgb}{0.0, 0.0, 0.55}

\usepackage{listings}
\usepackage{enumitem}

\lstdefinestyle{cocoaStyle}{
    basicstyle=\scriptsize\ttfamily\color{black},
    commentstyle=\color{red},
    stringstyle=\color{red},
    morecomment=[l]{--},
    moredelim=[is][\color{red}]{||}{||},
    breaklines=true,
    columns=fullflexible,
    xleftmargin=2em, 
}

\usepackage{mathtools} 
\usepackage{thmtools}
\usepackage{pictexwd}

\declaretheoremstyle[
  headfont=\bfseries,   
  notefont=\bfseries,   
  bodyfont=\normalfont, 
  headpunct={.}
]{boldremark}

\declaretheorem[style=boldremark,sibling=theorem,name=Remark]{myremark}

\newenvironment{myenumerate}
{
  \begin{enumerate}[
    label=(\arabic*),
    ref=(\arabic*),
    topsep=0pt,        
    partopsep=0pt,     
    parsep=0pt,        
    itemsep=2pt,       
    before=\vspace{-2pt} 
  ]
}
{
  \end{enumerate}
}

\newcommand{\propTitle}[1]{%
  \begin{proposition}\textbf{\textup{(#1)}}%
  \label{#1} 
  \leavevmode \\ \indent \hspace{-\parindent}%
}

\allowdisplaybreaks
\newcommand{\OO}{{\mathcal{O}}}
\newcommand{\X}{{\mathbb{X}}}
\renewcommand{\AA}{{\mathbb{A}}}
\newcommand{\NN}{{\mathbb{N}}}
\newcommand{\ZZ}{{\mathbb{Z}}}
\newcommand{\QQ}{{\mathbb{Q}}}
\newcommand{\TT}{{\mathbb{T}}}

\newcommand{\M}{{\mathfrak{M}}}

\newcommand{\m}{{\mathfrak{m}}}

\def\ind{\mathop{\rm ind}\nolimits}
\def\LI{\mathop{\rm LI}\nolimits}
\def\indO{\ind_{\OO}}
\def\bbb#1{{\mathbb{#1}}}
\def\Cal#1{{\mathcal{#1}}} 

\let\epsilon=\varepsilon

\def\phi{{\varphi}}
\let\Psi=\varPsi
\let\Phi=\varPhi
\let\theta=\vartheta
\let\rho=\varrho

\def\LT{\mathop{\rm LT}\nolimits}

\def\NF{\mathop{\rm NF}\nolimits}
\def\NR{\mathop{\rm NR}\nolimits}

\def\ND{\mathop{\rm ND}\nolimits}
\def\AR{\mathop{\rm AR}\nolimits}
\def\AS{\mathop{\rm AS}\nolimits}
\def\ARP{\mathop{\rm ARP}\nolimits}
\def\ASP{\mathop{\rm ASP}\nolimits}
\def\GL{\mathop{\rm GL}\nolimits}

\def\Mat{\mathop{\rm Mat}\nolimits}
\def\Hom{\mathop{\rm Hom}\nolimits}

\def\Supp{\mathop{\rm Supp}\nolimits}
\def\Spec{\mathop{\rm Spec}\nolimits}
\def\Hilb{\mathop{\rm Hilb}\nolimits}
\def\Cot{\mathop{\rm Cot}\nolimits}

\def\GFan{\mathop{\rm GFan}\nolimits}
\def\LTGFan{\mathop{\rm LTGFan}\nolimits}

\def\Syz{\mathop{\rm Syz}\nolimits}
\def\Ker{\mathop{\rm Ker}\nolimits}

\def\rk{\mathop{\rm rk}\nolimits}
\def\wt{\mathop{\rm wt}\nolimits}

\def\ord{{\mathop{\rm ord}\nolimits}}

\def\lcm{\mathop{\rm lcm}\nolimits}
\def\cvec{{\mathop{\rm cvec}\nolimits}}
\def\cmat{{\mathop{\rm cmat}\nolimits}}

\def\Ann{\mathop{\rm Ann}\nolimits}
\def\LGor{\mathop{\rm LocGor}\nolimits}
\def\NonLGor{\mathop{\rm NonLocGor}\nolimits}

\newcommand{\Lin}{\mathop{\rm Lin}\nolimits}
\newcommand{\edim}{\mathop{\rm edim}\nolimits}
\newcommand{\sepdim}{\mathop{\rm sepdim}\nolimits}
\newcommand{\BO}{\mathbb{B}_{\OO}}
\newcommand{\BOhom}{\mathbb{B}_{\OO}^{\rm hom}}
\newcommand{\BOdf}{\mathbb{B}_{\OO}^{\rm df}}

\newcommand{\LBO}{{\rm L}\BO}
\newcommand{\OOrim}{\OO^{\rm rim}}
\newcommand{\OOint}{\OO^{\rm int}}
\newcommand{\Cint}{C^{\rm int}}
\newcommand{\Crim}{C^{\rm rim}}
\newcommand{\Cexp}{C^{\rm exp}}

\newcommand{\df}{^{\rm df}}

\newcommand{\Cnull}{C_0}

\def\m{{\mathfrak{m}}}

\def\M{\mathfrak{M}}

\def\q{{\mathfrak{q}}}

\def\F{{\mathcal{F}}}

\def\Fbar{{\mathcal{F}}}

\def\longmono{\lhook\joinrel\longrightarrow}

\def\hom{^{\rm hom}}
\def\tr{^{\textsf{T}}}
\def\tfrac #1#2{{\textstyle\frac{#1}{#2}}}
\def\tsum_#1^#2{{\textstyle\sum\limits_{#1}^{#2}}}
\def\tbinom #1#2{{\textstyle\binom{#1}{#2}}}

\def\To{\longrightarrow}

\def\cocoa{\mbox{\rm C\kern-.13em o\kern-.07 em C\kern-.13em o\kern-.15em A}}
\def\Apcocoa{\mbox{\rm A\kern-0.13em p\kern -0.07em C\kern-.13em o\kern-.07 
em C\kern-.13em o\kern-.15em A}}

\usepackage{hyperref}
\hypersetup{
    linkcolor=black,   
    citecolor=red,
    filecolor=black,
    urlcolor=blue,
}


\newcommand{\cocoasection}[1]{\section[\small #1]{#1}}

\begin{document}

\title[BB and BBS]{Border Bases and Border Basis Schemes}

\author[L. Robbiano]{Lorenzo Robbiano\textsuperscript{1}}


\address{$^{1}$Dipartimento di Matematica, Universit\`a di Genova, 
Via Dodecaneso 35, 16146, Genova, Italy}
\email{lorobbiano@gmail.com}


\received{*****}
\revised{*****}
\accepted{*****}
\published{*****}
\communicated{******}
\doi{*********************}
\dedicatory{}
\na{a1}

\abstract{
This survey invites the readers on a journey spanning more than twenty years of 
research through the landscape of border basis schemes. 
During most of this period, I had the pleasure of working with Martin Kreuzer, 
and more recently with L\^e Ng\d{o}c Long. Along this journey, one encounters border bases, 
which generalize Gr\"obner bases for zero-dimensional ideals 
and are characterized by the remarkable property that their associated 
multiplication matrices commute pairwise.
This property, in turn, provides a natural foundation for defining border basis schemes (BBS).
These are beautiful schemes, elegantly defined by simple quadratic equations.
However, all that glitters is not gold, and the number of indeterminates in their 
coordinate rings can be huge. This necessitates a suitable re-embedding into 
a polynomial ring with fewer indeterminates. 
This task is accomplished in a more general setting, and the reward is reaped in the BBS scenario, 
where the notions of cotangent equivalence and exposed indeterminates
play a fundamental role, for instance, in showing that planar
Box BBS are affine cells. Then, the journey takes another detour. 
The center stage is taken by positive $P_0$-algebras and the 
unimodular matrix problem, which allows
us to prove that  regular algebras of this kind are free, i.e., isomorphic to a polynomial ring.
Finally, we turn our attention to special BBS and interesting subschemes of BBS.
Are there no more open problems? Fortunately, many remain, and 
a selection of these challenges marks not a final destination, but a new horizon, 
suggesting that this journey may continue in the near future.
\begin{center}
\small\it
\begin{tabular}{r @{\hskip 1cm} l}
&on the border of the soul \\[1pt]
&bases of unseen photographs, \\[1pt]
&schemes of ancient thoughts, \\[1pt]
&slowly return \\[1pt]
&into poems and theorems\\[6pt]
\multicolumn{2}{c}{\rm L.~Robbiano, 2026}\\
\end{tabular}
\end{center}
}

\maketitle

\makeatletter
\newcommand{\ignoresectionintoc}[1]{%
  \let\oldcontentsline\contentsline
  \renewcommand{\contentsline}[4]{%
    \ifnum\value{tocignore}<#1
      \addtocounter{tocignore}{1}%
    \else
      \oldcontentsline{##1}{##2}{##3}{##4}%
    \fi
  }
}
\newcounter{tocignore}
\makeatother

\ignoresectionintoc{2}

{
\setlength{\parskip}{0pt}
\tableofcontents
}
\newpage

\section*{Introduction}
\addcontentsline{toc}{section}{Introduction}
\label{sec-Introduction}

\begin{flushright}
\small\it 
I am not old,\\
in fact, I have a good memory,\\
and remember very well\\
when I was
\end{flushright} 

It was the end of February 2026, when Mehrdad Nasernejad, founder and Editor-in-Chief 
of the Galois Journal of Algebra, invited me to write a survey for that journal, and my first 
reaction was one of genuine surprise. Having been born in Genoa, 
Italy, in~1944\footnote{  \url{https://sites.google.com/view/cocoaschool2025}}, 
my very first step was to convince myself that I was not old.

After completing this work, 
reality suddenly brought me back to earth.  The task was to squeeze about twenty years of research 
into a not too large survey. Having written a few books myself, 
I soon realized that writing this article would be a journey punctuated by worries and insights.

\smallskip
\textbf{Remarks on Notation and Terminology}   

Should I follow the notation used in the cited papers and books?
The obvious answer is yes. However, this road is filled with obstacles.

\begin{myenumerate}

\item
For a marked polynomial, some papers (see, for instance,~\cite{KLR1}) 
use the convention of writing $(t,  f)$ instead of $(f, t)$. It was quite painful 
to make the decision;  in the end the \textit{winner} was $(f, t)$.

\item
The symbol $P$ for polynomial rings  is not commonly used. 
However, does it not seem natural to use $P$ 
for~\emph{$P\!$olynomials}?

\item For the transpose of a matrix $A$ what is a better choice, 
$A^{\rm tr}$ or $A^{\mathsf{T}}$.
After  \textit{considerable deliberation}, the choice was  $A^{\mathsf{T}}$.

\item
Throughout the paper,
we let~$\mathbb T^n=\{x_1^{\alpha_1}\cdots x_n^{\alpha_n}\mid \alpha_i\ge0 \}$
be the monoid of \textit{terms} in~$P$.
Even more contentious is the choice of the word \emph{term}. What is the difference between
a \emph{term} and a \emph{monomial}? In this case, the two opposing views seem to be evenly matched. For me, $5x^2y^3$ is a monomial, whereas $x^2y^3$ is a term, i.e., a monomial with coefficient~1. According to the opposing camp, it is exactly the other way around. In any case, 
sticking with my definition was not too difficult.

\item
And what about the name \emph{order ideal}? It was chosen because order ideals are the analogues of ideals in the theory of partially ordered sets. However, these objects appear in several branches of mathematics and, naturally, under different names.  A \textit{partial} list of such names is contained
in the introduction of~\cite{KR2}. For a while,  I was tempted to steal it and propose it here.
Although \textit{stealing from oneself} is not generally regarded as a serious sin, the final 
decision was not to do this.

\item What is a re-embedding? Re-embedding of what?
This survey adopts an algebraic/computational point of view.
Consequently, $(P,I)$ and $(Q,J)$ are explicitly given polynomial rings
over a field~$K$, together with their ideals $I$ and~$J$.
Suppose that an algebra homomorphism
$\phi \colon P \To Q$ satisfies $\phi(I)=J$
and induces an isomorphism
$\Phi \colon P/I \To Q/J$.
Strictly speaking, it is the homomorphism
$\phi \colon P \To Q$ which re-embeds the ideal~$I$
into the polynomial ring~$Q$.
However, since the induced isomorphism
$\Phi \colon P/I \To Q/J$ completely determines this re-embedding,
we shall also refer to~$\Phi$ as a
\textit{re-embedding of the ideal~$I$}.
After some deliberation, I decided to adopt this terminology
(see Definition~\ref{def-edim}).

\item Throughout the paper we let~$K$ be a field, let $P=K[x_1,\dots,x_n]$ be standard graded, 
and let~$\mathbb T^n=\{x_1^{\alpha_1}\cdots x_n^{\alpha_n} \mid \alpha_i\ge 0\}$ 
be the monoid of terms  in~$P$. At least this choice required no meditation.

\item 
From now until the end of the introduction, I have decided to use the 
abbreviation BBS for both ``border basis scheme'' and ``border basis schemes''.
\end{myenumerate}

\medskip
 \textbf{Remarks on the Content}.

Selecting the topics, arranging them into a coherent whole, and giving them a uniform
treatment turned into a sequence of worries and insights.
Perhaps I should not admit it, but in the end I am quite satisfied.

\begin{myenumerate}
\item   
For each statement (Theorem, Proposition, Corollary, or Remark), should I provide a proof, or simply refer the reader to the paper or book where it appears? Deciding how to strike the right balance was one of the most difficult aspects of writing this survey.
In most cases, I chose either to cite a book and present the proof in my own words, or, whenever this helped preserve the flow of the exposition, to refer directly to the original research papers.
Needless to say, none of these decisions was made without some hesitation.

\item
From the computational point of view, the algorithms and constructions presented in this work are defined over an arbitrary field $K$. In practical applications, $K$ is typically assumed to be a field supporting exact arithmetic, such as the field of rational numbers $\mathbb{Q}$ or a finite field $\mathbb{F}_p$.
To place the border basis schemes $\BO$ into a rigorous geometric framework, however, we frequently refer to notions such as smoothness and irreducibility. 
Throughout this work, these geometric properties are understood 
over the algebraic closure of~$K$, assuming that ${\rm char}(K)=0$.

This is the standard setting for the fundamental results on Hilbert schemes, including Fogarty's theorem~\cite{Fog} for surfaces and the more recent work of Jelisiejew et al.~\cite{DJNT} establishing the irreducibility of $\Hilb^{11}(\mathbb{A}^3)$.

\item  In a fortune cookie, I once found the following sentence:
\textit{The best gift you can bestow on others is a good example.}
Guided by this piece of practical wisdom from the very beginning of this project,
I was determined to include several examples illustrating the underlying theory.
One question, however, kept recurring: should some of them also be worked out explicitly
using a computer algebra system? In the end, I chose \cocoa\ (see~\cite{CoCoA}),
my favorite.
\end{myenumerate}

\bigskip\bigskip\bigskip
\centerline{\textbf{Content}}

\begin{center}
\begin{minipage}{0.65\textwidth} 
Having confessed my emotional journey through the challenge of writing this survey,
it is time to outline its actual content. As every mathematician knows (or should know),
every mathematical text should contain at least a proof and a joke, and they should not be the same.
There are many proofs in this paper, so...
\end{minipage}
\hskip .8cm
\begin{minipage}{0.23\textwidth} 
\small\it
\vskip 1cm
why work\\
with border bases?\\
because at last\\
the matrices commute
\end{minipage}
\end{center} 

\goodbreak
However, this is not the only reason why one should work with border
bases and BBS. In~\cite{Gro}, A.~Grothendieck introduced the Hilbert
scheme $\Hilb^\mu(\AA^n_K)$, which parametrizes all
zero-dimensional subschemes of length~$\mu$ of the affine
space~$\mathbb{A}^n_K$, as a quotient of a suitable Grassmannian
variety. If this construction is translated into a computable
presentation, its defining relations give rise to a rather large and
cumbersome system of equations. There is a close relationship between
Hilbert schemes and BBS, each providing a different perspective on the
other.
Indeed, BBS corresponding to order ideals of length~$\mu$ form an
open covering of the Hilbert scheme
$\Hilb^\mu(\mathbb{A}^n_K)$, and they admit simple presentations by
means of easily describable quadratic equations
(see~\cite{MS}, Section~18.4). An explicit description of how to glue
these open sets is given in~\cite[Section~2]{KR4}.

To start reading this survey, the readers are expected to have some basic knowledge
of commutative algebra and computer algebra.
For instance, they should be familiar with \textbf{Gr\"obner bases}, for which we
refer to~\cite{KR1, KR2}. Other excellent sources are~\cite{AL, BWK, Bu, CLOS, Froeb}.
Moreover, the readers should be familiar with the notion of an
\textbf{affine $K$-algebra}, i.e., an algebra of the form $P/I$, where
$P=K[x_1,\dots,x_n]$ is a polynomial ring over a field~$K$ and $I$ is a proper ideal of~$P$.

Clearly, to understand BBS, we first need to understand border bases.
They originated in the final years of the last millennium with the
seminal works~\cite{MMM, MS, Mou, Ste}. These were followed by numerous further
contributions, including~\cite{KK1, KK2, KKR, Sau}.

\textbf{Section~\ref{sec-Border Bases} (Border Bases)} 
provides an introduction to border bases.
The first two subsections contain the necessary background on border bases and
their relationship with Gr\"obner bases. Let~$K$ be a field, let
$P=K[x_1,\dots,x_n]$, and let
$I\subseteq P$ be a zero-dimensional ideal.
The basic idea of border basis theory is to describe the
zero-dimensional algebra $P/I$
using an order ideal of terms~$\OO$ (Definition~\ref{def-border!closure}) 
whose residue classes
form a $K$-basis of~$P/I$, together with the multiplication tables
associated with this basis.
A number of statements are presented without proof for two reasons.
First, this section is intended as an introduction to the main ideas of
border basis schemes. Second, most of the proofs can be found in
\cite[Section~6.4]{KR2}. There is, however, one exception, namely
\textit{Subsection~\ref{subsec-Commuting Matrices}}, where the connection
between border bases and commuting multiplication matrices is explained in
detail, as it is the key ingredient in the construction of BBS.

We also mention that several generalizations of the notion of a border basis
have recently been proposed; see, for instance,~\cite{BeBo, RoSa, SSFN}.

\medskip
\textbf{Section~\ref{sec-Border Bases Schemes} (Border Basis Schemes)} is meant to 
introduce the reader to the notion of a~BBS. The theory of these schemes has been developed over the
last twenty years, notably in~\cite{Hai}, \cite{Hui1}, \cite{MS}, and
\cite{Rob}, followed by many other papers that will be mentioned and used in
this survey.

\begin{myenumerate}

\item[] \textit{Subsection~\ref{subsec-IntroToBBS}} starts with
Definition~\ref{def-BBS}, which sets the stage by introducing the polynomial
ring $K[C]$. It contains the ideal $I(\BO)$ defined by the commutator matrices
$\mathcal{M}_i \mathcal{M}_j-\mathcal{M}_j\mathcal{M}_i$. In this way, the
\textbf{border basis scheme} $\BO$ takes center stage together with its
coordinate ring $K[C]/I(\BO)$ and the \textbf{universal $\OO$-border basis family}
$U_\OO$. Then Definition~\ref{def-represent} shows how the closed points of
$\BO$ represent zero-dimensional affine algebras whose coordinate rings have
the residue classes of the elements of $\OO$ as a $K$-basis.
The final part of the subsection is devoted to an alternative description of
BBS, namely through the use of \textbf{Neighbor Generators}
(Definition~\ref{def-neighbors}).

\smallskip
\newpage
\item[] \textit{Subsection~\ref{subsec-First Properties of BBS}} is devoted to
the fundamental properties of BBS. In particular, Proposition~\ref{prop-principal}
shows the connection between Hilbert schemes and BBS. Then,
Proposition~\ref{prop-irreducible} describes the extent to which BBS are
irreducible and recalls that planar BBS are smooth and irreducible.
The notion of an \textbf{affine cell} is recalled (Definition~\ref{def-cell}).

\smallskip

\item[] \textit{Subsection~\ref{subsec-Arrow Gradings}} introduces an important 
grading on the polynomial ring $K[C]$, namely the \textbf{total arrow grading},
and Proposition~\ref{prop-arrowhomog} shows that $I(\BO)$ is homogeneous with
respect to this grading, denoted by $\deg_W$.
\end{myenumerate}

\medskip
\textbf{Section~\ref{sec-Re-embedding} (Re-embedding)} marks a shift in perspective.
We have seen that BBS are described by nice equations which are suitable for computer
computations. However, there is a drawback: the number of indeterminates tends to be huge,
which in most cases makes the actual computations infeasible. This naturally raises the
question of whether it is possible to re-embed the ideal $I(\BO)$ into a polynomial ring
with fewer indeterminates. Since this is an interesting question in its own right, we
temporarily leave BBS aside and concentrate on general affine algebras $P/I$.

\begin{myenumerate}
\item[] \textit{Subsection~\ref{subsec-Z-Separating Tuples and Term Orderings}} starts
with Definition~\ref{def-sepindets}, which introduces the notion of $Z$-separating and
\textbf{coherently $Z$-separating} tuples of polynomials in $I$. A first important
consequence is given by Proposition~\ref{prop-ElimViaSubst}, which shows how coherently
$Z$-separating tuples can be used to compute elimination simply by substitution.

\smallskip

\item[] \textit{Subsection~\ref{subsec-Cotangent Spaces and GFans}} recalls the notion
of \textbf{cotangent space}, and Proposition~\ref{prop-cotancomp} shows how to compute
it, since the cotangent space coincides with $\M_1/\Lin(I)$, and $\Lin(I)$ can be
computed using any set of generators of $I$. Then Definition~\ref{def-coeffmat}
introduces the \textbf{coefficient matrix} $\cmat(G)$ and its submatrix
$\cmat_Z(G)$. These matrices can be fruitfully used to prove in
Proposition~\ref{prop-sepZ} that, if $I$ is $Z$-separating, then~$Z$ is contained in
$\LT_\sigma(\Lin(I))$ for every elimination term ordering for $Z$. This observation
connects the problem of finding suitable tuples $Z$ with the computation of suitable
\textbf{Gr\"obner fans}.
Definition~\ref{def-GFan} provides this connection by introducing the notion of a
marked Gr\"obner basis
$\overline{G}=((g_1,\LT_\sigma(g_1)),\dots,(g_r,\LT_\sigma(g_r)))$,
where $G=(g_1,\dots, g_r)$ is the reduced $\sigma$-Gr\"obner basis of $I$ for a term
ordering $\sigma$. The set
$\{z\in X\mid \LT_\sigma(g_i)=z \text{ for some } i\in\{1,\dots,r\}\}$, denoted
by $\LI(\overline{G})$,
is called the \textbf{set of leading indeterminates} of $\overline{G}$. 
Remark~\ref{rem-GFanLin} explains how computing the Gr\"obner fan 
of $\text{Lin}(I)$ is a useful way to obtain suitable $Z$  for detecting 
coherently $Z$-separating tuples of polynomials in the ideal~$I$.

\smallskip
\item[] \textit{Subsection~\ref{subsec-Z-Separating Re-embeddings and 
Optimal Re-embeddings}} reaps the benefits of the preceding subsections. First,
Proposition~\ref{prop-isoZseparating} clearly explains that a coherently
$Z$-separating tuple of polynomials in~$I$ allows one to re-embed $I$ into a
polynomial ring with fewer indeterminates.

Then Definition~\ref{def-edim} introduces a precise notion of re-embedding of $I$,
defines the notion of an \textbf{optimal embedding} and the corresponding
\textbf{embedding dimension} (edim), together with the \textbf{separating embedding dimension} and the notions of \textbf{best and optimal separating embedding} (sepdim). 
It also introduces the notion of a \textbf{free $K$-algebra} and of 
a \textbf{separating free $K$-algebra}.
Finally, Theorem~\ref{thm-ineqsep} gives a nice condition under which
$\edim=\sepdim$, and Theorem~\ref{thm-checkopt} gives a condition under which a
$Z$-separating re-embedding is optimal, together with a stronger condition under which
$P/I$ is a free $K$-algebra.
\end{myenumerate}

\newpage
\textbf{Section~\ref{sec-Back to BBS} (Back to Border Basis Schemes)} 
returns to BBS and shows how the results
of the preceding section can be put to good use. There is good news! In the case
of BBS, the linear parts of the polynomials in~$I(\BO)$ are very easy to
describe: they are either a single indeterminate or the difference of two
indeterminates in $K[C]$. This leads to a classification of the indeterminates
in~$K[C]$.

\begin{myenumerate}
\item[] \textit{Subsection~\ref{subsec-cotequiv}} introduces the above
classification using the notion of \textbf{cotangent equivalence}, which
partitions the indeterminates of $K[C]$ into \textbf{trivial},
\textbf{basic}, and \textbf{proper} ones (Definition~\ref{def-cotequiv}). Recall from
Remarks~\ref{rem-GFanLin} and~\ref{rem-LinLT} that we are interested in sets
$Z$ of indeterminates contained in
$\LT_\sigma(\Lin(I(\BO)))$, where $\sigma$ ranges over all term orderings.
These sets are described explicitly in Theorem~\ref{thm-shapeofSsigma}. 

\smallskip
\item[] \textit{Subsection~\ref{subsec-IntRimExp}} classifies the
indeterminates of $K[C]$ as \textbf{rim}, \textbf{interior}, and
\textbf{exposed} (Definitions~\ref{def-rim-and-interior} and~\ref{def-exposedIndets}).
Then Proposition~\ref{prop-CharExposed} provides an important characterization
of $\Cexp$, the set of exposed indeterminates, which is used in the next
subsection.

\smallskip

\item[] \textit{Subsection~\ref{subsec-PlanarBoxBBS}} starts with
Proposition~\ref{prop-elimNonExp}, which shows that if
$Z=C\setminus\Cexp$, then the ideal $I(\BO)$ of a planar BBS is
coherently $Z$-separating. The subsection ends with
Proposition~\ref{prop-PlanarBox}, which proves that, for planar Box BBS, the
$Z$-separating re-embedding defined by $Z=C\setminus\Cexp$ is an isomorphism
between the coordinate ring of $\BO$ and $K[\Cexp]$. It follows that
planar Box BBS are affine cells.
\end{myenumerate}

\medskip
\textbf{Section~\ref{sec-Positive P0-algebras} (Positive $P_0$-algebras)} is devoted to the study of a
special class of affine algebras, called positive $P_0$-algebras,
which will later be used to study an important family of BBS.

\begin{myenumerate}

\item[] \textit{Subsection~\ref{subsec-Best and Optimal Re-embeddings
of Positive P0-algebras}} starts with the definition of a positive
$P_0$-algebra (Definition~\ref{def-posalg}), followed by
Proposition~\ref{prop-optimalgraded}, which shows that if $P/I$ is a
\textbf{positively graded} algebra and
$d=\dim_K(\Lin(I))$, then there exist a $d$-tuple of indeterminates and
a $d$-tuple of polynomials in $I$ which define an optimal separating
re-embedding of $I$.

However, a positive $P_0$-algebra, when viewed as a $K$-algebra is, in general, only
\textbf{non-negatively graded}. In this setting, it is nevertheless
possible to define the \textbf{$P_0$-linear part} of a polynomial and of
an ideal (Definition~\ref{def-linZ}). Then
Proposition~\ref{prop-indepOfGen} shows that
$\Lin_{P_0}(I)$ can also be computed from any set of homogeneous
generators of $I$. 
The subsection ends with Theorem~\ref{thm-existZsep}, which describes
a condition under which a $W$-homogeneous $Z$-separating tuple exists.

\smallskip
\item[] \textit{Subsection~\ref{subsec-UMP}} introduces a tool which
will play an important role in what follows, namely the
\textbf{Unimodular Matrix Problem (UMP)} (Definition~\ref{def-UMP}). 
Under suitable assumptions,
Theorem~\ref{thm-GoodIso} shows how UMP can be used to construct a
$P_0$-graded automorphism $\phi$ of $P$ such that
$\langle\phi(G)\rangle$ is $X_k$-separating, where
$G=(g_1,\dots, g_k)$ is a tuple of $P_0$-homogeneous polynomials where
is  less than or equal to the number of indeterminates of positive
degree in $P$.

\smallskip

\item[] \textit{Subsection~\ref{subsec-Fibers}} introduces the notion of
the \textbf{fibers} of a positive $P_0$-algebra
(Definition~\ref{def-fibers}). A key observation is that, if $P/I$ is a
positive $P_0$-algebra, there is a canonical injective homomorphism
$\phi\colon P_0\To P/I$. Since $P_0$ is a polynomial ring, one can
consider the generic and the special fibers of $\phi$. They are
positively graded algebras and therefore admit optimal separating
re-embeddings, as shown in Proposition~\ref{prop-optFibers}. The
subsection ends with two important theorems. First,
Theorem~\ref{thm-fibers} shows that, if certain local rings are regular,
then the fibers are free algebras. Second,
Theorem~\ref{thm-connected} proves that the spectrum of a positive
$P_0$-algebra is \textbf{connected} by exhibiting an explicit way to join any
pair of points.

\smallskip

\item[] \textit{Subsection~\ref{subsec-Freeness of Regular Positive
P0-Algebras}} culminates in one of the main results of this section.
Using some preliminary results
(Proposition~\ref{prop-standard} and
Lemma~\ref{lem-unitJ}),
Theorem~\ref{thm-Free} proves that every \textbf{regular positive
$P_0$-algebra} $P/I$ is a free algebra.

\end{myenumerate}

\medskip
\textbf{Section~\ref{sec-SpecialBBS} (Special Border Basis Schemes)} 
studies several classes of BBS.

\begin{myenumerate}
\item[] \textit{Subsection~\ref{subsec-HomogeneousBBS}} introduces
\textbf{homogeneous $\OO$-border basis schemes} $\BOhom$
(Definition~\ref{def-homgBBS}) and
\textbf{MaxDeg border basis schemes}
(Definition~\ref{def-maxdegBBS}), followed by an important result:
Theorem~\ref{thm-homcommute} shows that, if $\OO$ is MaxDeg, the
generic homogeneous multiplication matrices are \textbf{pairwise commuting},
hence $\BOhom$ is an affine space, as shown in Proposition~\ref{prop-homcommute}.
Moreover, this proposition proves that, under the same
assumption, $I(\BO)\cap K[\Cnull]=\langle0\rangle$.
These two facts imply that the coordinate ring of a MaxDeg BBS is a \textbf{positive
$K[\Cnull]$-algebra}. Consequently, if~$K$ is a perfect field, every planar
MaxDeg BBS defined over $K$ is an affine cell of dimension~$2\mu$
(Theorem~\ref{thm-Free}).

The subsection ends with Example~\ref{ex-Lshape}, in which a specific planar
MaxDeg BBS is examined in detail. In agreement with
Theorem~\ref{thm-Free}, it turns out to be an affine cell. However, the
expected isomorphism between its coordinate ring and a polynomial ring is far
from being trivial, since the images of some indeterminates in $\QQ[C]$ are
polynomials with very large support (see Example~\ref{ex-Lshape}).

\smallskip

\item[] \textit{Subsection~\ref{subsec-Simplicial BBS}}\,  is devoted to
\textbf{simplicial BBS}. They are introduced in
Definition~\ref{def-simplicial}, and
Proposition~\ref{prop-postotgrad} characterizes them as the BBS whose
coordinate rings are positively graded with respect to the total arrow
grading. Several properties are established in
Lemma~\ref{lem-formulas},
Proposition~\ref{prop-simplicialOptimal}, and
Proposition~\ref{prop-optimalCint}. In particular, the positive grading
implies that simplicial BBS admit optimal separating re-embeddings.
The subsection concludes with
Proposition~\ref{prop-AffineOrSing}, which shows that planar simplicial BBS
are affine cells, whereas non-planar simplicial BBS have singularities.
\end{myenumerate}

\textbf{Section~\ref{sec-SubBBS} (Subschemes of the Border Basis Schemes)} 
studies subschemes of BBS.

\begin{myenumerate}
\item[] \textit{Subsection~\ref{subsec-BOdf}} presents an important
subscheme of a border basis scheme, namely the
\textbf{degree-filtered border basis scheme} $\BOdf$.
After introducing the notion of a \textbf{degree filtration}
(Definition~\ref{def-degfiltration}) and that of a
\textbf{degree-filtered $\OO$-border basis}
(Definition~\ref{def-degfiltrbasis}),
Definition~\ref{def-degfiltrScheme} formally constructs the scheme $\BOdf$.
It is determined by the ideal
$I(\BO)+I_\OO^{\mathrm{df}}$, where $I_\OO^{\mathrm{df}}$ is generated by the
indeterminates $c_{ij}$ such that
$\deg(t_i)>\deg(b_j)$. Consequently, the coordinate ring of $\BOdf$ is
non-negatively graded, although in general it is not a positive
$P_0$-algebra. Finally,
Proposition~\ref{prop-CharMaxdeg} shows that
$\BO=\BOdf$ if and only if $\BO$ is MaxDeg.

\smallskip

\item[] \textit{Subsection~\ref{subsec-LocGor}} deals with the computation
of special loci inside a BBS, with particular emphasis on the
\textbf{locally Gorenstein locus}. One of the main ingredients is the
\textbf{Cyclicity Test}, explained in the book~\cite{KR25}. Its relevance in
the present context is established in
Theorem~\ref{thm-CharGor}, which shows that a zero-dimensional affine
$K$-algebra $R$ is locally Gorenstein if and only if its \textbf{canonical module} is
cyclic. As a consequence, deciding whether $R$ is locally Gorenstein reduces
to computing the determinant of a suitable matrix
(Corollary~\ref{cor-cycBB}). The main application of this criterion is given
in Theorem~\ref{thm-LocGorLocus}, which provides an explicit description of
the locally Gorenstein locus inside a given BBS. The subsection concludes
with guidance on how to compute other loci.
\end{myenumerate}

\newpage
\textbf{Section~\ref{sec-Questions and Problems} (Questions and Problems)} presents
\textbf{questions and open problems} related to the topics covered in this survey.

\medskip
\textbf{Section~\ref{sec-CoCoA Examples} (\cocoa-Examples)} is the final section. It collects
several examples used throughout the survey, presented as
\cocoa-examples, that is, in the form of input/output produced by the
computer algebra system \cocoa\ (see~\cite{CoCoA}).

\bigskip
Now it is time to express deep gratitude to my collaborators.
First of all, I wish to acknowledge that without my longstanding collaboration 
with Martin Kreuzer, this survey would never have come into being.
Indeed, most of its content is based on a series of joint papers and books
written with Martin and, more recently, with L\^e Ng\d{o}c Long, which form
the backbone of this work and are cited throughout the text.

\smallskip 
It is customary to include formal thanks to the anonymous reviewers. 
In this case, however, I am especially pleased  to  express my sincere appreciation for 
the extremely careful reading of the paper and the truly constructive feedback they provided.

\bigskip\bigskip
Let me conclude this introduction with a question that is always lurking in
the background. Can BBS be used to produce models for \textit{real-world
problems}? The chances may be small, as the following short poem suggests,
but\ldots who knows?

\begin{center}
\small\it
\begin{tabular}{r @{\hskip 1cm} l}
{\bf modelli} & {\bf models} \\[4pt]
perpendicolare sublime & sublime perpendicular \\[1pt]
trafigge volute di nebbia, & pierces swirls of mist, \\[1pt]
rivela teoremi sepolti, & revealing buried theorems, \\[1pt]
arcane equazioni perdute, & lost arcane equations, \\[1pt]
ignoti, perfetti modelli & unknown, perfect models \\[6pt]
\multicolumn{2}{c}{\rm L.~Robbiano, 2023}\\[4pt]
\multicolumn{2}{c}{\rm English version, 2026}
\end{tabular}
\end{center}

\newpage   
\section{Border Bases}
\label{sec-Border Bases}

\begin{flushright}
\small\it 
on one side the mountains,\\
on the other side the sea,\\
on the \textbf{border} the city
\end{flushright} 

\index{border!basis}%
%
%

The interest in this subject is due to several reasons.

\begin{enumerate}[label=\textbf{\arabic*.}, ref=\arabic*]

\item \label{reason1} \textbf{Border bases generalize Gr\"obner bases}. 
If $\OO$ is taken to be the complement of the leading term ideal of~$I$
with respect to a term ordering~$\sigma$, the corresponding border basis
contains the reduced $\sigma$-Gr\"obner basis of~$I$ (Proposition~\ref{prop-bordandGr}).

\item \label{reason2} \textbf{Border bases are, in general, more versatile 
than reduced Gr\"obner bases}. 
For instance, if an ideal~$I$ is invariant under the action of a group of
symmetries, it is often possible to find a border basis that preserves these
symmetries. In contrast, a Gr\"obner basis is constrained by a term 
ordering, which frequently breaks the natural symmetry of the problem 
(Examples~\ref{ex-bb-no-gb} and~\ref{ex-bb-no-gb continued}).

\item \label{reason3} \textbf{Border bases are better suited for ``real world'' computations}.
They are more stable with respect to small variations in the coefficients 
of the polynomials generating~$I$. This property facilitates symbolic-numeric 
computations with polynomial systems involving approximate coefficients 
(see, for instance,~\cite{AFT, HKPP, Sau}).
\end{enumerate}

\medskip
\subsection{Introduction to Border Bases}
\label{subsec-IntroBB}


We begin by introducing some fundamental objects.
Given an order ideal, we define a sequence of associated  order ideals as follows.

\begin{definition}\label{def-border!closure}
Let $\TT^n$ be  the monoid of terms in $n$ indeterminates.

\begin{myenumerate}
\item  A set of terms $\OO=\{t_1,\dots,t_\mu\}$
is called an {\bf order ideal} in~$\mathbb{T}^n$ if
$t\in\OO$ implies that every term $t'\in\mathbb{T}^n$
which divides~$t$ is also contained in~$\OO$. 

\item The \textbf{border} of~$\OO$ is the set
$\partial\OO=\TT_1^n\cdot \OO\setminus \OO
=(x_1\OO\cup\dots\cup x_n\OO)\setminus\OO$.
The \textbf{first border closure} of~$\OO$ is the set
$\overline{\partial \OO}=\OO\cup\partial \OO$.
\index{border!of an order ideal}%

\item For every $k\ge 1$, we inductively define
the \textbf{$(k+1)^{\rm st}$ border} of~$\OO$ by
$\partial^{k+1}\OO=\partial(\overline{\partial^{k}\OO})$
and the {\bf $(k+1)^{\rm st}$ border closure} of~$\OO$ by the rule
$\overline{\partial^{k+1}\OO}=\overline{\partial^{k}\OO}\cup
\partial^{k+1}\OO$.
For convenience, we let $\partial^0\OO
=\overline{\partial^0\OO}=\OO$.
\index{border!closure}%
\index{higher borders}%

\end{myenumerate}
\end{definition}

Notice that the $k^{\rm th}$ border closure of an order ideal
is again an order ideal for every $k\ge 0$.

\begin{example}\label{ex-borders}
Let $\OO=\{1, x, y, x^2, xy, y^2, x^3, x^2y, y^3, x^4, x^3y\}\subset\TT^2$. Then 
$\OO$ is an order ideal. 
The order ideal~$\OO$
together with its first two borders is visualized below.

\medskip
\noindent
\hspace{1.5cm} 
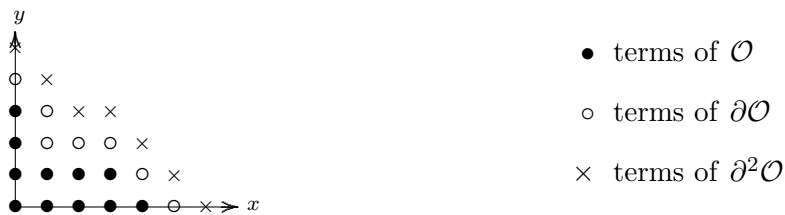
\begin{minipage}[c]{0.40\textwidth} 
\leavevmode
\beginpicture
\setcoordinatesystem units <0.42cm,0.42cm> 
\setplotarea x from 0 to 7, y from 0 to 5.5
\axis left /
\axis bottom /

\arrow <2mm> [.2,.67] from  6.5 0  to 7 0
\arrow <2mm> [.2,.67] from  0 5  to 0 5.5

\put {$\scriptstyle x$} [lt] <0.5mm,0.8mm> at 7.1 0
\put {$\scriptstyle y$} [rb] <1.7mm,0.7mm> at 0 5.6

\multiput {$\bullet$} at 0 0  1 0  2 0  3 0  4 0  0 1  1 1  2 1  3 1  0 2  0 3 /

\multiput {$\circ$} at 5 0  4 1  3 2  2 2  1 2  1 3  0 4 /

\multiput {$\scriptstyle\times$} at 6 0  5 1  4 2  3 3  2 3  1 4  0 5 /
\endpicture
\end{minipage}
\hfill
\begin{minipage}[c]{0.45\textwidth} 
\small
\begin{itemize}
    \item[$\bullet$] terms of $\OO$
    \item[$\circ$] terms of $\partial\OO$
    \item[$\times$] terms of $\partial^2\OO$
\end{itemize}
\end{minipage}
\hspace{0.5cm} 
\end{example}

The following lemma gathers some properties of borders and border closures.

\begin{lemma}\label{lem-BorderProps}
Let $\OO\subseteq \TT^n$ be an order ideal.

\begin{myenumerate}
\item For every $k\ge 0$, we have a disjoint union
$\overline{\partial^k\OO}=\bigsqcup_{i=0}^k\partial^i\OO$.
Consequently, we have a disjoint union $\TT^n=\bigsqcup_{i=0}^\infty
\partial^i\OO$.

\item For every $k\ge 1$, we have
$\partial^k\OO=\TT_k^n\cdot \OO\setminus 
\TT_{<k}^n\cdot \OO$.

\item A term $t\in\TT^n$ is divisible by a term in~$\partial\OO$
if and only if $t\in\TT^n\setminus \OO$.

\end{myenumerate}
\end{lemma}

\begin{proof}
See~\cite[Proposition~6.4.6.]{KR2}
\end{proof}

The disjoint decomposition in part~(1) of the previous lemma allows us
to measure the ``distance'' of a term from an order ideal as follows.

\begin{definition}\label{def-index}
Let $\OO\subseteq \TT^n$ be an order ideal.

\begin{myenumerate}
\item For every $t\in\TT^n\!$, the unique number 
$k\in\mathbb{N}$ such that $t\in\partial^k\OO$ is 
called the \textbf{index} of~$t$ with respect to~$\OO$ 
and is denoted by~$\indO(t)$.
\index{index!of a term}%

\item For a polynomial $f\in P\setminus \{0\}$, we define
the \textbf{index} of~$f$ with respect to~$\OO$ (or the
\textbf{$\OO$-index} of~$f$) by $\indO(f)=
\max\{\indO(t) \mid t\in\Supp(f)\}$.\index{index!of a polynomial}%

\end{myenumerate}
\end{definition}

We now collect  some useful properties of the index.

\begin{lemma}\label{lem-indprop}
Let $\OO\subseteq\TT^n$ be an order ideal.

\begin{myenumerate}
\item For a term $t\in\TT^n$, the number $k=\indO(t)$ is
the smallest natural number such that $t=t' t''$
where $t'\in \TT^n_k$ and $t''\in\OO$.

\item Given two terms $t, t'\in \TT^n$, we have $\indO(t\,t')
\le \deg(t)+\indO(t')$.

\item For non-zero polynomials $f, g\in P$ 
such that $f+g \ne 0$, we have the inequality
$\indO(f+g)\le \max\{\indO(f),\indO(g)\}$.

\item For non-zero polynomials $f, g\in P$, we have 
the inequality 
\begin{equation*}
\indO(f\,g)\le \min\{\deg(f)+\indO(g),\,
\deg(g)+\indO(f)\}
\end{equation*}
\end{myenumerate}
\end{lemma}

\begin{proof}
See \cite[Proposition~6.4.8.]{KR2}
\end{proof}

\begin{myremark}\label{rem-incompatibe}
Using the $\OO$-index to order terms has one drawback:
this ordering is incompatible with multiplication, i.e.\
the inequality $\indO(t)\le \indO(t')$ does not, in general, imply
the inequality $\indO(t\,t'')\le \indO(t'\,t'')$, as shown in the next example.
\end{myremark}

\begin{example}\label{ex-minipageindex}
Let $\OO=\{1, x, x^2\}\subseteq\TT^2$. 
Then~$\OO$ is an order ideal, and its border is
$\partial\OO=\{y, xy, x^2y, x^3\}$.
The following diagram illustrates the situation.

\smallskip
\noindent
\hspace*{0.5cm} 
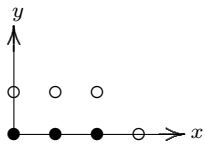
\begin{minipage}[c]{0.25\textwidth} 
\leavevmode 
\beginpicture
\setcoordinatesystem units <0.55cm,0.55cm> 
\setplotarea x from 1 to 5, y from 0 to 2.5
\axis left /
\axis bottom /

\arrow <2.5mm> [.2,.67] from  4.5 0  to 5.1 0 
\arrow <2.5mm> [.2,.67] from  1 2  to 1 2.6

\put {$\scriptstyle x$} [lt] <0.5mm,0.8mm> at 5.1 0
\put {$\scriptstyle y$} [rb] <1.7mm,0.7mm> at 1 2.6
\put {$\bullet$} at 1 0
\put {$\bullet$} at 2 0
\put {$\bullet$} at 3 0
\put {$\circ$} at 1 1
\put {$\circ$} at 2 1
\put {$\circ$} at 3 1
\put {$\circ$} at 4 0
\endpicture
\end{minipage}
\hfill 
\begin{minipage}[c]{0.6\textwidth} 
\small
We have $0=\indO(x^2)<\indO(y)=1$. Multiplying
both sides of the inequality by~$x^2$, 
we get $2=\indO(x^2\cdot x^2)>\indO(x^2\cdot y)=1$.
Similarly, we have $1=\indO(y)=\indO(x^2y)$. Multiplying both sides
of the equality by~$x$, we get 
$1 =\indO(x\cdot y)<\indO(x\cdot x^2y)=2$.
\end{minipage}
\hspace*{0.5cm} 
\end{example}

\medskip
Recalling  the notation  $\OO=\{t_1,\dots,t_\mu\}$ and  $\partial\OO=\{b_1,\dots,b_\nu\}$,
we are ready to introduce a key definition. 

\begin{definition}\label{def-border!prebasis}
A set of polynomials $G=\{g_1,\dots,g_\nu\}$ is called an 
\textbf{$\OO$-border prebasis} if the polynomials have the form
$g_j=b_j-\sum_{i=1}^\mu\,\alpha_{ij}t_i$ with $\alpha_{ij}\in K$
for $1\le i\le\mu$ and $1\le j\le\nu$.
\index{border!prebasis}%
\end{definition}
 
Border prebases are already sufficient to perform polynomial
division with remainder. The following algorithm provides 
a fundamental tool in working with border prebases.

\begin{proposition}\textbf{\textup{(The Border Division Algorithm)}}\\
\label{prop-BorderDiv}%
Let $\OO=\{t_1,\dots,t_\mu\}$ be an order ideal in~$\TT^n$, 
let $\partial\OO= \{b_1,\dots,b_\nu\}$ be its border, 
and let $\{g_1,\dots,g_\nu\}$
be an $\OO$-border prebasis.
Given a non-zero polynomial $f\in P$, consider the following 
sequence of instructions.

\begin{myenumerate}
\item Let $f_1=\cdots=f_\nu=0$, $c_1=\cdots =c_\mu=0$,
and $h=f$.

\item If $h=0$, return $(f_1,\dots, f_\nu, c_1,\dots, c_\mu)$
and stop.

\item If $\indO(h)=0$, then write $h=c_1t_1+\cdots+c_\mu t_\mu$
with $c_1,\dots,c_\mu\in K$.\\
Return $(f_1,\dots, f_\nu, c_1,\dots,c_\mu)$ and stop.

\item If $\indO(h)>0$, then let $h=a_1 h_1 + \cdots +a_s h_s$
with $a_1,\dots,a_s\in {K\setminus \{0\}}$ and $h_1,\dots,h_s
\in\TT^n$ such that we have $\indO(h_1)=\indO(h)$. Compute the smallest
index~$i\in\{1,\dots,\nu\}$ for which~$h_1$ factors as
$h_1=t'\,b_i$ with a term $t'\in\TT^n$ of degree $\indO(h)-1$.
Subtract $a_1 t' g_i$ from~$h$, add~$a_1 t'$ to~$f_i$,
and continue with step~2).
\end{myenumerate}

\vskip-.1cm
\noindent This algorithm  returns a tuple
$(f_1,\dots, f_\nu, c_1,\dots,c_\mu)\in P^\nu \times K^\mu$ such that
$$
f= f_1 g_1 + \cdots + f_\nu g_\nu + c_1 t_1 +\cdots + c_\mu t_\mu
$$
and $\deg(f_i)\le \indO(f)-1$ for all~$i\in\{1,\dots,\nu\}$
with $f_i g_i\ne 0$. This representation does not depend on the
choice of the term~$h_1$ in step~4).
\end{proposition}

\begin{proof} See~\cite[Proposition~6.4.11]{KR2}.
\end{proof}

The above algorithm suggests the following definition.

\begin{definition}
Let $\OO=\{t_1,\dots,t_\mu\}$ be an order ideal in~$\TT^n$, 
let $\partial\OO= \{b_1,\dots,b_\nu\}$ be its border, 
let $\{g_1,\dots,g_\nu\}$
be an $\OO$-border prebasis, and let $f \in P$ be 
 a non-zero polynomial. Then  the polynomial $c_1 t_1 +\cdots + c_\mu t_\mu$ 
 is denoted by $\NR_{\OO,G}(f)$ and called the 
 \textbf{normal $\OO$-remainder} of~$f$.
 \end{definition}

However, as it happens with the normal remainder in Gr\"obner basis theory,
remainders may depend on the order of the polynomials in the $\OO$-border 
prebasis. So, it is natural to introduce the following notion.

\goodbreak
\begin{definition}\label{defofbb}
Let $G=\{g_1,\dots,g_\nu\}$ be an $\OO$-border prebasis,
let $G$ be the tuple~$(g_1,\dots,g_\nu)$, and let
$I\subseteq P$ be an ideal containing~$G$.
The set (or the tuple)~$G$  is called an
\textbf{$\OO$-border basis} of~$I$ if one of the following
equivalent conditions is satisfied.
\index{border!basis}%
\index{border basis}%

\begin{myenumerate}
\item The residue classes $\overline{\OO}=
\{\bar t_1,\dots,\bar t_\mu\}$ form a $K$-vector space
basis of~$P/I$.

\item We have $I\cap \langle \OO\rangle_K =\{0\}$.

\item We have $P=I\oplus \langle \OO\rangle_K$.
\end{myenumerate}
\end{definition}

The following example illustrates the definition and 
will serve as a \textit{running example} throughout this section
(see also \cocoa-Example~\ref{coex-bb-no-gb}).

 \begin{example}\label{ex-bb-no-gb}
Let $P=\mathbb{Q}[x, y]$ and $I = (xy -y^2 -x^2 ,\, x^3,\, x^2y,\, xy^2,\, y^3)$.
The set $\OO = \{1, y, x, y^2, x^2\}$ is an order ideal
consisting of $\dim_K(P/I)=5$ terms. The border is 
$\partial\OO = \{xy, y^3, xy^2, x^2y, x^3\}$. Let $G = \{g_1, \dots, g_5\}$ be the 
$\OO$-border prebasis where $g_1 = xy -y^2 -x^2$, $g_2 = y^3$, $g_3 = xy^2$, 
$g_4 = x^2y$, and $g_5 = x^3$. 
Since $I = \langle G \rangle$ is a homogeneous ideal and we have 
$I_{\le 2} = \langle xy -y^2 -x^2  \rangle_K$ (see \cocoa-Example~\ref{coex-bb-no-gb}), 
the condition $I \cap \langle \OO \rangle_K = \{0\}$ is satisfied. 
This implies that $G$ is indeed an $\OO$-border basis of the ideal~$I$.
\end{example}
 
Next, we see that by slightly changing Example~\ref{ex-bb-no-gb},
we obtain a fundamentally different outcome.

\begin{example}\label{ex-nobb}
Let $P$, $\OO$, and $\partial\OO$ be as in Example~\ref{ex-bb-no-gb},
and let $G = \{g_1, \dots, g_5\}$ where
$g_1 = xy - x^2 - y^2 - x$, $g_2 = y^3$, $g_3 = xy^2$, $g_4 = x^2y$, and $g_5 = x^3$.
Then $G$ forms an $\OO$-border prebasis of~$I = \langle G \rangle$. 
For $\sigma = \text{\texttt{DegRevLex}}$, the reduced $\sigma$-Gr\"obner basis of~$I$ is 
$(x^2,\, xy,\, y^2 - x)$.
Consequently, we have $x^2 \in I$, which implies $I \cap \langle\OO\rangle_{\mathbb{Q}} \neq \{0\}$ 
since $x^2 \in \OO$.
According to Definition~\ref{defofbb}, the set $G$ is \textbf{not} an $\OO$-border basis of $I$.
\end{example}

Now we show that the definition ensures that 
an \hbox{$\OO$-border} basis of~$I$ actually generates~$I$.

\begin{proposition} \label{prop-bbgen}
\index{border basis!generates the ideal}
Let $G$ be an $\OO$-border basis of an ideal $I \subseteq P$. 
Then the ideal $I$ is generated by $G$.
\end{proposition}

\begin{proof} 
By definition, we have $\langle g_1, \dots, g_\nu \rangle \subseteq I$.
To prove the converse inclusion, let $f \in I$. Using the Border Division Algorithm 
(Proposition~\ref{prop-BorderDiv}), the polynomial $f$ can be expanded as:
$ f = f_1 g_1 + \dots + f_\nu g_\nu + c_1 t_1 + \dots + c_\mu t_\mu$
where $f_1, \dots, f_\nu \in P$ and $c_1, \dots, c_\mu \in K$.
In the quotient ring $P/I$, this implies the equality of residue classes:
$ 0 = \bar{f} = c_1 \bar{t}_1 + \dots + c_\mu \bar{t}_\mu. $
By assumption, the residue classes $\bar{t}_1, \dots, \bar{t}_\mu$ are $K$-linearly 
independent in $P/I$.  Hence, we must have $c_1 = \dots = c_\mu = 0$.
The expansion of $f$ therefore becomes $f = f_1 g_1 + \dots + f_\nu g_\nu$, which shows that $f \in \langle G \rangle$.
\end{proof}

The following proposition addresses the existence and uniqueness of border bases.

\begin{proposition}\textbf{\textup{(Existence and Uniqueness of Border Bases)}}\\
\label{prop-UniqueBorder}%
Let $\OO=\{t_1,\dots,t_\mu\}$ be an order ideal, let~$I\subseteq P$ 
be a zero-dimensional ideal, and assume that the residue classes 
of the elements of~$\OO$ form a $K$-vector space basis of~$P/I$.

\begin{myenumerate}
\item There exists a unique $\OO$-border basis of~$I$.

\item  Let $G$ be an $\OO$-border prebasis whose elements are in~$I$.
Then~$G$ is the $\OO$-border basis of~$I$.

\item Let $k$ be the field of definition of~$I$.
Then the $\OO$-border basis of~$I$ is contained 
in~$k[x_1,\dots,x_n]$.

\end{myenumerate}
\end{proposition}

\begin{proof}
See~\cite[Proposition~6.4.17]{KR2}.
\end{proof}

\medskip
\subsection{Border Bases and Gr\"obner Bases}
\label{subsec-BBandGB}

Given a zero-dimensional ideal, is there a relation between border 
bases and Gr\"obner bases? To answer this question, we recall that,
given an order ideal~$\OO\subset \TT^n$, its complement
$\TT^n \setminus \OO$ is the set of terms of a monomial ideal.
We know that every monomial ideal has a unique minimal set
of generators (see~\cite[Proposition~1.3.11]{KR1}).

\begin{definition}\label{def-corners}
Given an order ideal~$\OO\subset \TT^n$, the elements in the unique minimal
set of generators of the monomial ideal corresponding to~$\TT^n\setminus \OO$
are called the \textbf{corners} of~$\OO$. 
\end{definition}

\noindent
\hspace{0.5cm}
\begin{minipage}[c]{0.45\textwidth}
\small
The diagram illustrates the appropriateness of the name:
the \textbf{corners} (marked with~$\times$) are the ``minimal" 
monomials outside~$\OO$. They are the unique minimal 
generators of the monomial ideal $\langle \mathbb{T}^n \setminus \OO \rangle$.
\end{minipage}
\hfill
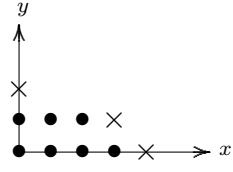
\begin{minipage}[c]{0.40\textwidth}
\leavevmode
\beginpicture
\setcoordinatesystem units <0.42cm,0.42cm>
\setplotarea x from 0 to 6, y from 0 to 4
\axis left /
\axis bottom /

\arrow <2mm> [.2,.67] from  5.5 0  to 6 0
\arrow <2mm> [.2,.67] from  0 3.5  to 0 4

\put {$\scriptstyle x$} [lt] <0.5mm,0.8mm> at 6.1 0
\put {$\scriptstyle y$} [rb] <1.7mm,0.7mm> at 0 4.1

\multiput {$\bullet$} at 0 0  1 0  2 0  3 0  0 1  1 1  2 1 /

\multiput {$\times$} at 0 2  3 1  4 0 /
\endpicture
\end{minipage}

Our next proposition clarifies the relationship between 
border bases and Gr\"obner bases of a zero-dimensional polynomial ideal. 
Moreover, it provides strong support for \textcolor{blue}{Reason}~\ref{reason1} 
stated at the beginning of this section.

\begin{proposition}
\textbf{\textup{(Border Bases, Gr\"obner Bases, and Corners)}}\\
\label{prop-bordandGr}%
Let~$\sigma$ be a term ordering on~$\TT^n$, and let
$\OO_\sigma(I)$ be the order ideal $\TT^n\setminus \LT_\sigma\{I\}$. 
\begin{myenumerate}
\item There exists a unique $\OO_\sigma(I)$-border basis~$G$ of~$I$, 

\item The reduced $\sigma$-Gr\"obner basis of~$I$ 
is the subset of~$G$ corresponding to the corners of~$\OO_\sigma(I)$.
\end{myenumerate}
\end{proposition}%
  \vspace{-0.5cm} 
\begin{proof}
By Macaulay's Basis Theorem (see~\cite[Theorem~1.5.7]{KR1}), 
the residue classes of the elements 
in~$\OO_\sigma(I)$ form a $K$-vector space basis of~$P/I$. 
Thus, Proposition~\ref{prop-UniqueBorder} ensures the existence and 
uniqueness of the $\OO_\sigma(I)$-border basis of~$I$.

To prove the second claim, let $b \in \mathbb{T}^n \setminus \mathcal{O}_\sigma(I)$ 
be a corner of $\mathcal{O}_\sigma(I)$. The element of the reduced 
$\sigma$-Gr\"obner basis of $I$ with leading term $b$ has the form 
$g = b - \text{NF}_{\sigma,I}(b)$, where $\text{NF}_{\sigma,I}(b) \in \langle \mathcal{O}_\sigma(I) \rangle_K$. 
Since $g \in I$ and it has the form of a border prebasis element, the uniqueness part of 
Proposition~\ref{prop-UniqueBorder} implies that this Gr\"obner basis element 
must coincide with the border basis element corresponding to the corner $b$.
\end{proof}

\begin{definition}\label{def-marked}
Let $P$ be a polynomial ring. 

\begin{myenumerate}
\item Let $f$ be a non-zero polynomial in~$P$. 
Then the pair $(f, t)$ is said to be a \textbf{marked polynomial} if  
$t \in \Supp(f)$ with coefficient~1. 

\item Let $(f_1,\dots, f_\nu)$ be a  tuple of non-zero polynomials in~$P$.  
Then the tuple $(f_1,\dots, f_\nu)$ is said to be \textbf{marked} 
by a tuple of terms $(t_1, \dots, t_\nu)$  if 
$(f_1, t_1),\dots, (f_\nu, t_\nu)$ are marked polynomials.

\item Let $(f_1,\dots, f_\nu)$ be a  tuple of non-zero polynomials in~$P$. 
Then the tuple $(f_1,\dots, f_\nu)$ 
is said to be \textbf{coherently marked} by a tuple of terms 
$(t_1, \dots, t_\nu)$ if there exists a term ordering $\sigma$ such 
that $t_i = \LT_\sigma(f_i)$ for $i=1,\dots, \nu$.
\end{myenumerate}
\end{definition}

\begin{myremark}\label{rem-marked}
In the following, when we refer to a \textbf{marked border prebasis} 
$G=(g_1, \dots, g_\nu)$ associated with an order ideal $\mathcal{O}$, 
where $g_j = b_j - \sum_{i=1}^\mu \alpha_{ij}t_i$, we implicitly assume 
that each $g_j$ is marked by its corresponding border term 
$b_j \in \partial\OO$.
\end{myremark}

\begin{proposition}\label{prop-cohmark}
Let $\OO$ be an order ideal such that the residue classes of
the elements of $\OO$ form a $K$-vector space basis of $P/I$.
 Let $G$  be the $\OO$-border basis of $I$, and let~$G'$ be the subset of $G$
 consisting of the elements marked by
the corners of $\OO$. 
Then the following conditions are equivalent.

\begin{myenumerate}
\item There exists a term ordering $\sigma$ such that $\OO =\OO_\sigma(I)$.

\item The elements in $G'$ are marked coherently.

\item The elements in $G$ are marked coherently.
\end{myenumerate}%
Moreover, if these conditions are satisfied, then $G'$ is the 
reduced $\sigma$-Gr\"obner basis of~$I$.
\end{proposition}

\goodbreak
\begin{proof}
Let us prove that (1) implies both (2) and the additional claim. The fact
that $G'$ is the reduced $\sigma$-Gr\"obner basis of~$I$  follows 
from Proposition~\ref{prop-bordandGr}, and hence $G'$ is marked 
coherently.
Now we show that (2) implies (3). For every
polynomial $g\in G\setminus G'$, there exists a polynomial $g' \in G'$ 
such that the marked term of $g$ is of the form $b= t\LT_\sigma(g')$. 
Then the support of the polynomial $g - t g'$ is contained in $\OO_K$, 
and therefore $g= t g'$. Consequently, $\LT_\sigma(g) = t\LT_\sigma(g') = b$
which proves that also $g$ is marked coherently with respect to $\sigma$.
Since it is obvious that (3) implies (2), it remains to be shown that (2) implies~(1).
Let $\sigma$  be a term ordering which marks $G'$ coherently. 
Denote the monomial ideal generated by the leading terms 
of the elements in $G'$ by $\LT_\sigma(G')$. 
Since we have $\LT_\sigma(G')\subseteq \LT_\sigma(G)$
we get $\OO_\sigma(I)= {\mathbb T}^n\setminus  \LT_\sigma(I) 
\subseteq {\mathbb T}^n \setminus \LT_\sigma(G') = \OO$. 
Both the residue classes of the elements of $\OO_\sigma(I)$ and those of $\OO$ 
form a $K$-vector space basis of $P/I$. Since they are both bases of the 
same finite-dimensional vector space, the inclusion $\OO_\sigma(I) \subseteq \OO$ 
must be an equality, which proves (1).
\end{proof}

\begin{example}\textbf{(Example~\ref{ex-bb-no-gb} continued)}\\
\label{ex-bb-no-gb continued}%
We continue the discussion initiated in  Example~\ref{ex-bb-no-gb}.
 There, the ideal $I$ is symmetric under the interchange of~$x$ and~$y$.
 If we look at the polynomial $g_1= xy -x^2 -y^2$ we note that for 
 every term ordering, its leading term is either $x^2$ or $y^2$.
 Consequently the tuple $(xy -x^2 -y^2,\,x^3,\,x^2y,\,xy^2,\,y^3)$ cannot be 
 coherently marked by the tuple of border terms $(xy,\,x^3,\,x^2y,\,xy^2,\,y^3)$.
 
 There are only two possible Gr\"obner bases of $I$. They are 
 $(x^2 -xy +y^2,  y^3,  xy^2)$ and $(y^2 -xy +x^2,  x^3,  x^2y)$.
 Neither is symmetric, thus the corresponding quotient bases
$(1,  y,  x,  y^2,  xy)$ and $( 1,  x,  y,  x^2,  xy)$ 
(see \cocoa-Example~\ref{coex-bb-no-gb}) 
are not symmetric. 
However, we have seen in Example~\ref{ex-bb-no-gb} that the 
symmetric set of terms
$\OO= \{1,\,x,\,y,\,x^2,\,y^2\}$
represents a vector space basis of~$\bbb Q[x,y]/I$.
In this case,  Gr\"obner bases break the symmetry while border 
bases do not!
\end{example}

The example above  provides evidence for the correctness 
of \textcolor{blue}{Reason \ref{reason2}} 
stated at the beginning of this section.
However, the following example shows that the merit of
border bases in \textit{preserving} symmetry
must be taken with a grain of salt.

\begin{example}\label{ex-notsymmetric}
Let $P=\mathbb{Q}[x,y]$ and $I = \langle x^2 + y^2 - 1, \, xy - 1 \rangle$.
We have $\dim_{\mathbb{Q}}(P/I) = 4$ and the only symmetric order ideal 
consisting of four terms is $\OO = \{ 1, x, y, xy \}$. However, the ideal
$I$ does not have an $\OO$-border basis because $xy-1 \in I$.
\end{example}

The following remark shows that, notwithstanding the fact that a
zero-dimensional ideal~$I$ may have more border bases than Gr\"obner bases,
not every order ideal  represents a basis of $P/I$.

\begin{myremark}\label{rem-notallO}
Let $P=\QQ[x, y]$ and $I = \langle x^2, y^2 \rangle$. Then the order ideal 
$\OO =\{1, x, y, xy\}$ is the only one that represents a  
$\QQ$-vector space basis of $\QQ[x, y]/I$.
\end{myremark}

The following proposition shows that, when we divide by a border basis, 
the normal remainder does not depend on the order of the elements.

\begin{proposition}\label{prop-normalrem}
Let $G=(g_1,\dots,g_\nu)$ be the $\OO$-border
basis of an ideal $I\subseteq P$, let $\pi \colon \{1,\dots,\nu\}\To 
\{1,\dots,\nu\}$ be a permutation, and
let ${G'=(g_{\pi(1)},\dots,g_{\pi(\nu)})}$ be the corresponding
permutation of the tuple~$G$. Then we have 
$\NR_{\OO,G}(f)=\NR_{\OO,G'}(f)$ 
for every polynomial $f\in P$.
\end{proposition}

\begin{proof}
See~\cite[Proposition~6.4.19]{KR2}.
\end{proof}

This result allows us to generalize the concept of a normal form 
to border basis theory.

\begin{definition}\label{def-NF}
Let $G=\{g_1,\dots,g_\nu\}$ be the $\OO$-border basis of~$I$. 
The \textbf{normal form} of a polynomial $f\in P$
with respect to~$\OO$ is the polynomial 
$\NF_{\OO,I}(f)=\NR_{\OO,G}(f)$. 
\end{definition}

\smallskip
\begin{myremark}\label{rem-NF}
The normal form $\NF_{\OO,I}(f)$ of $f\in P$ can be calculated by
dividing~$f$ by the $\OO$-border basis of~$I$. It is zero if and
only if $f\in I$. 
\end{myremark}

We now turn to \textcolor{blue}{Reason~\ref{reason3}}, 
stated at the beginning of the section, 
which concerns the numerical stability of border bases.  
To this end, we use the following example (see also \cocoa-Example~\ref{coex-numstabil}).

\begin{example}\label{ex-numstabil}
Let $P= \QQ[x,y]$ and let $I=\langle f_1, f_2 \rangle$ where $f_1 = \tfrac{1}{4}\, x^2+y^2-1$ 
and $f_2 = x^2 + \tfrac{1}{4}\, y^2-1$. 
Let  $\epsilon $ be a \textit{small number}, and let
$I^\epsilon = \langle f_1^\epsilon, f_2^\epsilon \rangle$ where
$$
f_1^\epsilon =  \tfrac{1}{4}\, x^2 + y^2 +\epsilon\, xy -1 \qquad
f_2^\epsilon = x^2+ \tfrac{1}{4}\, y^2 +\epsilon\, xy-1
$$

The zeros of the ideals $I$ and $I^\epsilon$ are illustrated in the 
following \textcolor{blue}{Figure~\ref{fig:zeros}}, depicted in the real plane.
The intersection of the two conics~$\Cal{Z}(f_1)$ and $\Cal{Z}(f_2)$ in $\mathbb R$ 
consists of the set $\bbb X$ of four points with coordinates $(\pm \sqrt{4/5},\, \pm\sqrt{4/5})$, 
while the intersection of the two perturbed conics $\Cal{Z}(f_1^\epsilon)$ and $\Cal{Z}(f_2^\epsilon)$ 
consists of  four points  close to those in~$\bbb X$.

\begin{figure}[ht]
    \centering
    \input{pic64a.tex} \qquad \input{pic64b.tex}
    \caption{$\Cal{Z}(f_1)$, $\Cal{Z}(f_2)$  (left), \qquad \quad  $\Cal{Z}(f_1)$, $\Cal{Z}(f_2)$,
    $\Cal{Z}(f_1^\epsilon)$, $\Cal{Z}(f_2^\epsilon)$   (right)}
    
    \label{fig:zeros}
\end{figure}

The set $\{x^2-\tfrac{4}{5},\,y^2-\tfrac{4}{5}\}$ 
is the reduced Gr\"obner basis
of the ideal $I$. Therefore we have
$\LT_\sigma(I)= (x^2,\,y^2)$, and if we let $\OO = \{1,\,x,\,y,\,xy\}$, then
the residue classes of the
terms in $\TT^2\setminus \LT_\sigma\{I\}=\OO$ form a
\hbox{$\QQ$-vector} space basis of~$\QQ[x,y]/I$.

The border of~$\OO$ is $\partial\OO=\{x^2, x^2y, xy^2, y^2\}$,
and both $I$ and $I^\epsilon$  have an $\OO$-border basis.
The $\OO$-border basis of~$I$ is 
$$
\{y^2-\tfrac{4}{5},\ x^2-\tfrac{4}{5},\  xy^2-\tfrac{4}{5}x, \, x^2y-\tfrac{4}{5}y, \} \eqno{(a)}
$$
The $\OO$-border basis of~$I^\epsilon$ is
\begin{equation}
\left\{ 
\begin{aligned}
&y^2- (\tfrac{4}{5} -\tfrac{4}{5}\,\epsilon xy),  & &x^2- (\tfrac{4}{5} -\tfrac{4}{5}\,\epsilon xy), \\
&xy^2- (\tfrac{16\epsilon}{16\epsilon^2-25}\,y -\tfrac{20}{16\epsilon^2-25}\,x \big),& &
x^2y -( \tfrac{20}{16\epsilon^2-25}\,y -  \tfrac{16\epsilon}{16\epsilon^2-25}\,x \big)
\end{aligned} 
\right\} \tag{b} \label{base-b}
\end{equation}

Next, using \cocoa\, we compute the Gr\"obner Fan  (see Definition~\ref{def-GFan}) of $I^\epsilon$  
viewed as an ideal in $R=\QQ(\epsilon)[x,y]$ (see \cocoa-Example~\ref{coex-numstabil}). 
In this way we get all marked reduced Gr\"obner bases of the ideal $I^\epsilon$
\begin{gather*}
\big( (x^2 -y^2, x^2), \ \  (xy +\tfrac{5}{4e}y^2 -\tfrac{1}{e}, y^2), \ \  
(y^3  -\tfrac{16e}{16e^2 -25}x +\tfrac{20}{16e^2 -25}_{\mathstrut} y, y^3) \big)  \\
 \big( (y^2 -x^2, y^2), \   (xy  +\tfrac{5}{4e}x^2 -\tfrac{1}{e}, xy), \ \  
 (x^3 -\tfrac{16e}{16e^2 -25}y +\tfrac{20}{16e^2 -25}_{\mathstrut}  x, x^3) \big) \\
\big( (y  -\tfrac{16e^2 +25}{16e}x^3 -\tfrac{5}{4e} x, y),\ \  
(x^4 +\tfrac{40}{16e^2 -25}x^2 -\tfrac{6}{16e^2 -25} _{\mathstrut}, x^4) \big) \\
\big( (x  -\tfrac{16e^2 +25}{16e}y^3  -\tfrac{5}{4e}  y, x) , \ \ 
(y^4 +\tfrac{40}{16e^2 -25} y^2 -\tfrac{6}{16e^2 -25}, y^4) \big)
\end{gather*}
and hence the generators of the  corresponding leading term ideals and  the order ideals which represent vector space bases of $R/I^\epsilon$
\begin{gather*}
  \big( \{x^2,  xy,  y^3\}, \quad \{1,  y,  x,  y^2\} \big) \\
 \big( \{y^2,  xy,  x^3\}, \quad \{1,  x,  y,  x^2\} \big) \\
  \big( \{y,  x^4\},   \qquad       \{1,  x,  x^2,  x^3\} \big) \\
  \big( \{x,  y^4\},    \qquad         \{1,  y,  y^2,  y^3\} \big).
\end{gather*}

We may consider the ideal $I^\epsilon$ as a family of ideals with $\epsilon$ acting as a parameter.
Then, when we vary the coefficients of $xy$ in the two generators from zero to $\epsilon$, 
we see that the border basis of~$I$ changes continuously into that of $I^\epsilon$.
Thus, the border basis  is numerically stable under small perturbations of the coefficient of $xy$.

On the other hand, we have seen that a {\it small}\/ change in the coefficients of~$xy$
has led to a {\it big}\/ change in the Gr\"obner bases of~$I$
and in the associated vector space bases of~$\QQ[x,y]/I$, although the
zeros of the system have not changed much. 

In addition, this example is another instance of the phenomenon wherein 
Gr\"obner bases fail to preserve the symmetry of the given system of equations, 
whereas certain border bases do. 

In conclusion,  this example provides further  support for 
the \textcolor{blue}{Reasons} mentioned at the beginning of the section.
\end{example}

\medskip
\subsection{Commuting Matrices}
\label{subsec-Commuting Matrices}

In \cite[Subsection 6.4.B]{KR2}, further characterizations of border 
bases are given (see \cite[Propositions~6.4.23, 6.4.25, and 6.4.34]{KR2}), 
and it is shown that one can fully imitate the characterization of Gr\"obner bases, 
including an algorithm for computing a border basis of a zero-dimensional
ideal (see \cite[Theorem~6.4.36]{KR2}).

Here, we concentrate on a special characterization of border bases 
introduced by Bernard Mourrain in \cite{Mou}, which turns out to be 
essential in preparing the ground for the discussion of border basis schemes.

Throughout the subsection, we let $P=K[x_1, \dots, x_n]$, 
let $\OO = \{t_1, \dots, t_\mu\}$ be an order ideal, and 
let $\langle \OO \rangle_K$
be the $K$-vector space spanned by $\OO$. 
Let $\partial\OO = \{b_1, \dots, b_\nu\}$ be the border of $\OO$, and let 
$G = \{g_1, \dots, g_\nu\}$ be an $\OO$-border prebasis, where we have 
$g_j = b_j - \sum_{i=1}^\mu \alpha_{ij}t_i$ with $\alpha_{ij} \in K$
for $1 \le i \le \mu$ and $1 \le j \le \nu$.
Moreover, the column vector $(0, \dots, 0, 1, 0, \dots, 0)\tr$ with $1$ in position $k$ 
will be denoted by $e_k$.

\begin{definition}\label{def-linoper}
For every $i = 1, \dots, n$, the endomorphism of $\langle \OO \rangle_K$ defined by 
sending~$t$ to $\text{NR}_{\OO, G}(x_i t)$ for all $t \in \OO$ is denoted by $m_{x_i}$
and called the \textbf{formal multiplication endomorphism} with respect to $x_i$.
Its associated matrix with respect to the basis $\OO$ is denoted by $M_{x_i}$
and called the \textbf{formal multiplication matrix} with respect to $x_i$.
Consequently, if $x_i t_k = t_s \in \OO$, then 
$m_x(t_k)=\text{NR}_{\OO, G}(x_i t_k) = t_s$, 
and the $k$-th column of $M_{x_i}$ is $e_s$. 
Instead, if we have $x_i t_k = b_j \in \partial\OO$, then 
$m_{x_i}(t_k)=\text{NR}_{\OO, G}(x_i t_k) = \sum_{r=1}^\mu \alpha_{rj} t_r$, 
and the $k$-th column of $M_{x_i}$ is $(\alpha_{1j}, \dots, \alpha_{\mu j})\tr$.
\end{definition}

\begin{example}\label{ex-multmat}
We revisit Example~\ref{ex-bb-no-gb}. 

\smallskip
\noindent
\begin{minipage}[c]{0.55\textwidth}
With the basis $\OO = \{1, y, x, y^2, x^2\}$, the images of the 
formal multiplication by $x$ are
$m_x(1) = x$, $m_x(y) = y^2 + x^2$, $m_x(x) = x^2$, and $m_x(y^2) = m_x(x^2) = 0$. 
The resulting formal multiplication matrix $M_x$ is shown on the right.
\end{minipage}
\hfill
\begin{minipage}[c]{0.4\textwidth}
\centering
$M_x = \begin{bmatrix}
0 & 0 & 0 & 0 & 0 \\
0 & 0 & 0 & 0 & 0 \\
1 & 0 & 0 & 0 & 0 \\
0 & 1 & 0 & 0 & 0 \\
0 & 1 & 1 & 0 & 0 
\end{bmatrix}$
\end{minipage}
\end{example}

With  the next result we introduce a new structure on $\langle \OO\rangle_K$ 
under the special condition that the formal multiplication endomorphisms
are pairwise commuting.

\begin{lemma}\label{lem-cyclic}
Assume that the formal multiplication endomorphisms are pairwise commuting.
Then the vector space $\langle \OO\rangle_K$ acquires the structure of 
a cyclic $P$-module with generator $1$  via 
$$f\circ (c_1 t_1+\dots +c_\mu t_\mu)=
(t_1,\dots,t_\mu)f(M_{x_1},\dots,M_{x_n})(c_1,\dots,c_\mu)^{\rm tr}$$
with the convention that $f(M_{x_1},\dots,M_{x_n}) = f\cdot \rm{Id}_\mu$ if $f$ is a constant.
\end{lemma}

\begin{proof}
Since $1 \in \OO$, we may assume that $t_1=1$.
The standard properties required for the operation $\circ$ to define a  $P$-module 
structure are easily checked.
It remains to prove that the module is cyclic. To do so, we use induction on the
degree to show that $t_i\circ 1=t_i$ for $i=1,\dots,\mu$. Let~$e_i$ denote the
matrix of size $\mu\times 1$ whose $i^{\rm th}$ entry is~1
and whose other entries are~0. The induction starts 
with $t_1=1$. By definition we have    $1\circ 1= (1, t_2,\dots,t_\mu)\text{Id}_\mu\, e_1 =1$.
Next, we let $t_i=x_j\,t_k$, so that by induction we have 
$t_k\circ 1=  (1, t_2,\dots,t_\mu)t_k(M_{x_1},\dots,\ M_{x_n})\, e_1  = t_k$ and hence 
$t_k(M_{x_1},\dots, M_{x_n})\,e_1\overset{\text{(0)}}{=} e_k$. Consequently, we get
\begin{eqnarray*}
t_i\circ 1&=& (1, t_2,\dots,t_\mu)t_i(M_{x_1},\dots,M_{x_n})\, e_1 
\overset{\text{(1)}}{=}
(1, t_2,\dots,t_\mu)M_{x_j}\,t_k(M_{x_1},\dots,M_{x_n})\, e_1\\
&\overset{\text{(2)}}{=}& (1, t_2,\dots,t_\mu)M_{x_j}\, e_k \overset{\text{(3)}}{=}
(1, t_2,\dots,t_\mu)e_i = t_i
\end{eqnarray*}
where equality~(1) follows  from $t_i = x_jt_k$, equality~(2)  follows from equality~(0)
proved above, and equality~(3) is a consequence of  $M_{x_j}e_k =e_i$ which follows 
from $t_i = x_jt_k$, and the definition of formal multiplication matrices.
\end{proof}

The following example illustrates why the assumption of 
commutativity is essential in the above lemma.

\begin{example}\label{ex-commuteessential}
Let $f = xy$. Since $x y=y x$ in~$P$,  the module action must satisfy
$$
f\circ (c_1 t_1+\dots +c_\mu t_\mu) =
(t_1,\dots,t_\mu)M_x\, M_y(c_1,\dots,c_\mu)^{\rm tr} \\
= (t_1,\dots,t_\mu)M_y\, M_x(c_1,\dots,c_\mu)^{\rm tr} 
$$
Since $\{t_1, \dots, t_\mu\}$ is a basis of $\langle \OO \rangle_K$, we must have
$$M_x\, M_y(c_1,\dots,c_\mu)^{\rm tr}=
M_y\, M_x(c_1,\dots,c_\mu)^{\rm tr}$$
for every $(c_1,\dots,c_\mu)\in K^\mu$, which forces the condition
$M_x\, M_y = M_y\, M_x$.
\end{example}

We are ready to prove the main result of this section.

\medskip
\begin{theorem}\label{thm-Mourrain}%
\textbf{\textup{(Border Bases and Commuting Matrices)}}\\
Let $P=K[x_1, \dots, x_n]$, 
let $\OO = \{t_1, \dots, t_\mu\}$ be an order ideal,
let $\partial\OO = \{b_1, \dots, b_\nu\}$ be the border of $\OO$, 
let  $G = \{g_1, \dots, g_\nu\}$ be an $\OO$-border prebasis, 
and let $I=\langle G \rangle$. Then the following conditions are equivalent.
\begin{myenumerate}
 \item The set~$G$ is an $\OO$-border basis of~$I$. 

\item The formal multiplication matrices  are pairwise commuting.
\end{myenumerate}
In that case, the formal multiplication matrices represent the 
formal multiplication endomorphisms of~$P/I$ with respect to the basis 
$\{\bar t_1,\dots,\bar t_\mu\}$.
\end{theorem}

\begin{proof}
Let $m_{x_1}, \dots, m_{x_n}$ and  $M_{x_1}, \dots, M_{x_n}$  be  the formal 
multiplication endomorphisms and matrices. 
Recall that  $g_j = b_j - \sum_{i=1}^\mu \alpha_{ij}t_i$ with $\alpha_{ij} \in K$
for $1 \le i \le \mu$ and $1 \le j \le \nu$.

First,  we prove  $(1)\implies (2)$. 
Since~$G$ is an $\OO$-border basis, the set
$\{\bar t_1,\dots,\bar t_\mu\}$ is a $K$-vector space basis
of~$P/I$, hence there exists an isomorphism of $K$-vector spaces
$\theta \colon \langle \OO \rangle_K \to P/I$, defined by $t_i \to \bar{t}_i$ 
for $i =1,\dots, \mu$.
Via $\theta$, the formal multiplication endomorphisms 
$m_{x_1}, \dots, m_{x_n}$  define $K$-linear endomorphisms
$\phi_{x_1}, \dots, \phi_{x_n}$ on $P/I$ whose associated matrices 
are still $M_{x_1}, \dots, M_{x_n}$. 
The product $x_r\,t_\ell$ either equals some term~$t_i$ in the order 
ideal~$\OO$ or some border term $b_j\in \partial\OO$. 
In the former case, we have $\phi_r(\bar t_\ell)=\bar t_i=
x_r \bar t_\ell$, and in the latter case we have
$\phi_r(\bar t_\ell)=\alpha_{1j}\bar t_1 +\cdots +\alpha_{\mu j}
\bar t_\mu= \bar b_j=x_r\bar t_\ell$. It follows that 
the map $\phi_r$  is multiplication  by~$x_r$ for $r=1,\dots,n$.
Therefore we have $M_{x_r}M_{x_s}= M_{x_s}M_{x_r}$ for 
$r,s\in\{1,\dots,n\}$, i.e.\ the matrices $M_{x_1},\dots, M_{x_n}$ 
are pairwise commuting.

Next, we prove $2\implies 1$. It follows from Lemma~\ref{lem-cyclic} that 
there exists a surjective homomorphism $\theta \colon P \To \langle \OO\rangle_K$ 
of $P$-modules defined by $\theta(f) = f\circ 1$. 
Let us show that $I \subseteq \Ker(\theta)$. 
Assume that  $b_j=x_k\,t_\ell$ with $t_\ell \in \OO$. Then we have
\begin{equation*}
\begin{array}{lcl}
g_j\circ 1
&=& {b_j(M_{x_1}}_{\mathstrut}, \dots, M_{x_n})e_1-\sum_{i=1}^\mu \alpha_{ij}
      t_i(M_{x_1}, \dots, M_{x_n})e_1\\
&=& M_{x_k}\,t_\ell(M_{x_1}, \dots, M_{x_n})e_1-
    \sum_{i=1}^\mu  \alpha_{ij}e_i
= {M_{x_k}}_{\mathstrut}\,e_\ell-\sum_{i=1}^\mu \alpha_{ij}e_i \\
&=& \sum_{i=1}^\mu \alpha_{ij}e_i-\sum_{i=1}^\mu \alpha_{ij}e_i=0
\end{array}
\end{equation*}
 Consequently, the map $\theta$
induces a surjective homomorphism  $\bar{\theta} \colon P/I \To \langle \OO\rangle_K$ of $P$-modules.
From $\dim_K(P/I) = \dim_K(\langle \OO\rangle_K)$ we deduce that $\bar{\theta}$
is an isomorphism. This shows that~$G$
is an $\OO$-border basis of~$I$, and concludes the proof.
\end{proof}

We go back to Examples~\ref{ex-bb-no-gb} and~\ref{ex-nobb}, and  check
their conclusions using formal multiplication matrices.

\begin{example}\label{ex-comm-bb-no-gb}
In Example~\ref{ex-bb-no-gb} the formal multiplication matrices are
$$M_x = \begin{bmatrix}
0 & 0 & 0 & 0 & 0 \\
0 & 0 & 0 & 0 & 0 \\
1 & 0 & 0 & 0 & 0 \\
0 & 1 & 0 & 0 & 0 \\
0 & 1 & 1 & 0 & 0 
\end{bmatrix} \qquad
M_y = \begin{bmatrix}
0 & 0 & 0 & 0 & 0 \\
1 & 0 & 0 & 0 & 0 \\
0 & 0 & 0 & 0 & 0 \\
0 & 1 & 1 & 0 & 0 \\
0 & 0 & 1 & 0 & 0 
\end{bmatrix}$$
and 
$$M_x M_y -M_yM_x= \begin{bmatrix}
0 & 0 & 0 & 0 & 0 \\
0 & 0 & 0 & 0 & 0 \\
0 & 0 & 0 & 0 & 0 \\
0 & 0 & 0 & 0 & 0 \\
0 & 0 & 0 & 0 & 0 
\end{bmatrix}
$$
(see \cocoa-Example~\ref{coex-bb-no-gb}). 
This confirms that $G$ is a border basis.
\end{example}

\begin{example}\label{ex-comm-nobb}
In Example~\ref{ex-nobb} the formal multiplication matrices are
$$M_x = \begin{bmatrix}
0 & 0 & 0 & 0 & 0 \\
0 & 0 & 0 & 0 & 0 \\
1 & 1 & 0 & 0 & 0 \\
0 & 1 & 0 & 0 & 0 \\
0 & 1 & 1 & 0 & 0 
\end{bmatrix} \qquad
M_y = \begin{bmatrix}
0 & 0 & 0 & 0 & 0 \\
1 & 0 & 0 & 0 & 0 \\
0 & 0 & 1 & 0 & 0 \\
0 & 1 & 1 & 0 & 0 \\
0 & 0 & 1 & 0 & 0 
\end{bmatrix}$$
and 
$$M_x M_y -M_yM_x= \begin{bmatrix}
0 & 0 & 0 & 0 & 0 \\
0 & 0 & 0 & 0 & 0 \\
0 & -1 & 0 & 0 & 0 \\
0 & -1 & 0 & 0 & 0 \\
0 & -1 & 1 & 0 & 0 
\end{bmatrix}
$$
which confirms that $G$ is not a border basis.
\end{example}

\newpage   
\section{Border Basis Schemes}
\label{sec-Border Bases Schemes}

We begin this section by recalling the final part of Section~\ref{sec-Border Bases}, 
specifically Theorem~\ref{thm-Mourrain}, which establishes the connection 
between border bases and commuting matrices.

Let~$\OO = \{t_1, \dots, t_\mu\}$ be an order ideal in~$\mathbb{T}^n$, and let
$\partial\OO = \{b_1, \dots, b_\nu\}$ be its border. Recall that an $\OO$-border prebasis is a tuple
of polynomials $G=(g_1,\dots,g_\nu)$, where 
$g_j=b_j-\sum_{i=1}^\mu\,\alpha_{ij}t_i$ with $\alpha_{ij}\in K$ 
(see Definition~\ref{def-border!prebasis}).

\medskip
\subsection{Introduction to Border Basis Schemes}
\label{subsec-IntroToBBS}

From Definition~\ref{def-linoper} we recall the notions of formal multiplication endomorphism
and formal multiplication matrices. In particular, it was shown that 
\begin{myenumerate}
\item if $x_i t_k = t_s \in \OO$, then 
the $k$-th column of $M_{x_i}$ is $e_s$, 
\item if  $x_i t_k = b_j \in \partial\OO$, then the
$k$-th column of $M_{x_i}$ is $(\alpha_{1j}, \dots, \alpha_{\mu j})\tr$.
\end{myenumerate}

The key idea here is to replace  the coefficients $\alpha_{ij} \in K$ with 
new indeterminates. We now proceed in this way.

\begin{definition}\label{def-BBS} 
Let $\OO = \{t_1, \dots, t_\mu\}$ be an order ideal in~$\mathbb{T}^n$, and let
$\partial\OO = \{b_1, \dots, b_\nu\}$ be its border. 
Let $X = (x_1, \dots, x_n)$ and, for all $i=1,\dots,\mu$ and $j=1,\dots,\nu$, let~$c_{ij}$ 
be new indeterminates, and let ${C=(c_{11}, c_{12}, \dots, c_{\mu\nu})}$ 
be the tuple of these indeterminates.
Moreover,  let $K[X] = K[x_1, \dots, x_n]$, let ${K[C] = K[c_{11},\dots,c_{\mu\nu}]}$, and let
$K[C,X] = K[c_{11}, c_{12}, \dots, c_{\mu\nu}, x_1,\dots, x_n]$.
\begin{myenumerate}

\item The tuple of polynomials $\mathsf{G}=(g_1,\dots,  g_\nu)$ in $K[C,X]$, 
where  $g_j =b_j-\sum_{i=1}^\mu\,c_{ij}t_i$ for $j=1,\dots,\nu$, is called the 
{\bf generic $\OO$-border prebasis}.  

\item For $r=1,\dots,n$, the matrix $\mathcal{M}_r  \in \Mat_\mu(K[C])$,
whose $k$-th column is $e_s$ if  $x_r t_k = t_s \in \OO$ and is $(c_{1j},  \dots, c_{\mu j})\tr$
if $x_r t_k = b_j \in \partial\OO$,
is called the $r$-th {\bf generic multiplication matrix} for~$\OO$.

\item Let $I(\BO) \subseteq K[C]$ be the ideal generated
by the entries, called \textbf{natural generators}, of the commutator 
matrices $\mathcal{M}_i \mathcal{M}_j - \mathcal{M}_j
\mathcal{M}_i$, where $1\le i < j \le n$. Then the subscheme of $\mathbb{A}^{\mu \nu}_K$
defined by $I(\BO)$ is called the {\bf $\OO$-border basis scheme} and is denoted by~$\BO$.

\item The coordinate ring  $K[C] / I(\BO)$ of~$\BO$ is denoted by $B_\OO$.

\item Let $U_\OO = B_\OO[X] / \langle \bar{g}_1,\dots, \bar{g}_\nu\rangle 
\cong K[C,X]/\langle I(\BO), g_1,\dots,g_\nu\rangle$.
The canonical ring homomorphism $B_\OO \to U_\OO$ is 
called the \textbf{universal $\OO$-border basis family}, 
and~$U_\OO$ is called the \textbf{universal $\mathcal{O}$-border basis algebra}.
\end{myenumerate}
\end{definition}

\medskip
\begin{myremark}\label{rem-Rob}
Gr\"obner basis schemes were introduced in~\cite{Rob}, where it is shown
that, in some cases, border basis schemes can be described
in terms of suitable Gr\"obner basis schemes.
\end{myremark}

\bigskip
The following example illustrates the above definitions.

\bigskip
\goodbreak
\begin{example}{\bf (The (2,2)-Box)}\label{ex-box22}\\
\begin{minipage}[c]{0.53\textwidth}
\medskip
\small
In the polynomial ring $P=K[x, y]$ over a field~$K$, 
the order ideal  $\OO = \{t_1,t_2,t_3,t_4\} = \{1,\, y,\, x,\, xy\}$ 
is called  the  {\bf (2,2)-box}. Its border is 
$\partial\OO= \{b_1, b_2, b_3, b_4\} = \{y^2,  x^2,  xy^2,  x^2y\}$.
The picture shows the order ideal ($\bullet$)
and its border  ($\circ$).
\end{minipage}
\hfill
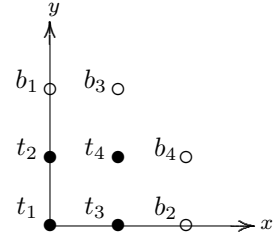
\begin{minipage}[c]{0.30\textwidth}
\leavevmode
\beginpicture
\setcoordinatesystem units <0.6cm,0.6cm>
\setplotarea x from 0 to 4, y from 0 to 4
\axis left /
\axis bottom /
\arrow <2mm> [.2,.67] from  4 0  to 4.5 0
\arrow <2mm> [.2,.67] from  0 4  to 0 4.5
\put {$\scriptstyle x$} [lt] <0.5mm,0.8mm> at 4.5 0
\put {$\scriptstyle y$} [rb] <1.7mm,0.7mm> at 0 4.5
\multiput {$\bullet$} at 0 0  1.5 0   0 1.5  1.5  1.5 /
\multiput {$\circ$} at 0 3  1.5  3   3 1.5   3 0 /
\put {\footnotesize $t_1$} at -.5 .4 
\put {\footnotesize $ t_2$} at -.5 1.7 
\put {\footnotesize $t_3$} at  1  .4 
\put {\footnotesize $t_4$} at  1 1.7 

\put {\footnotesize $b_1$} at -.5 3.2
\put {\footnotesize $b_2$} at  2.55 .4
\put {\footnotesize $b_3$} at 1  3.2
\put {\footnotesize $b_4$} at  2.55 1.7 

\endpicture
\end{minipage}

\smallskip
For the computations see  \cocoa-Example~\ref{coex-22Box}.

We have 
\footnotesize
\begin{align*}
g_1 &= y^2 -c_{11} -c_{21}y  -c_{31}x - c_{41}xy, \qquad  & 
g_2 &= x^2 -c_{12} -c_{22}y  -c_{32}x - c_{42}xy, \\
g_3 &= xy^2 -c_{13} -c_{23}y  -c_{33}x - c_{43}xy, \qquad  & 
g_4 &= x^2y -c_{14} -c_{24}y  -c_{34}x - c_{44}xy.
\end{align*}

\normalsize
The generic multiplication matrices are 
$$\mathcal{M}_x = 
\begin{bmatrix}
0 & 0 & c_{12} & c_{14} \\
0& 0 & c_{22} & c_{24} \\
1 & 0 & c_{32} &  c_{34}  \\
0 & 1 & c_{42} & c_{44} 
\end{bmatrix} \qquad
\mathcal{M}_y = \begin{bmatrix}
0 & c_{11} & 0 & c_{13} \\
1 & c_{21} & 0 & c_{23} \\
0 & c_{31} & 0 & c_{33}  \\
0 & c_{41} & 1 & c_{43}
\end{bmatrix}$$
The non-zero entries of $\mathcal{M}_x\mathcal{M}_y - \mathcal{M}_y\mathcal{M}_x$ are 
\footnotesize 
\begin{align*}
& f_{1} = -c_{12}c_{31} -c_{14}c_{41} +c_{13}, 
&& f_{2} = -c_{22}c_{31} -c_{24}c_{41} +c_{23}, \\
& f_{3} = -c_{31}c_{32} -c_{34}c_{41} -c_{11} +c_{33}, 
&& f_{4} = -c_{31}c_{42} -c_{41}c_{44} -c_{21} +c_{43}, \\
& f_{5} = -c_{11}c_{22} -c_{13}c_{42} +c_{14}, 
&& f_{6} = -c_{21}c_{22} -c_{23}c_{42} -c_{12} +c_{24}, \\
& f_{7} = -c_{22}c_{31} -c_{33}c_{42} +c_{34}, 
&& f_{8} = -c_{22}c_{41} -c_{42}c_{43} -c_{32} +c_{44}, \\
& f_{9} = -c_{11}c_{24} +c_{12}c_{33} +c_{14}c_{43} -c_{13}c_{44}, 
&& f_{10} = -c_{21}c_{24} +c_{22}c_{33} +c_{24}c_{43} -c_{23}c_{44} -c_{14}, \\
& f_{11} = -c_{24}c_{31} +c_{32}c_{33} +c_{34}c_{43} -c_{33}c_{44} +c_{13}, 
&& f_{12} = -c_{24}c_{41} +c_{33}c_{42} +c_{23} -c_{34}.
\end{align*}
\normalsize
Hence $I(\BO) = \langle f_1,\dots, f_{12}\rangle$ and the 
universal $\OO$-border family is 
$$
\BO=K[C]/\langle f_1,\dots, f_{12}\rangle 
\To U_\OO=K[C,X]/ \langle f_1,\dots, f_{12}, g_1,\dots, g_4\rangle
$$
\end{example}

One main reason why $\BO$ is a good moduli space is the following result.

\begin{theorem}\label{thm-freeness}
The residue classes of the elements of~$\OO$ form
a $B_\OO$-basis of~$U_\OO$.
\end{theorem}
\begin{proof}
See~\cite[Theorem~3.4]{KR3}.
\end{proof}

\medskip
The scheme $U_\OO$ is a free, hence flat, $\BO$-module,
and the scheme $\BO$ parametrizes all $\OO$-border bases.
Its $K$-rational points correspond to $\OO$-border bases
in the following way.

\begin{definition}\label{def-represent}
Let $\Gamma=(\gamma_{ij})\in K^{\mu\nu}$
be a $K$-rational point of~$\BO$. 
Then the polynomials $g_j(X,\Gamma)$ with $j\in\{1,\dots,\nu\}$ form an $\OO$-border
basis. Let $I_\Gamma$ be the ideal in~$P$ which is generated
by these polynomials. 
Then the zero-dimensional scheme $\X_\Gamma$ in~$\AA^n_K$
defined by~$I_\Gamma$ is called the zero-dimensional scheme
{\bf represented by~$\Gamma$}.

Conversely, given a zero-dimensional scheme~$\X$ in~$\AA^n_K$
whose vanishing ideal $I_\X$ has an $\OO$-border basis,
the coefficients of that $\OO$-border basis define a
$K$-rational point $\Gamma_\X$ of~$\BO$. We say that the
point $\Gamma_\X$ {\bf represents} the zero-dimensional scheme~$\X$.

Let $\Gamma_0=(\underbrace{0, \dots, 0}_{\mu\nu}) \in K^{\mu\nu}$ be the point of~$\BO$ which represents 
the $0$-dimensional scheme in $\AA_K^n$ defined by the  ideal generated by $\partial\OO$.
Then $\Gamma_0$ is called the \textbf{origin of $\BO$}.
\end{definition}

\begin{example}\textbf{(Example~\ref{ex-box22} continued)}

\label{ex-box22 continued}%
In Example~\ref{ex-box22} we have $\mu=\nu=4$, hence there are 16 
indeterminates $c_{11},\dots, c_{44}$.
Let $\Gamma_0 =(0,0,0,0,0,0,0,0,0,0,0,0,0,0,0,0 )$. Then $f_k(\Gamma_0) = 0$ 
for $k =1,\dots, 12$,
hence $\Gamma_0$ is a $K$-rational point of $\BO$. Moreover, $g_1(X,\Gamma_0) = y^2$,
$g_2(X,\Gamma_0) = x^2$, $g_3(X,\Gamma_0) = xy^2$, $g_4(X,\Gamma_0) = x^2y$.
Therefore $I_{\Gamma_0} = \langle y^2,\, x^2, \, xy^2,\, x^2y \rangle =\langle y^2,\, x^2\rangle$
and $\X_{\Gamma_0} = \Spec(K[x,y]/I_{\Gamma_0})$
is the zero-dimensional scheme \textit{represented} by $\Gamma_0$.

Conversely,  denote $\Spec(K[x,y]/I)$ by $\X$ where $I =\langle x^2,\, y^2 \rangle$.
We notice that $I$ has the $\OO$-border basis  $\{y^2,\, x^2,\, xy^2,\, x^2y\}$.
The coefficients of that $\OO$-border basis define the $K$-rational 
point $\Gamma_\X =(0,0,0,0,0,0,0,0,0,0,0,0,0,0,0,0 )$ of $\BO$ which \textit{represents} $\X$.
\end{example}

\begin{example}\textbf{(Example~\ref{ex-numstabil} continued)}

\label{ex-numstabil continued}%
Recall that in Example~\ref{ex-numstabil} we claimed that 
The $\OO$-border basis of~$I^\epsilon$ is
\begin{equation*}
\left\{ 
\begin{aligned}
&y^2- (\tfrac{4}{5} -\tfrac{4}{5}\,\epsilon xy),  & &x^2- (\tfrac{4}{5} -\tfrac{4}{5}\,\epsilon xy), \\
&xy^2- (\tfrac{16\epsilon}{16\epsilon^2-25}\,y -\tfrac{20}{16\epsilon^2-25}\,x \big),& &
x^2y -( \tfrac{20}{16\epsilon^2-25}\,y -  \tfrac{16\epsilon}{16\epsilon^2-25}\,x \big)
\end{aligned} 
\right\}
\end{equation*}

Let $\Gamma_\epsilon = (\tfrac{4}{5}, 0,0,-\tfrac{4}{5}, \   \tfrac{4}{5}, 0,0,-\tfrac{4}{5},\
0,\tfrac{16\epsilon}{16\epsilon^2-25}, -\tfrac{20}{16\epsilon^2-25}, 0, \ 
0, \tfrac{20}{16\epsilon^2-25}, -\tfrac{16\epsilon}{16\epsilon^2-25}, 0 )$.
Then $f_k(\Gamma_\epsilon) = 0$ for $k =1,\dots, 12$,
hence $\Gamma_\epsilon$ is a $K$-rational point of $\BO$.
Moreover,
\begin{equation*}
\begin{aligned}
&g_1(X,\Gamma_\epsilon) = y^2- (\tfrac{4}{5} -\tfrac{4}{5}\,\epsilon xy),  
& &g_2(X,\Gamma_\epsilon) = x^2- (\tfrac{4}{5} -\tfrac{4}{5}\,\epsilon xy), \\
&g_3(X,\Gamma_\epsilon) = xy^2- (\tfrac{16\epsilon}{16\epsilon^2-25}\,y 
     -\tfrac{20}{16\epsilon^2-25}\,x),& &
g_4(X,\Gamma_\epsilon) =  x^2y -( \tfrac{20}{16\epsilon^2-25}\,y -  \tfrac{16\epsilon}{16\epsilon^2-25}\,x)
\end{aligned} 
\end{equation*}
Therefore $I_{\Gamma_\epsilon} = \langle g_1(X,\Gamma_\epsilon), \, g_2(X,\Gamma_\epsilon),\,
g_3(X,\Gamma_\epsilon),\,  g_4(X,\Gamma_\epsilon) \rangle$ and 
$\X_{\Gamma_\epsilon} = \Spec(K[x,y]/I_{\Gamma_\epsilon})$
is the zero-dimensional scheme \textit{represented} by $\Gamma_\epsilon$.

Next,  let $I =\ \langle g_1(X,\Gamma_\epsilon), \, g_2(X,\Gamma_\epsilon), \, 
g_3(X,\Gamma_\epsilon),\,  g_4(X,\Gamma_\epsilon) \rangle$, and 
let $\X = \Spec(K[x,y]/I)$.
It turns out that $I$ has the $\OO$-border basis  
$\{g_1(X,\Gamma_\epsilon),\,  g_2(X,\Gamma_\epsilon),\, g_3(X,\Gamma_\epsilon), \,
g_4(X,\Gamma_\epsilon)\}$.
The coefficients of that $\OO$-border basis define the $K$-rational 
point 
$$\Gamma_\X = (\tfrac{4}{5}, 0,0,-\tfrac{4}{5}, \   \tfrac{4}{5}, 0,0,-\tfrac{4}{5},\
0,\tfrac{16\epsilon}{16\epsilon^2-25}, -\tfrac{20}{16\epsilon^2-25}, 0, \ 
0, \tfrac{20}{16\epsilon^2-25}, -\tfrac{16\epsilon}{16\epsilon^2-25}, 0 )$$ 
of $\BO$ which \textit{represents} $\X$.
\end{example}

\bigskip
At the beginning of Subsection~\ref{subsec-Commuting Matrices} we 
wrote that in \cite[Subsection 6.4.B]{KR2}, further characterizations of border 
bases are given. Along the same lines, we recall an alternative way to 
describe border basis schemes.

\begin{definition}{\bf (Neighbor Generators)}\label{def-neighbors}

Let $\OO = \{t_1,\dots,t_\mu \}$ be an order ideal in~$\mathbb{T}^n$ with border
$\partial\OO = \{ b_1,\dots,b_\nu \}$.
Let $\mathcal{M}_1,\dots,\mathcal{M}_n$ be the
generic multiplication matrices, and 
let $c_j=(c_{1j},\dots,c_{\mu j})\tr$ be the $j$-th column
of~$(c_{ij})$, for $j=1,\dots,\nu$.
\begin{myenumerate}
\item Let $j, j'\in\{1,\dots,\nu\}$ be such that
$b_j=x_\ell b_{j'}$ for some $\ell\in\{1,\dots,n\}$. Then $b_j, b_{j'}$
are called {\bf next-door neighbors}, $(b_j, b_{j'})$ is called a \textbf{ND pair}),
and the tuple of polynomials in
$(c_j - \mathcal{M}_\ell c_{j'})\tr$ is denoted by $\ND(j, j')$.

\item The union of all entries of the tuples $\ND(j, j')$ is called the set of
{\bf next-door generators} of~$I(\BO)$ and is denoted by
$\ND_\OO$.

\item Let $j, j'\in\{1,\dots,\nu\}$ be such that $b_j = x_\ell t_m$ and
$b_{j'} = x_k t_m$ for some $m\in\{1,\dots,\mu\}$ and
 for some $k,\ell\in\{1,\dots, n\}$. Then $b_j, b_{j'}$
are called {\bf across-the-rim neighbors}, $(b_j, b_{j'})$ is called an \textbf{AR pair}), 
and the tuple of polynomials in 
$(\mathcal{M}_k c_j - \mathcal{M}_\ell c_{j'})\tr$
is denoted by $\AR(j, j')$.

\item The union of all entries of the tuples $\AR(j, j')$ is called the set
of {\bf across-the-rim generators} of~$I(\BO)$ and is denoted by $\AR_\OO$.

\item The polynomials in $\ND_\OO \cup \AR_\OO$ are called the 
{\bf neighbor generators} of~$I(\BO)$.

\item  Let $j, j'\in\{1,\dots,\nu\}$ be such that $x_kb_j = x_\ell b_{j'}$ 
 for some $k,\ell\in\{1,\dots, n\}$. Then $b_j, b_{j'}$
are called {\bf across-the-street neighbors}, $(b_j, b_{j'})$ is called an \textbf{AS pair}),
and the tuple of polynomials in $(\mathcal{M}_k c_j - \mathcal{M}_\ell c_{j'})\tr$
is denoted by $\AS(j, j')$.

\item The union of all entries of the tuples $\AS(j, j')$ is called the set
of {\bf across-the-street generators} of~$I(\BO)$ and is denoted by $\AS_\OO$.

\end{myenumerate}
\end{definition}

\begin{myremark}\label{rem-neighb=commut}
In~\cite[Proposition~4.1]{KR3}, it is shown that 
the neighbor generators are precisely the non-trivial 
entries of the commutators $\mathcal{M}_i\, \mathcal{M}_j - \mathcal{M}_j\, 
\mathcal{M}_i$. Consequently,  computing $\ND_\OO \cup \AR_\OO$
is an alternative method for computing the natural set of generators of~$I(\BO)$.
\end{myremark}

\begin{myremark}\label{rem-a-t-s useful for exposed}
Let ASP denote  the set of AS pairs and ARP the set of 
AR pairs.
It is clear from the definition that $\ARP \subseteq \ASP$. To show that the inclusion can be strict,
see Example~\ref{ex-exposedNeedAcrossTheCorner}, where $(b_3, b_4) \in \ASP\setminus \ARP$ 
(called across-the-corner pairs in~\cite[Section 4]{KR3}).
Moreover, in~\cite[Proposition~4.1]{KR3} it is shown that the set of polynomials in 
$\AS_\OO \setminus \AR_\OO$ is contained in the ideal generated 
by the neighbor generators of~$I(\BO)$.
\end{myremark}

We reuse Example~\ref{ex-box22} to illustrate  Remark~\ref{rem-neighb=commut}.

\begin{example}\label{ex-box22continued}
In Example~\ref{ex-box22} there are two next-door  and one across-the-rim neighbors.
The first next-door neighbors are $b_3, b_1$ since $b_3 = x b_1$ and 
$\ND(3,1)$ consists of the entries of  $c_3 - \mathcal{M}_x c_1$. They are 
\begin{align*}
&-c_{1,2}c_{3,1} -c_{1,4}c_{4,1} +c_{1,3},\\
 &-c_{2,2}c_{3,1} -c_{2,4}c_{4,1} +c_{2,3},\\
&-c_{3,1}c_{3,2} -c_{3,4}c_{4,1} -c_{1,1} +c_{3,3},\\
&-c_{3,1}c_{4,2} -c_{4,1}c_{4,4} -c_{2,1} +c_{4,3}.
\end{align*}
The second next-door neighbors are $b_4, b_2$ since $b_4 = y b_2$, and 
$\ND(4,2)$ consists of the entries of  $c_4 - \mathcal{M}_y c_2$. They are 
\begin{align*}
&-c_{1,1}c_{2,2} -c_{1,3}c_{4,2} +c_{1,4} \\
&-c_{2,1}c_{2,2} -c_{2,3}c_{4,2} -c_{1,2} +c_{2,4} \\
& -c_{2,2}c_{3,1} -c_{3,3}c_{4,2} +c_{3,4} \\
& -c_{2,2}c_{4,1} -c_{4,2}c_{4,3} -c_{3,2} +c_{4,4}.
\end{align*}
The across-the-rim neighbors are $b_3, b_4$ since $b_3 = y t_4$ and $b_4 = x t_4$, and 
$AR(3, 4)$ consists of the entries of $\mathcal{M}_x b_3 - \mathcal{M}_y b_4$.  They are
\begin{align*}
& -c_{1,1}c_{2,4} +c_{1,2}c_{3,3} +c_{1,4}c_{4,3} -c_{1,3}c_{4,4}, \\
& -c_{2,1}c_{2,4} +c_{2,2}c_{3,3} +c_{2,4}c_{4,3} -c_{2,3}c_{4,4} -c_{1,4}, \\
& -c_{2,4}c_{3,1} +c_{3,2}c_{3,3} +c_{3,4}c_{4,3} -c_{3,3}c_{4,4} +c_{1,3}, \\
& -c_{2,4}c_{4,1} +c_{3,3}c_{4,2} +c_{2,3} -c_{3,4}.
\end{align*}
Putting together the three sets we get exactly the same set computed in  Example~\ref{ex-box22},
i.e.\ the non-zero entries of $\mathcal{M}_x\mathcal{M}_y - \mathcal{M}_y\mathcal{M}_x$.
\end{example}

\medskip
\subsection{First Properties of Border Basis Schemes}
\label{subsec-First Properties of BBS}

To put the next results in the correct perspective, 
let us recall some fundamental facts about  punctual Hilbert schemes.
A punctual Hilbert scheme $\Hilb^\mu(\mathbb{A}^n_K)$ is a scheme which 
parametrizes all zero-dimensional subschemes of the affine 
space $\mathbb{A}^n_K$ having  multiplicity~$\mu$.

\begin{proposition}\label{prop-principal}
Let $\mu$ be a positive natural number.
\begin{myenumerate}
\item Let $\OO = \{t_1, \dots, t_\mu\}$  be an order ideal in~$\mathbb{T}^n$. 
Then the scheme $\BO$ is an open  subscheme of $\Hilb^\mu(\mathbb{A}^n_K)$.

\item Let $\{\OO_1, \dots, \OO_s\}$  be the set of all order ideals in~$\mathbb{T}^n$ 
with cardinality $\mu$.
Then $\{ \mathbb{B}_{\OO_1}, \dots, \mathbb{B}_{\OO_s}\}$ forms an open 
covering of $\Hilb^\mu(\mathbb{A}^n_K)$.
\end{myenumerate}
\end{proposition}

\begin{proof}
Proofs are contained in \cite{Hui1} and \cite{Hui3}. A detailed description 
of the gluing used to achieve the covering is contained in~\cite[Theorem~2.7 
and Remark~2.8]{KR4}. 
\end{proof}

\begin{definition}\label{def-cell}
In geometric jargon, when a scheme is isomorphic to an affine space, 
it is said to be an \textbf{affine cell}.
\end{definition}

\begin{definition}\label{def-planar}
Let  $\OO = \{t_1,\dots,  t_\mu\}$ be an order ideal  in $K[x, y]$. Then the 
corresponding $\OO$-border basis scheme $\BO$ is called \textbf{planar}.

Let $\OO$ be an order ideal in $K[x, y, z]$. Then the corresponding $\OO$- border basis 
scheme~$\BO$ is called  \textbf{spatial}.
\end{definition}

\begin{proposition}\label{prop-irreducible}
Let $\OO = \{t_1, \dots, t_\mu\}$  be an order ideal in~$\mathbb{T}^n$ and 
let $\mu$ be a positive natural number.
\begin{myenumerate}
\item Border basis schemes are irreducible for any $n$ and $\mu \le 7$.

\item Planar border basis schemes are smooth and irreducible for any $\mu$.

\item  Spatial border basis schemes are irreducible for $\mu \le 11$.
\end{myenumerate}
\end{proposition}
\begin{proof}
The first claim follows from~\cite{Maz} (see also  the introduction of~\cite{DJNT}).
The second claim follows from a classical result of Fogarty (see~\cite[Theorem~2.4]{Fog}).
For the third claim, it is proved in~\cite{DJNT}  that 
$\Hilb^{\mu}(\mathbb{A}^3_K)$ is irreducible for $\mu\le 11$,
hence the claim  follows from Proposition~\ref{prop-principal}\,(1).
\end{proof}

\begin{myremark}\label{rem-HilbConnected}
While a classical result of R. Hartshorne states that the Hilbert scheme 
$\Hilb^\mu(\mathbb{A}^n_K)$ is connected, it remains an open question whether the open 
subschemes $\mathbb{B}_\mathcal{O}$ share this property in general (see~\cite{Rob}). 
In the following sections, we provide a positive answer to this problem for a specific class of 
order ideals (see  Proposition~\ref{prop-homcommute}\,(5)).
\end{myremark}

Next, we introduce the notion of the principal component of $\BO$.
As noted in \cite[Section 18.2]{MS}, the principal component of the Hilbert scheme  
$\Hilb^\mu(\mathbb{A}^n_{\overline{K}})$
is irreducible by definition, as it is the closure of the locus of $\mu$ 
distinct points in $\mathbb{A}^n_{\overline{K}}$,
where $\overline{K}$ is the algebraic closure of $K$.
In the following, we let $\overline{P} = \overline{K}[x_1, \dots, x_n]$ and
let $\OO = \{t_1, \dots, t_\mu\}$  be an order ideal in~$\mathbb{T}^n$.
Finally, we denote $\BO \times_{\Spec(K)}\Spec(\overline{K})$ by  $\overline{\mathbb{B}}_\OO$.

\begin{definition}\label{def-PrincComp}
For each zero-dimensional scheme $\X$ in the affine space $\mathbb{A}^n_{\overline{K}}$ whose 
vanishing ideal in $\overline{P}$ has an $\OO$-border basis, let $\Gamma_\X$ 
be the $\overline{K}$-rational point in ${\overline{\mathbb{B}}_\OO}$ which represents $\X$
(see Definition~\ref{def-represent}),
and let $Y \subset \overline{\mathbb{B}}_\OO$ be the set of all points $\Gamma_\X$ 
such that $\X \subset \mathbb{A}^n_{\overline{K}}$ is a reduced scheme of length $\mu$. 
The \textbf{principal component} of $\BO$, denoted by $\mathbb{C}_\OO$, 
is the closed subscheme of $\BO$ such that $\mathbb{C}_\OO \times_{\Spec(K)} \Spec(\overline{K})$ 
is the closure of $Y$ in $\overline{\mathbb{B}}_\OO$, endowed with its reduced 
subscheme structure.
\end{definition}

\begin{myremark}\label{rem-irreduc-}
Since $Y$ is an open subset of the principal component of $\Hilb^\mu(\mathbb{A}^n_{\overline{K}})$, 
which is irreducible, its closure is irreducible. 
Consequently, $\mathbb{C}_\OO \times_{\Spec(K)} \Spec(\overline{K})$ is irreducible, 
in other words $\mathbb{C}_\OO$ is geometrically irreducible.
\end{myremark}

\begin{myremark}\label{rem-nmu}
The dimension of the principal component $\mathbb{C}_\OO$ is $n\mu$, as this is 
the number of parameters required to describe $\mu$ distinct points in $\mathbb{A}^n_{\overline{K}}$. 
An explicit description of this component, particularly in terms of the equations 
defining it within the border basis scheme, can be found in~\cite[Section 3]{KR4}.
\end{myremark}

\medskip
\subsection{Arrow Gradings}
\label{subsec-Arrow Gradings}

An important property of $I(\BO)$ is its homogeneity with respect to the following gradings.

\begin{definition}\label{def-arrdeg}
Let $\OO \!=\! \{t_1,\dots,t_\mu\}$ be an order ideal in~$\mathbb{T}^n$, let
$\partial\OO\!=\! \{b_1,\dots, b_\nu\}$ be its border, and let
$C=\{c_{11}, c_{12}, \dots, c_{\mu\nu}\}$ be the set of indeterminates 
representing the coefficients of the generic $\OO$-border prebasis.
\begin{myenumerate}
\item The $\mathbb{Z}^n$-grading on~$K[C]$ defined by 
$\deg_U(c_{ij}) = \log(b_j) - \log(t_i)$ for $i=1,\dots,\mu$
and $j=1,\dots,\nu$, is called the {\bf arrow grading}.

\item The $\mathbb{Z}$-grading on~$K[C]$ defined by
$\deg_W(c_{ij}) = \deg(b_j) - \deg(t_i)$ for $i=1,\dots,\mu$
and $j=1,\dots,\nu$ is called the {\bf total arrow grading} and denoted by $\deg_W$.

\item The $\mathbb{Z}$-grading on~$K[C,X]$ defined by
the total arrow grading on $K[C]$ and the standard grading on $K[X]$ is 
called the \textbf{extended total arrow grading} and 
denoted by $\deg_{\overline{W}}$.
\end{myenumerate}
\end{definition}

The terminology stems from the fact that each $c_{ij}$ can be viewed 
as an \textbf{arrow} pointing from $b_j$ to $t_i$, as explained in 
\cite[p.~210]{Hai} and \cite[Section~3]{Hui2}.

The following proposition shows why the total arrow degree is particularly useful.

\begin{proposition}\label{prop-arrowhomog}
Let $\OO$ be an order ideal, let $\deg_W$ be the total arrow grading on $K[C]$, 
let $\deg_{\overline{W}}$ be the extended total arrow grading on $K[C,X]$,
let $\mathsf{G}$ be the generic $\OO$-border prebasis,
let $\BO$ be the $\OO$-border basis scheme,  and let 
$U_\OO =K[C,X]/\langle I(\BO), \mathsf{G}\rangle$ be the universal $\OO$-border basis algebra.
\begin{myenumerate}
\item The ideal $I(\BO)$ is $W$-homogeneous. 

\item The ideal $\langle \textsf{G}\, \rangle$ is $\overline{W}$-homogeneous.

\item The ring $U_\OO$ is  $\overline{W}$-graded.

\end{myenumerate}
\end{proposition}

\begin{proof}
To prove the first claim, we consider $U_\OO$ as a free $B_\OO$-module of rank $\mu$
with basis~$\OO$ (see~Theorem~\ref{thm-freeness}). 
For $i=1,\dots,n$, the generic multiplication
matrix~$\mathcal{M}_i$ represents multiplication by~$x_i$
in this basis. Consequently it defines a homogeneous $B_\OO$-linear map
\begin{equation*}
\mu_{x_i} \colon \bigoplus_{j=1}^\mu B_\OO({-}\deg(t_j){-}1) \longrightarrow
\bigoplus_{j=1}^\mu B_\OO(-\deg(t_j))
\end{equation*}
of $\overline{W}$-degree zero.
It follows that the entries of the commutators 
$\mathcal{M}_k \mathcal{M}_\ell - \mathcal{M}_\ell \mathcal{M}_k$ are $\overline{W}$-homogeneous.
Since these entries belong to $K[C]$, they are $W$-homogeneous, which proves that $I(\BO)$ is $W$-homogeneous.

The second claim follows directly from the definition of $\overline{W}$ 
and the construction of the generic $\OO$-border prebasis.  
The third claim follows  from the homogeneity 
of the ideals~$I(\BO)$ and~$\langle \mathsf{G} \rangle$, established above.
\end{proof}

\begin{myremark}\label{rem-arrowGrading}
In~\cite[Lemma~3.4]{KSL}, it is shown that~$I(\BO)$ is even homogeneous
with respect to the arrow grading. 
\end{myremark}

 We now verify Proposition~\ref{prop-arrowhomog} using the data
from Example~\ref{ex-box22}.

\begin{example}\textbf{(Example~\ref{ex-box22} continued)}\\
\label{ex-box22-continued}
In Example~\ref{ex-box22} the  total arrow degrees of the parameters $c_{ij}$  are
$$
(\deg_W(c_{11}), \deg_W(c_{12}), \dots, \deg_W(c_{44})) \;=\; 
(2,2,3,3,\ 1,1,2,2,\ 1,1,2,2,\ 0,0,1,1).
$$
As predicted, the generators $f_1, \dots, f_{12}$ of $I(\mathbb{B}_\OO)$ 
are homogeneous, with degrees
\begin{equation*}
\begin{aligned}
&\deg_W(f_1) = 3, \deg_W(f_2) = 2, \deg_W(f_3) = 2, \deg_W(f_4) = 1, \deg_W(f_5) = 3, \deg_W(f_6) = 2, \\
&\deg_W(f_7) = 2, \deg_W(f_8) = 1, \deg_W(f_9) = 4, \deg_W(f_{10}) = 3, \deg_W(f_{11}) = 3, \deg_W(f_{12}) = 2.
\end{aligned}
\end{equation*}
\end{example}

\newpage   
\section{Re-embedding}
\label{sec-Re-embedding}

Remember that rings of the form $P/I$, where $P$ is a polynomial ring over a field $K$
and~$I$ is an ideal in $P$, are called \textbf{affine $K$-algebras}, or simply 
\textbf{affine algebras} if $K$ is clear from the context.

In this section, we introduce a technique that allows us to construct an isomorphism 
between an affine algebra and another affine algebra with a lower embedding dimension.

Let $K$ be a field, let $P = K[x_1,\dots,x_n]$, let $X =(x_1,\dots, x_n)$ and
let $\M =\langle X \rangle$.
In the following, we assume that~$I$ is an ideal in~$P$ such that $I \subseteq\M$.
We let  $\m = \M/I$ be the corresponding maximal ideal of $P/I$, so that
$I\subseteq \M$ implies $(P/I)/\m \cong P/\M \cong K$.

\medskip
\subsection{Z-Separating Tuples and Term Orderings}
\label{subsec-Z-Separating Tuples and Term Orderings}

The following definition plays a central role in this subsection.

\begin{definition}{\bf (Separating and Coherently Separating Tuples)}\label{def-sepindets}

Let $Z=(z_1,\dots,z_t)$ be a tuple of distinct indeterminates in $X=(x_1,\dots,x_n)$, 
and let $I$ be an ideal in $P$.
\begin{myenumerate}
\item Let $i\in \{1,\dots, t\}$. A polynomial $f\in P$ is called {\bf $z_i$-separating}
if it is of the form $f=z_i - h$ where $h \in P$ is such that $z_i$ divides
no term in $\Supp(h)$.

\item A tuple of polynomials $F=(f_1,\dots, f_t) \in P^t$ is called
{\bf $Z$-separating} if there exists a term ordering~$\sigma$ such that
$\LT_\sigma(f_i)=z_i$ for $i=1,\dots, t$, hence  if $F$ is a minimal 
$\sigma$-Gr\"obner basis of $\langle F \rangle$. If $F\subset I$,
the term ordering $\sigma$ is called a \textbf{$Z$-separating term ordering for~$I$}.

\item A tuple of polynomials $F=(f_1,\dots, f_t) \in P^t$ is called
{\bf coherently $Z$-separating} if there exists a term ordering~$\sigma$ such that
each $f_i$ is of the form $f_i = z_i - h_i$,
where $h_i \in P$ and $z_i$ divides no term in $\Supp(h_j)$ for all $i,j \in \{1,\dots, t\}$,
hence  if $F$ is the reduced $\sigma$-Gr\"obner basis of $\langle F \rangle$.

\item The tuple~$Z$ is called a {\bf separating tuple of indeterminates} 
for~$I$ if there exists a tuple of polynomials in $I$ which is $Z$-separating.

\end{myenumerate} 
\end{definition}

\begin{myremark}\label{rem-SepCohSep}
Passing from a minimal Gr\"obner basis to the reduced Gr\"obner basis 
is a standard operation. Hence,
from items (2) and (3) of the above definition, it follows that the existence of a 
$Z$-separating tuple of polynomials in $I$ is equivalent to the existence of a 
coherently $Z$-separating tuple of polynomials in $I$.
\end{myremark}

The following proposition highlights a useful property of $Z$-separating
tuples and term orderings.
 
\begin{proposition}\label{prop-sepandcoherentlysep}
Let~$F$ be a $Z$-separating tuple of polynomials in $I$,
and let $\sigma$ be a corresponding $Z$-separating term ordering for~$I$. 
Then every elimination term ordering for~$Z$ is a $Z$-separating term ordering for~$I$.
\end{proposition}

\begin{proof}
By applying interreduction to~$F$, we may assume that~$F$ is the reduced  
$\sigma$-Gr\"obner basis of $\langle F \rangle$. 
Hence, the reduced $\sigma$-Gr\"obner basis of~$I$ has the 
shape $(F,G)$ where $G$ is a tuple of polynomials not involving any $z_i \in Z$ in their supports.
On the other hand, since~$F$ consists of $Z$-separating polynomials, 
every elimination term ordering for $Z$ yields a reduced Gr\"obner 
basis of the shape $(F,G')$, hence it is a $Z$-separating term ordering for~$I$.
\end{proof}

The following example illustrates the above proposition.

\begin{example}\label{ex-sepandcohsep}
Let $P = \QQ[x,y,z]$ and let $F=(f_1,f_2)$, where $f_1 = x - z^2$ and $f_2 = y - x^2 - z^5$.
Then $F$ is $(x,y)$-separating, since with respect to the term ordering $\sigma$ defined by 
a matrix whose first row is $(3,7,1)$, we have $\LT_\sigma(f_1) = x$ and $\LT_\sigma(f_2) = y$.
Moreover, the reduced $\sigma$-Gr\"obner basis of $\langle F \rangle$ is $\{f_1, f_2' \}$, 
where $f_2' = y - z^4 - z^5$, and $(f_1, f_2')$ is coherently $(x,y)$-separating.
\end{example}

The computational strength of coherently $Z$-separating tuples lies in their
ability to eliminate the variables in $Z$ as follows.

\begin{definition}\label{def-rewriting}
Let $Z=(z_1,\dots, z_t)$ be a tuple of distinct indeterminates in~$X$, and let
$F=(f_1,\dots, f_t)$ be a coherently $Z$-separating tuple of polynomials in~$I$,
where $f_i = z_i - h_i$ for $i=1,\dots, t$.
\begin{myenumerate}
\item  For every polynomial $g\in P$ and every $i\in \{1,\dots, t\}$,
we replace each occurrence of~$z_i$ in~$g$ by~$h_i$.
The resulting polynomial $\hat{g}$ is said to be obtained by \textbf{rewriting~$g$ using~$F$}.

\item  For a tuple of polynomials $G=(g_1,\dots,g_r)$ in~$P$, we let
$\widehat{G} = (\hat{g}_1,\dots,\hat{g}_r)$ and call it the tuple obtained by
\textbf{rewriting~$G$ using~$F$}.
\end{myenumerate}
\end{definition}

\begin{proposition}{\bf (Computing Elimination By Substitution)}
\label{prop-ElimViaSubst}

Let $K$ be a field, let $P=K[x_1,\dots,x_n]$. Let $I=\langle G\rangle$ be an ideal in~$P$
generated by a tuple of polynomials $G=( g_1,\dots,g_r)$, let $Z=(z_1,\dots,z_t)$
be a tuple of distinct indeterminates in~$X =(x_1,\dots, x_n)$, and let $F=(f_1,\dots, f_t)$ be a coherently
$Z$-separating tuple of polynomials in~$I$.
Then the tuple $\widehat{G} =(\hat{g}_1, \dots, \hat{g}_r)$ obtained by 
rewriting~$G$ using~$F$ is a system of generators (not necessarily a Gr\"obner basis)
of the elimination ideal $I\cap K[X{\setminus}Z]$.
\end{proposition}

\begin{proof}
For each $i=1, \dots, t$, let $f_i = z_i - h_i$ where $h_i \in P$. Since $F$ is coherently 
$Z$-separating, the polynomials $h_i$ involve no variables from $Z$, hence $h_i \in K[X{\setminus}Z]$.
Since $z_i-h_i \in I$ for all $i$, the result of applying the substitution
$z_i\mapsto h_i$ to an element $g \in I$ yields a polynomial still in~$I$. 
Because the result  $\hat{g}$  of this substitution no longer 
contains any indeterminates from $Z$, we have $\hat{g} \in I\cap K[X{\setminus}Z]$.
Applying this to  $G$ shows that 
$\widehat{G} = \{ \hat{g}_1, \dots, \hat{g}_r \} \subseteq I \cap K[X{\setminus}Z]$.
Conversely, let $h \in I \cap K[X{\setminus}Z]$ and consider the substitution 
homomorphism $\sigma \colon P \to P$ defined by:
\begin{equation*}
\sigma(x_j) = 
\begin{cases} 
h_i & \text{if } x_j = z_i \in Z, \\ 
x_j & \text{if } x_j \in X{\setminus}Z. 
\end{cases} 
\end{equation*}
For any polynomial $q \in P$, let $\hat{q}$ denote its image $\sigma(q)$.
Since $h \in I$,  there exist polynomials  $p_1,\dots, p_r \in P$ 
such that $h=p_1 g_1 +\cdots + p_r g_r$.
Observing that $h$ does not involve any indeterminates from $Z$, we  obtain
$
h \;=\; \hat{h} \;=\; \hat{p}_1\, \hat{g}_1 + \cdots + \hat{p}_r\, \hat{g}_r
$.
This shows that $h$ is contained in the ideal generated by $\widehat{G}$ in $K[X{\setminus}Z]$, 
concluding the proof.
\end{proof}

\medskip
\subsection{Cotangent Spaces and Gr\"obner Fans}
\label{subsec-Cotangent Spaces and GFans}

So far, we have discussed several useful properties of $Z$-separating tuples and
the associated $Z$-separating term orderings. We now turn to the problem of
checking and, more generally, detecting them.

The first fundamental object required for our discussion is 
 the cotangent space of~$P/I$ at~$\m$.  If we deal with a  linear maximal ideal
 different from $\M =\langle X \rangle$, a simple translation of coordinates can be used 
 to modify the following definition.

\begin{definition}\label{def-cotangent}
Let $P$, $\M$ and  $I$  be as defined in the introduction to the current  section, 
with $\m = \M/I$ and, for simplicity, let $R = P/I$.
\begin{myenumerate}
\item The $K$-vector space $\Cot_\m(R) = \Cot_\m(R_\m)= \m/\m^2$ is called 
the Zariski cotangent space, or simply  the \textbf{cotangent space} of the ring~$R$ at $\m$.

\item Given $f \in P$, let $\ell(x_1,\dots, x_n) \in \M_1$ be the homogeneous component 
of standard degree one of $f$.
Then the polynomial $\Lin(f)=   \ell(x_1,\dots, x_n)$ is called the \textbf{linear part} of~$f$.

\item The $K$-vector space $\Lin(I) = \langle \Lin(f) \mid f\in I \rangle_K$ is called
the \textbf{linear part} of~$I$.
\end{myenumerate}
\end{definition}

\begin{myremark}\label{rem-Lin}
In the above definition and in the following text we use the symbol $\Lin$, in accordance 
with~\cite{KR5}, instead of ${\rm Lin}_\M$, the notation used, for instance, in~\cite{KLR0, KLR1}.
There are two reasons for this choice. First, $\M$ is fixed, so there is no need to specify it.
Second, in Section~\ref{sec-Positive P0-algebras} we use the symbol $\Lin_A$ 
to highlight an analogy, while at the same time emphasizing an important distinction.
\end{myremark}

\begin{myremark}\label{rem-tangent}
We recall that the dual space $\Hom_K(\m/\m^2, K)$ of $\Cot_\m(R)$  
is called the Zariski tangent space of $R$ at $\m$.
\end{myremark}

The cotangent space of $R$ at $\m$ can be described as follows. 
This description is particularly useful for practical computations.

\begin{proposition}\label{prop-cotancomp}
In the setting described above, let $\M=\langle X\rangle$, and 
let $I = \langle f_1, \dots, f_s \rangle$ be an ideal in~$P$ which is contained in~$\M$.
\begin{myenumerate}
\item We have $\Lin(I) = \M_1 \cap (I+\M^2)$.

\item We have $\Lin(I) = \langle \Lin(f_1), \dots, \Lin(f_s) \rangle_K$.

\item There is an isomorphism of $K$-vector spaces
\begin{equation*}
\Cot_\m(R) \cong \M_1 / \langle \Lin(f_1), \dots, \Lin(f_s) \rangle_K.
\end{equation*}

\item We have $\dim(\Cot_\m(R)) = n - \dim_K(\Lin(I))$.
\end{myenumerate}
\end{proposition}

\begin{proof} 
To prove (1), we start with the containment ``$\subseteq$''.
An element $f\in I$ can be written uniquely in the form $f = \ell + h$ with
$\ell \in \M_1$ and $h\in \M ^2$.
Then $\ell \in \M_1$ and 
$\ell = f - h \in I + \M^2$ imply the claim.  
Conversely, let $\ell \in \M_1$ be of the form $\ell=f+g$ with $f\in I$ and $g\in \M^2$. 
Then $f=\ell-g$, hence $\ell=\Lin(f)$ and $\ell \in \Lin(I)$.

To prove~(2), for each $i=1,\dots,s$, we write $f_i = \ell_i + g_i$ 
with $\ell_i = \Lin(f_i)$ and $g_i \in \M^2$. Since any element in $I$ is a polynomial 
combination of the $f_i$, reducing modulo $\M^2$ implies that 
$I + \M^2 = \langle \ell_1, \dots, \ell_s \rangle_K + \M^2$. 
Using~(1), we obtain 
\begin{equation*}
\Lin(I) = \M_1 \cap (I+\M^2) = \M_1 \cap (\langle \ell_1,\dots,\ell_s\rangle_K+ \M^2).
\end{equation*} 
By the modular law for vector spaces, since $\langle \ell_1,\dots,\ell_s\rangle_K \subseteq \M_1$, 
and using the equality $\M^2 \cap \M_1 = \{0\}$, we get the equality 
$\Lin(I) = \langle \ell_1,\dots,\ell_s \rangle_K$.

To prove~(3), we use the following chain of $K$-vector space isomorphisms.
\begin{align*}
\Cot_\m(R) &= \m/\m^2 \cong \M/(I + \M^2) \cong (\M_1 + I + \M^2)/(I+\M^2) \\
&\cong \M_1/(\M_1 \cap (I + \M^2)).
\end{align*}
The conclusion follows from~(1) and~(2).

To prove~(4), we note that $\dim_K(\M_1) = n$, hence the conclusion follows from~(3).
\end{proof}

\begin{example}\label{ex-cotanglinear}
Let $P=\QQ[x,y,z]$, let $\M=\langle x,y,z\rangle$, and let $I=\langle f_1,f_2,f_3\rangle$,
where $f_1= x^3 -x -3z$, $f_2=y^2 +2x$, and $f_3 = xy + 5z$. 
Then we get 
$$\Lin(I) = \langle \Lin(f_1),\; \Lin(f_2),\; \Lin(f_3) \rangle_\QQ=
\langle -x-3z,\, 2x,\, 5z \rangle_\QQ = \langle x,\, z\rangle_\QQ$$
Hence, we get $\Cot_\m(P/I)\cong  
\langle x, y, z\rangle_\QQ / \langle x, z \rangle_\QQ \cong \QQ\cdot \bar{y}$.
\end{example}

Given a tuple of indeterminates $Z$, suppose we wish to determine whether 
a tuple of polynomials $(f_1,\dots, f_t)$ with $f_i \in I$ is $Z$-separating. 
One possible approach is to use the methods described in~\cite[Section~3]{KLR1}, 
which rely on feasibility solvers for linear programming problems. 
We now present an alternative approach.

\begin{definition}\label{def-coeffmat}
Let $P= K[x_1, \dots, x_n]$, and let $\M =\langle x_1,\dots, x_n\rangle$.
\begin{myenumerate}
\item Let $L=(\ell_1, \dots, \ell_r)$ be a tuple of linear forms in $P$. For $i=1,\dots,r$, 
we write $\ell_i = c_{i1} x_1 + \cdots + c_{in} x_n$ with $c_{ij} \in K$. 
Then the matrix $(c_{ij}) \in \Mat_{r, n}(K)$ 
is called the \textbf{coefficient matrix} of $L$.

\item Let $G {=} (g_1,\dots, g_r)$ be a tuple of polynomials in~$\M$, and 
$L=(\Lin(g_1),\dots, \Lin(g_r))$. Then the coefficient matrix of $L$ is denoted by $\cmat(G)$.

\item  Let $Z=(z_1,\dots, z_t)$ be a tuple of distinct indeterminates. Then the 
submatrix of~$\cmat(G)$  formed by the columns corresponding to the indeterminates in~$Z$
is denoted by $\cmat_Z(G)$.
\end{myenumerate}
\end{definition}

The next proposition explains why we are interested in the condition 
$Z \subseteq \LT_\sigma(\langle \Lin(I) \rangle)$.
Indeed, separating tuples~$Z$ for~$I$ whose entries lie in $\Lin(I)$ satisfy 
the following properties.

\begin{proposition}\label{prop-sepZ}
Let $P=K[x_1, \dots, x_n]$,  $\M =\langle x_1,\dots, x_n\rangle$,
 $I=\langle g_1, \dots, g_r \rangle \subset \M$, 
and let $Z=(z_1,\dots, z_t)$ be a tuple of distinct indeterminates.
Assume that $I$ is $Z$-separating.

\begin{myenumerate}
\item We have  $Z \subseteq \LT_\sigma( \langle \Lin(I) \rangle)$ 
for every elimination term ordering $\sigma$ for $Z$.

\item The matrix $\cmat_Z(G)$ has rank $t$. 
In particular, we have $t \le \dim_K(\Lin(I))$.
\end{myenumerate}
\end{proposition}

\begin{proof} 
We first prove (1).
By assumption, there exists a $Z$-separating term ordering $\sigma$ for~$I$. 
By Proposition~\ref{prop-sepandcoherentlysep}, 
it follows that every elimination term ordering $\tau$ for $Z$ is also a
$Z$-separating term ordering for $I$. Hence, by Definition~\ref{def-sepindets}\,(2), we have
$Z \subseteq \LT_\tau(I)$. Let $z \in Z$. Then there exists $f \in I$ such that $z = \LT_\tau(f)$.
Thus $z = \LT_\tau(\Lin(f))$, and therefore 
$z \in \LT_\tau(\langle \Lin(I) \rangle)$.
This proves (1).

To prove (2), let $F=(f_1,\dots,f_t)$ be a $Z$-separating tuple in $I$ and
let $\sigma$ be a term ordering such that $\LT_\sigma(f_i)=z_i$ for $i=1,\dots, t$.
Without loss of generality, we may assume that $z_1 >_\sigma \cdots >_\sigma z_t$.
Then  $\cmat_Z(F)$ is an upper triangular matrix in $\GL_t(K)$. 
By Proposition~\ref{prop-cotancomp}\,(2),  
we know that $\Lin(f_i) \in \langle\Lin(g_1), \dots, \Lin(g_r)\rangle_K$ for $i=1,\dots, t$. 
Hence there exists a matrix $B \in \Mat_{ t, r}(K)$ such that $\cmat(F) = B \cdot \cmat(G)$, and therefore
$\cmat_Z(F) = B \cdot \cmat_Z(G)$. Since $\cmat_Z(F) $ is invertible, 
it follows that $\cmat_Z(G)$ has rank $t$.
\end{proof}

This proposition suggests the following definition.

\begin{definition}\label{def-toprank}
Let $I \subseteq P$ be an ideal contained in $\M$, let $Z$ be a tuple of distinct indeterminates, 
and let $G=(g_1,\dots, g_r) \subset I$.
If $\rk(\cmat_Z(G)) = \#Z$, we say that $Z$ is a tuple of \textbf{top rank} for $G$.
\end{definition}

\smallskip
Moreover, Proposition~\ref{prop-sepZ} brings us back to the $K$-vector space $\Lin(I)$ already 
analyzed in Proposition~\ref{prop-cotancomp}. The  idea  is to explore all  possible 
monomial ideals $\LT_\sigma(\langle \Lin(I) \rangle)$, hence the Gr\"obner 
fan  of $\Lin(I)$. 

The following definition recalls the notion of the Gr\"obner fan of an ideal.

\begin{definition}\label{def-GFan}
Let $I$ be an ideal in $P$.
\begin{myenumerate}
\item A tuple of pairs
$\overline{G} = ((g_1, \LT_\sigma(g_1)), \dots, (g_r, \LT_\sigma(g_r)))$,
 where $G=(g_1,\dots, g_r)$ is the reduced $\sigma$-Gr\"obner
basis of~$I$ for a term ordering $\sigma$, is called a \textbf{marked Gr\"obner basis} of~$I$. 

\item Let $\overline{G} = ((g_1, \LT_\sigma(g_1)), \dots, (g_r, \LT_\sigma(g_r)))$ be 
a marked Gr\"obner basis of~$I$.
The set $\{ z\in X \mid \LT_\sigma(g_i) = z \text{ for some } i\in\{1,\dots,r\} \}$ is called 
the \textbf{set of leading indeterminates} of~${\overline{G}}^{\mathstrut}$ and denoted by 
$\LI(\overline{G})$.

\item The set of all marked reduced Gr\"obner bases of~$I$ is called the
\textbf{Gr\"obner fan} of~$I$ and denoted by ${\rm GFan}(I)$.

\item The set of all leading term sets of marked
reduced Gr\"obner bases in $\GFan(I)$ is called the \textbf{leading term Gr\"obner fan}
of~$I$ and denoted by $\LTGFan(I)$.

\end{myenumerate}
\end{definition}

\begin{myremark}\label{rem-GFan}
Recalling the original paper where the Gr\"obner fan was introduced (see \cite{MR}), 
we note that the Gr\"obner fan is actually a polyhedral complex. Since all term 
orderings associated with a given reduced Gr\"obner basis are described by the vectors 
in a polyhedral cone, with a slight abuse of terminology, we adopt the 
simplified definition given above.
\end{myremark}

\begin{myremark}\label{rem-LI}
It follows from the definition that $\#(\LI(\overline{G})) \le \dim_K(\Lin(I))$,
since the leading indeterminates of a reduced Gr\"obner basis of $\langle \Lin(I) \rangle$ 
form a basis of the monomial ideal $\LT_\sigma(\langle \Lin(I) \rangle)_1$, 
whose $K$-dimension equals $\dim_K(\Lin(I))$.
In the terminology of Definition~\ref{def-marked}\,(3), the tuple $\overline{G}$ 
is coherently marked. 
\end{myremark}

\begin{myremark}\label{rem-GFanLin}
Proposition~\ref{prop-sepZ} motivates the idea of using $\Lin(I)$, in particular 
the Gr\"obner fan of $\Lin(I)$, which is much easier to compute than the Gr\"obner fan of $I$,
to find tuples $Z$ for which the ideal $I$ is $Z$-separating. This idea is explained in  great detail 
in \cite[Section 3]{KLR2}. 
The following example shows how it works. (See also Example~\ref{ex-PropTrivBas}).
\end{myremark}

\begin{example}\label{ex-LinearGFan}
Let $P = \QQ[x, y, z, w]$,  $\M= \langle x, y, z, w \rangle$,
 $f_1= x +y -z +4w -x^2$, $f_2= x -y -z +xy$,
$F=(f_1,f_2)$, and $I = \langle F \rangle$.  We have 
$\ell_1 = \Lin(f_1) = x +y -z +4w$ and $\ell_2 = \Lin(f_2) = x -y -z$, and 
\begin{equation*}
\cmat(F) = \begin{bmatrix}
1 & 1  & -1 & 4 \\  
1 & -1  & -1 & 0
\end{bmatrix}.
\end{equation*}
The $2\times2$-submatrix $C_{13}$ is singular. The others are 
\[
C_{12} = \begin{bmatrix} 1 & \;\; 1 \\  1 & -1  \end{bmatrix},\,
C_{14} = \begin{bmatrix} 1 & 4 \\  1 & 0  \end{bmatrix},\,
C_{23} = \begin{bmatrix}  1 & -1 \\ -1 & -1  \end{bmatrix},\,
C_{24} = \begin{bmatrix} 1 &  4 \\  -1 & 0  \end{bmatrix},\,
C_{34} = \begin{bmatrix} -1 & 4 \\ -1 & 0 \end{bmatrix}.
\]
This shows that $(x, z)$ is not of top rank, while $(x, y)$, $(x, w)$, $(y, z)$, $(y, w)$, and $(z, w)$ are 
of top rank.
Multiplying  their inverses by $\cmat(F)$, we obtain the matrices
\[
\begin{bmatrix} 1 & 0 & -1 & 2   \\ 0 & 1  & 0 & 2 \end{bmatrix},\ 
\begin{bmatrix} 1 & -1 & -1 & 0 \\  0 & \tfrac{1}{2}  & 0  & 1 \end{bmatrix},\ 
\begin{bmatrix} 0 & 1  & 0  & 2\\  -1 & 0  & 1 & -2 \end{bmatrix},\ 
\begin{bmatrix} -1 & 1 & 1 & 0 \\ \tfrac{1}{2}  &  0 & -\tfrac{1}{2}  & 1 \end{bmatrix},\ 
\begin{bmatrix} -1 & 1  & 1 & 0 \\ 0  & \tfrac{1}{2} & 0 & 1 \end{bmatrix}.
\]
They correspond to  the following marked reduced Gr\"obner bases of~$\langle \ell_1, \ell_2 \rangle$. 
\begin{align*}
\big((\ x-z +2w, \ x),\  (y +2w, \ y) \big), & \quad \big( (x -y - z, \ x),\  (w + \tfrac{1}{2}y, \ w) \big) \\
\big((y+2w, \ y), \ (z-x-2w, \ z) \big), & \quad \big( (y-x+z, \ y),\  (w +\tfrac{1}{2}x -\tfrac{1}{2}z,\ w) \big), \\
\big((z-x+y,\ z), \ (w +\tfrac{1}{2}y, \ w)\big) 
\end{align*}

Next, we multiply the inverses of $C_{12},\dots, C_{34}$ by 
$\begin{bmatrix} f_1 \\ f_2 \end{bmatrix}$, obtaining five different pairs of generators of $I$.
They are 
\begin{align*}
G_1 &= ( -\tfrac{1}{2}x^2 +\tfrac{1}{2}xy +x -z +2w,\  -\tfrac{1}{2}x^2 -\tfrac{1}{2}xy +y +2w)\\
G_2 &= ( xy +x -y -z,\  -\tfrac{1}{4}x^2  -\tfrac{1}{4}xy +\tfrac{1}{2}y +w) \\
G_3 &= (-\tfrac{1}{2}x^2  - \tfrac{1}{2}xy +y +2w, \  \tfrac{1}{2}x^2  - \tfrac{1}{2}xy -x +z -2w) \\
G_4 &= (-xy -x +y +z,\  -\tfrac{1}{4}x^2 +\tfrac{1}{4}xy +\tfrac{1}{2}x  -\tfrac{1}{2}z +w) \\
G_5 &=  (-xy -x +y +z, \ -\tfrac{1}{4}x^2  -\tfrac{1}{4}xy +\tfrac{1}{2}y +w)
\end{align*}
Note that $x$ cannot be the leading term of the first polynomial in $G_1$, since it cannot 
be bigger than $x^2$. Using similar arguments, we discard all but $G_5$.
A direct check using an elimination term ordering for $(z, w)$ shows that the reduced Gr\"obner basis 
of $I$ is indeed~$G_5$ written in the correct order as
$ (z -xy -x +y,  \ \ w -\tfrac{1}{4}x^2 -\tfrac{1}{4}xy +\tfrac{1}{2}y) $.
An important use of this information is illustrated in Example~\ref{ex-LinGFanContinued}.
\end{example}

\medskip
\subsection{Z-Separating Re-embeddings and Optimal Re-embeddings}
\label{subsec-Z-Separating Re-embeddings and Optimal Re-embeddings}

The coherently $Z$-separating tuples (see Definition~\ref{def-sepindets})
allow us to re-embed ideals in the following sense.

\begin{proposition}\label{prop-isoZseparating}
Let $I$ be an ideal in $P$, let $Z=(z_1,\dots, z_t)$ be a tuple of distinct indeterminates in $X$, 
and let $F=(f_1,\dots, f_t)$ be a coherently $Z$-separating tuple in~$I$.
For $i=1,\dots, t$, write $f_i = z_i - h_i$.
Then, letting $\widehat{P} = K[X{\setminus}Z]$, the $K$-algebra homomorphism
$\phi \colon P \To \widehat{P}$ defined by $\phi(z_i) = h_i$ for $i=1,\dots, t$ 
and $\phi(x_j) = x_j$ for $x_j \in X{\setminus}Z$, induces a $K$-algebra isomorphism
$\Phi \colon P/I \longrightarrow \widehat{P} / (I\cap \widehat{P})$.
\end{proposition}

\begin{proof} 
By definition, $\phi$ and hence $\Phi$ are surjective.
Using the division algorithm, we know that every polynomial $f \in P$ 
can be written as $f = \sum_{j=1}^t q_j(z_j - h_j) + \tilde{f}$ with $\tilde{f} \in \widehat{P}$.
We have $f \in I$ if and only if $\tilde{f} \in I$. Since $\phi(f) = \tilde{f}$, we get
$f \in I$ if and only if $\phi(f) \in I\cap \widehat{P}$, which shows that $\Phi$ is well-defined, 
injective and surjective.
\end{proof}

\begin{definition}\label{def-isoZseparating}
 The $K$-algebra isomorphism $\Phi$ defined in the above proposition is called the \textbf{$Z$-separating re-embedding} of $I$ (or of $P/I$).
\end{definition}

The following examples illustrate the use of coherently $Z$-separating tuples.
For the actual computations, see \cocoa-Example~\ref{coex-Zsepar}.

\begin{example}\label{ex-Zseparation}
Let $P = \QQ[x, y, z, w]$, and consider the ideal $I = \langle f_1, f_2, f_3 \rangle$,
where
\begin{equation*}
f_1 = x^3 - y + z, \quad f_2 = xw^2 - x - y, \quad f_3 = x^2 - yw. 
\end{equation*}
Let $X = \{x, y, z, w\}$ and $\M = \langle X \rangle$. 
The linear parts are $\Lin(f_1) = -y + z$, $\Lin(f_2) = -x - y$, and $\Lin(f_3) = 0$.
Let $\sigma$ be a term ordering on $\mathbb{T}^4$ defined by a weight matrix 
whose first row is $(0, 1, 2, 0)$. Under this ordering, we have 
$\LT_\sigma(f_1) = z$ and $\LT_\sigma(f_2) = y$. 
The reduced $\sigma$-Gr\"obner basis of $I$ is
\begin{equation*}
G = ( z + x^3 - xw^2 + x, \quad y - xw^2 + x, \quad xw^3 - x^2 - xw ). 
\end{equation*}
Let $Z = (y, z)$. The first two polynomials in $G$ form a coherently $Z$-separating tuple. 
In~particular, the third polynomial is obtained by substituting $y \mapsto xw^2 - x$ into $f_3$, 
as described in Proposition~\ref{prop-ElimViaSubst}.
Since $\sigma$ is an elimination term ordering for $Z$, we deduce that the 
elimination ideal is $I \cap \QQ[x, w] = \langle xw^3 - x^2 - xw \rangle$. 
This yields the following isomorphism of $\QQ$-algebras.
\begin{equation*}
\Phi \colon P/I \To \QQ[x, w] / \langle xw^3 - x^2 - xw \rangle.
\end{equation*}
More information about this example is given in Example~\ref{ex-Zseparation continued}.
\end{example}

We now address the following question: using coherently $Z$-separating tuples, 
can one obtain a re-embedding that uses the minimal number of indeterminates among all re-embeddings? 
In view of our goal to find presentations of $P/I$ with minimal ambient dimensions,
we introduce the following terminology.

\begin{definition}\label{def-edim}
Let $P = K[x_1,\dots, x_n]$ and let~$I$ be an ideal in $P$.

\begin{myenumerate}
\item A $K$-algebra isomorphism $\Psi \colon P/I  \longrightarrow Q/J$,
where $Q$ is a polynomial ring over~$K$ and~$J$ is an ideal in~$Q$, is
called a \textbf{re-embedding} of~$I$.

\item A re-embedding $\Psi \colon P/I \longrightarrow Q/J$ is called
\textbf{optimal} if every $K$-al\-ge\-bra isomorphism $P/I \longrightarrow Q'/J'$ 
with a polynomial ring~$Q'$ over~$K$ and~$J'$ an ideal in~$Q'$ 
satisfies $\dim(Q')\ge \dim(Q)$.

\item For an optimal re-embedding $\Psi \colon P/I \longrightarrow Q/J$, we define
$\edim(P/I) = \dim(Q)$ and call it the \textbf{embedding dimension} of~$P/I$.

\item The ring $P/I$ is called a  free commutative $K$-algebra
if there exist a polynomial ring~$Q$ over~$K$ and a $K$-algebra isomorphism 
$P/I \cong Q$. Since we are exclusively 
dealing with commutative rings here, we also call $P/I$ a \textbf{free $K$-algebra}.

\item A  $Z$-separating re-embedding
$\Phi \colon P/I \longrightarrow \widehat{P}/(I\cap \widehat{P})$, 
where $\widehat{P} = K[X{\setminus}Z]$,
is called \textbf{best separating} if for every subset $Z' \subseteq X$
for which the ideal $I$ admits a $Z'$-separating tuple, one has $\#Z' \le \#Z$.

\item For a best $Z$-separating re-embedding $\Phi$, we define
$\sepdim(P/I) = n - \#Z$
and call it the \textbf{separating embedding dimension} of~$P/I$.

\item A $Z$-separating re-embedding that is optimal is called an
\textbf{optimal separating re-embedding}.

\item The ring~$P/I$ is called a \textbf{separating free $K$-algebra} if there exists 
a $Z$-separating re-embedding $\Phi \colon P/I \longrightarrow K[X{\setminus}Z]$.  

\end{myenumerate}
\end{definition}

\medskip
\begin{myremark}\label{rem-best-optimal}
In~\cite[Definition~3.3.d]{KLR0}, the term ``optimal separating re-embedding'' 
was used for $Z$-separating re-embeddings such that $\#Z$ is as large as possible. 
To avoid confusion, in~\cite[Section~2]{KR6} we introduced the term 
\textit{best separating re-embedding}, which we also adopt here (see item~(5) of Definition~\ref{def-edim}).
Consequently, when we use the term \textit{optimal separating re-embedding} we mean a 
$Z$-separating re-embedding that is optimal (see item~(7) of Definition~\ref{def-edim}).
\end{myremark}

\medskip
\begin{myremark}\label{rem-correction}
In~\cite[Definition~3.3]{KLR0} and~\cite[Proposition~3.2]{KLR2},
the re-embedding defined above is called either a re-embedding of~$I$
or a re-embedding of~$P/I$. Throughout this survey, we consistently use
the terminology \textit{re-embedding of~$I$}, as explained in
Notation~6 of the Introduction.
\end{myremark}

\begin{example}\label{ex-LinGFanContinued}
We return to Example~\ref{ex-LinearGFan}. 
Recall that, using an elimination term ordering for $(z, w)$, we obtained 
the  Gr\"obner basis $( z - xy - x + y,\; w - \tfrac{1}{4}x^2 - \tfrac{1}{4}xy + \tfrac{1}{2}y )$.
Consequently, the associated separating re-embedding of the ideal~$I$ is the isomorphism 
$\Phi \colon P/I \To \QQ[x,y]$ defined by $\Phi(\bar{x}) = x$, $\Phi(\bar{y}) = y$, 
$\Phi(\bar{z}) = xy + x - y$, and $\Phi(\bar{w}) = \tfrac{1}{4}x^2 + \tfrac{1}{4}xy - \tfrac{1}{2}y$.
Hence $P/I$ is a separating free $\QQ$-algebra with $Z = (z, w)$.
\end{example}

Using a \textit{reverse-engineering technique}, the following example shows that an ideal $I$ can be
$Z$-separating, although this property is not apparent from a given set of generators of~$I$.

\begin{example}\label{ex-reverse}
Consider the ring $P=\QQ[x,y,z]$  and the ideal
$I=\langle f_1,\dots, f_8\rangle$, where
\begin{align*}
f_1 &= y^3 - xy - y^2 + x,\\
f_2 &= xy^2 - \tfrac{1}{3}y^2z - x^2 - y^2 + \tfrac{1}{3}xz + x,\\
f_3 &= y^2z^2 - 3y^2z - xz^2 + 3xz,\\
f_4 &= x^2y^2 - x^3 - 2y^2z + 2xz,\\
f_5 &= x^2y + y^3 - x^2 - y^2,\\
f_6 &= x^3 + xy^2 - \tfrac{1}{3}x^2z - \tfrac{1}{3}y^2z - x^2 - y^2,\\
f_7 &= x^2z^2 + y^2z^2 - 3x^2z - 3y^2z,\\
f_8 &= x^4 + x^2y^2 - 2x^2z - 2y^2z.
\end{align*}

The ideal $I$ does not appear to be $(x)$-separating. 
However, for any elimination term ordering~$\sigma$ for $(x)$, 
the reduced $\sigma$-Gr\"obner basis of $I$ is $\{x - y^2,\; y^4 + y^2\}$, 
which reveals that $I$ is, in fact, $(x)$-separating. 
Moreover, we have an optimal $(x)$-separating re-embedding 
$\Phi \colon P/I \longrightarrow \mathbb{Q}[y, z] / \langle y^4 + y^2 \rangle$. 
Note that $P/I$ is not a free $\mathbb{Q}$-algebra, since it is 
not an integral domain.
\end{example}

The following example shows that the embedding dimension can be smaller than the 
separating embedding dimension.

 \begin{example}\label{ex-isomorphtoK[x]} 
Let  $P=\QQ[x,y]$, let $f = x +2x^8 +8x^6y +12x^4y^2 +8x^2y^3 +2y^4$, and
let $I=\langle f \rangle$.
If we let $\M =\langle x, y\rangle$, we have $\Lin(f) = x$. 

Clearly, $x$ cannot be the leading term of $f$,
hence we have $\sepdim(P/I) = 2$.
On the other hand,  the $\QQ$-algebra homomorphism 
$\phi \colon \QQ[x, y] \To \QQ[x]$ defined by $\phi(x) = -2x^4$ and $\phi(y) = x - 4x^8$
satisfies  $\phi(x^2 +y) = x$ as well as $\Ker(\phi) = \langle f \rangle$
(see \cocoa-Example~\ref{coex-isotoK[x]}).
It follows that $\Phi \colon \QQ[x, y]/ \langle f \rangle \To \QQ[x]$ 
is an optimal re-embedding of~$I$,  hence $\edim(P/I) = 1$,
and $P/I$ is a free $\QQ$-algebra, not a separating free $\QQ$-algebra.
 \end{example}

Next, we provide further insight into the inequality $\edim(R) \le \sepdim(R)$.

\goodbreak
\begin{theorem}\label{thm-ineqsep}
Let $P=K[x_1,\dots,x_n]$, let $\M=\langle x_1,\dots,x_n\rangle$, let~$I\subseteq \M$,
 let $R=P/I$, and let $\m=\M/I$.
We have
\begin{equation*}
\begin{aligned}
n - \dim_K(\Lin(I)) 
&= \dim_K(\Cot_\m(P/I)) \\
&\le \edim(R) \\
&\le \sepdim(R) \\
&\le n - \#\LI(\overline{G}).
\end{aligned}
\end{equation*}
\end{theorem}

\begin{proof} 
The first equality was proved in Proposition~\ref{prop-cotancomp}\,(4).
To prove the first inequality, we note that both $\edim(R)$ and the cotangent space 
$\Cot_\m(R) = \m/\m^2$ are invariants of~$R$ (see~\cite[Chapter 10]{Ei}), 
and hence do not depend on the chosen presentation.
Consequently, we may assume that $R=P/I$ is a presentation with $n=\dim(P)=\edim(R)$.
Then
\begin{equation*}
\dim_K(\Cot_\m(P/I)) = n - \dim_K(\Lin(I)) \le n = \edim(R).
\end{equation*}
The second inequality follows from the definition of $\edim(R)$.

The last inequality follows from the fact that, by definition,  $\sepdim(P/I) = n - m$, where 
$m = \max\{\, \#\LI(\overline{H}) \mid \overline{H} \in \GFan(I) \}$.
\end{proof}

From the theorem we see that a sufficient condition for $\edim(R)=\sepdim(R)$ is that 
$\dim_K(\Lin(I)) = \#\LI(\overline{G})$. We now make this condition more explicit.
 Moreover, the following theorem provides a criterion for an affine $K$-algebra 
 to be  a separating free $K$-algebra.

\begin{theorem}{\bf{(Criteria for Optimality and Free Algebra)}}\label{thm-checkopt}

Let $K$ be a field, let $P = K[x_1, \dots, x_n]$, let $X = (x_1, \dots, x_n)$, and
let $\M = \langle X \rangle$. Let~$I$ be an ideal in~$P$ such that $I \subseteq \M$, 
and let $Z$ be a tuple of distinct indeterminates in~$X$ 
such that $I$ is $Z$-separating. Assume that  $\#(Z) = \dim_K(\Lin(I))$. 
Then the following claims hold.

\begin{myenumerate}
\item 
The $Z$-separating re-embedding
$\Phi \colon P/I \longrightarrow \widehat{P} / (I\cap \widehat{P})$ 
where $ \widehat{P} = K[X{\setminus}Z]$
is an optimal  re-embedding of~$I$, and $\edim(P/I) = n - \#(Z)$.

\item If in addition we assume that  $P_\M/I_\M$ is regular, then 
$P/I$ is a free $K$-algebra.
\end{myenumerate}
\end{theorem}

\begin{proof}
Let us prove (1).
Let $G$ be the reduced Gr\"obner basis of~$I$ with respect to an elimination term ordering for $Z$.
We know from Remark~\ref{rem-LI} that 
$\#\LI(\overline{G}) \le \dim_K(\Lin(I))$.
On the other hand, since $I$ contains a coherently $Z$-separating tuple,
$G$ contains polynomials whose leading indeterminates are the elements of $Z$.
Hence we have 
\begin{equation*}
\#(Z)\le \#\LI(\overline{G}) \le \dim_K(\Lin(I)) = \#(Z).
\end{equation*}
Consequently, $\dim_K(\Lin(I)) = \#\LI(\overline{G})$, 
so all inequalities in Theorem~\ref{thm-ineqsep}
become equalities, and the assertions follow.

Next, we prove (2). From the regularity of $P_\M/I_\M$,
we know that $\dim_K(\Cot_{\M/I}(P/I)) = \dim(P_\M/I_\M)$, hence 
$\dim_K(\Lin(I)) = n - \dim(P_\M/I_\M)$ follows from Proposition~\ref{prop-cotancomp}\,(4).
Consequently, we obtain $\#(X{\setminus}Z) =  \dim(P_\M/I_\M)$.
Thus, the $Z$-sep\-arating re-embedding of~$I$ is of the form 
$\Phi \colon P/I \longrightarrow \widehat{P}/(I \cap \widehat{P})$,
where $\widehat{P} = K[X{\setminus}Z]$ is a polynomial ring in $d=\dim(P_\M/I_\M)$ indeterminates.
From the fact that  $ \dim(P_\M/I_\M)$ is invariant under the isomorphism $\Phi$, we deduce
$\dim(\widehat{P} / (I \cap \widehat{P})) \ge d$, which, together with $\dim(\widehat{P}) = d$,
implies $I \cap \widehat{P} = \langle 0 \rangle$.
\end{proof}

The following example continues Example~\ref{ex-Zseparation}.

\begin{example}\label{ex-Zseparation continued}
From Example~\ref{ex-Zseparation} we recall that 
$\Lin(f_1) = -y + z$, $\Lin(f_2) = -x - y$, and $\Lin(f_3) = 0$.
It was shown that the tuple $Z = (y, z)$ is separating,
and that it can be used to produce the re-embedding 
$\Phi \colon P/I \To \QQ[x, w] / \langle xw^3 - x^2 - xw \rangle$.

Note that in this case we have $\#Z = \dim_K(\Lin(I)) = 2$.
By Theorem~\ref{thm-checkopt}, we deduce that $\Phi$ is an optimal re-embedding.

This fact can also be checked directly as follows.
Since $\dim(P/I) = 1$, the only possibility for $\edim(P/I) < 2$ is that 
there exists an isomorphism $P/I \cong Q$, where $Q$ is a polynomial ring in one indeterminate. 
This is not possible, since $\QQ[x, w] / \langle xw^3 - x^2 - xw \rangle$ is not an integral domain.
We can also conclude that $P/I$ is not a free $\QQ$-algebra.
\end{example}

We have already seen Example~\ref{ex-isomorphtoK[x]} where $\edim(P/I) < \sepdim(P/I)$.
In general, computing $\edim(P/I)$ is a difficult problem.
For instance, consider the following example.
 
\begin{example}\label{ex-Crachiola}
Let $P=\QQ[x, y, z, t]$ and  let $I=\langle f\rangle$, where $f = x + x^2y + z^2 + t^3$.
Clearly, the only reduced Gr\"obner basis of $I$ is $\{f\}$.
Since $f$ is not separating for any indeterminate, the optimal separating re-embedding
of~$I$ is the identity map, and hence $\sepdim(P/I) = 4$.
A natural question is whether this is also an optimal re-embedding, i.e., whether 
$\edim(P/I)=4$. Since $\dim(P/I) = 3$, the only possibility for $\edim(P/I) = 3$ 
is that $P/I$ is isomorphic to a polynomial ring in three indeterminates.
Using a technique based on special invariants of polynomial rings,
the existence of such an isomorphism can be excluded (see~\cite{Cr}).
\end{example}

\newpage   
\section{Back to Border Basis Schemes}
\label{sec-Back to BBS}

In this section we come back to the theory of border basis schemes.
The following is an important remark which motivates the entire subsection.

\begin{myremark}(\bf{Either} $c_{ij}$ \bf{or} $c_{ij}-c_{k\ell}$)\label{rem-easylin}\\
\normalfont  In~\cite{KSL} and~\cite[Propositions~6.3 and~6.4]{KLR2}, a detailed 
description of the linear and quadratic parts of the generators of $I(\BO)$ is 
given using the definitions introduced in Subsection~\ref{subsec-IntRimExp}. 
The key observation is that the linear parts (if nonzero)
are only of two types, namely $c_{ij}$ and $c_{ij}-c_{k\ell}$. 
Consequently, the idea of using $\Lin(I)$, explained in Remark~\ref{rem-GFanLin}, 
becomes particularly effective in the case of border basis schemes, as we now
describe.
\end{myremark}

\medskip
\subsection{Cotangent Equivalence}
\label{subsec-cotequiv}

First,  we subdivide the tuple of indeterminates $C=(c_{ij})$ as follows.

\begin{definition}\label{def-cotequiv}
Let $\BO$ be a border basis scheme and let $B_\OO=K[C]/I(\BO)$ be its coordinate ring.
We denote the maximal ideal $\langle C\rangle$ of $K[C]$ by $\M$ and $\langle C\rangle/I(\BO)$
by~$\m$.
For every indeterminate $c_{ij} \in C$, we let $\bar{c}_{ij}$ denote
its residue class in the cotangent space $ \m/ \m^2$ 
of $B_\OO$ at the origin.
\begin{myenumerate}
\item The relation~$\sim$ on~$C$ defined
by $c_{ij}\sim c_{kl} \Leftrightarrow K\cdot \bar{c}_{ij}= K \cdot \bar{c}_{kl}$
is an equivalence relation called \textbf{cotangent equivalence}.

\item An indeterminate $c_{ij} \in C$ is called \textbf{trivial}
if $\bar{c}_{ij}=0$. The trivial indeterminates form the \textbf{trivial
cotangent equivalence class} in~$C$.

\item An  indeterminate $c_{ij} \in C$ is called \textbf{basic} if $\bar{c}_{ij} \ne 0$
and its cotangent equivalence class  contains only ${c}_{ij}$.
In this case, the cotangent equivalence class $\{ c_{ij} \}$ is also called \textbf{basic}.

\item An indeterminate $c_{ij}\in C$ is called \textbf{proper}
if $\bar{c}_{ij} \ne 0$ and its cotangent equivalence class contains at least two elements. 
In this case, the cotangent equivalence class of~$c_{ij}$ is also called \textbf{proper}.
\end{myenumerate}
\end{definition}

\begin{lemma}\label{lem-equivclasses}
Let $\BO$ be a border basis scheme, let $B_\OO=K[C]/I(\BO)$ 
be its coordinate ring, and let $U$ be the union of the supports
of the linear polynomials in a set of generators of~$\Lin(I(\BO))$.
\begin{myenumerate}
\item The set~$U$ does not depend on the choice of a
set of generators  of~$\Lin(I(\BO))$. 

\item The set of basic indeterminates is $C\setminus U$.

\item The union of the sets of trivial and proper indeterminates is~$U$.
\end{myenumerate}
\end{lemma}

\begin{proof}
To prove (1) we note that any set of generators of the ideal
$\Lin(I(\BO))$ is also a set of generators 
of the $K$-vector space $\Lin(I(\BO))_K$.
Let~$A$ and~$B$ be two such sets. Since every linear form in~$A$ 
is a linear combination of linear forms in~$B$, each indeterminate in its support
is in the support of some linear form in~$B$. By interchanging the roles of~$A$ and~$B$, 
the conclusion follows.

Since $\sim$ is an equivalence relation on~$C$, to prove claims~(2) and~(3) it suffices 
to show that basic indeterminates are not in~$U$, while trivial and proper indeterminates
are in~$U$.
Firstly, let $c_{ij}$ be a basic indeterminate. For a contradiction, assume that $c_{ij} \in U$.
From~$c_{ij} \in U$ and the fact that $c_{ij}$ is the only element in its equivalence class, 
we deduce that there is a polynomial in~$I(\BO)$ of the form $c_{ij} +q$ 
with $q \in \M^2$. Hence we get that $c_{ij}$ is trivial, a contradiction.
Secondly, let $c_{ij}$ be trivial.  Then there is a polynomial in~$I$ of the form $c_{ij}+q$
with $q \in \M^2$, and hence we get $c_{ij} \in U$. 
Thirdly, let $c_{ij}$ be proper. Then there exist another indeterminate 
$c_{kl}$ and a polynomial in~$I(\BO)$ 
of the form  $c_{ij} -c_{k l} +q$ with  $q \in \M^2$. Hence we get $c_{ij} \in U$.
\end{proof}

Our next step is to order the indeterminates in the cotangent equivalence classes
using a term ordering. 

\begin{definition}\label{def-leadsetofPr}
Let $\TT^{\mu\nu}$ denote the  monoid of terms in the indeterminates $c_{ij}\in C$ and
let~$\sigma$ be a term ordering on~${\mathbb T}^{\mu\nu}$.
\begin{myenumerate}
\item Let $E= \{c_{(i_1, j_1)}, \dots, c_{(i_p, j_p)} \}$ be a proper equivalence class in~$C$
such that we have 
$c_{(i_1, j_1)} >_\sigma \cdots  >_\sigma c_{(i_p, j_p)} $. Then the set
$E\setminus\{ c_{(i_p, j_p)} \}$ is called the {\bf $\sigma$-leading set} of~$E$ and
denoted by~$E^\sigma$.

\item The unique minimal set of indeterminates generating the ideal 
$\LT_\sigma(\Lin(I(\BO)))$ is denoted by~$S_\sigma$.
\end{myenumerate}
\end{definition}

\begin{theorem}\label{thm-shapeofSsigma}
Let $E_0$ be the trivial equivalence class and $E_1, \dots, E_q$ 
the proper equivalence classes in~$C$. 

\begin{myenumerate}
\item Let~$\sigma$ be a term ordering on $\mathbb{T}^{\mu\nu}$.
Then we have $S_\sigma = E_0 \cup E_1^\sigma \cup \cdots \cup E_q^\sigma$.

\item  For $i=1,\dots,q$, let $E_i^\ast$ be a set obtained from~$E_i$ 
by deleting one of its elements. Then there exists a term ordering~$\sigma$
on $\mathbb{T}^{\mu\nu}$ such that we have
$S_\sigma = E_0 \cup E_1^\ast \cup \cdots \cup E_q^\ast$.
\end{myenumerate}
\end{theorem}

\begin{proof}
To prove claim~(1) we observe that $E_0 \subseteq S_\sigma$ follows from 
$E_0 \subseteq \Lin_\M(I(\BO))$, and that  the inclusion 
$S_\sigma \subseteq E_0 \cup E_1\cup \cdots \cup E_q$
follows from Lemma~\ref{lem-equivclasses}\,(3).

For $k\in \{1,\dots,q\}$, we write the proper equivalence class 
$E_k= \{c_{i_1j_1}, \dots, c_{i_p j_p} \}$ 
such that $c_{i_1 j_1} >_\sigma \cdots  >_\sigma c_{i_p j_p}$. 
Using the definition of a proper equivalence class, it follows that 
$c_{i_\ell j_\ell} - c_{i_p j_p} \in \Lin_\M(I(\BO))$, so that
$\LT_\sigma(c_{i_\ell j_\ell} - c_{i_p j_p}) = c_{i_\ell j_\ell}$ 
for every $\ell \in \{1,\dots, p-1\}$.
Hence we have proved 
$$
E_0 \cup E_1^\sigma \cup \cdots \cup E_q^\sigma \;\subseteq\;  
S_\sigma \;\subseteq\;   E_0 \cup E_1\cup \cdots \cup E_q
$$ 
Therefore the $\sigma$-smallest 
element in each proper equivalence class does not belong to~$S_\sigma$, 
which finishes the proof of~(1).
 
Claim~(2) follows from~(1) if we show that there exists a term ordering~$\sigma$ 
such that $E_i^\sigma = E_i^\ast$ for $i=1,\dots,q$.
By definition, the sets~$E_i$ are pairwise disjoint. Consequently, a term ordering 
which solves the problem can be chosen as a block term ordering, and hence it 
suffices to consider the case $q=1$. 
So, let $E_1 = \{c_{i_1j_1}, \dots, c_{i_p j_p} \}$  and w.l.o.g. assume 
that $E_1^\ast = \{c_{i_2 j_2}, \dots, c_{i_p j_p} \}$. As we observed before, we have 
$c_{i_\ell j_\ell} - c_{i_1 j_1} \in \Lin(I(\BO))$ for $\ell=2,\dots, p$.
To finish the proof it suffices to take a term ordering~$\sigma$ such that 
$c_{i_\ell j_\ell} >_\sigma c_{i_1 j_1}$ for $\ell=2,\dots, p$.
\end{proof}

\begin{myremark}\label{rem-LinLT}
Recall from Proposition~\ref{prop-sepZ} that  $I$ is $Z$-separating
if $Z \subset \LT_\sigma(\Lin(I))$ where $\sigma$ is an elimination ordering for $Z$.
In the case of border basis schemes, Theorem~\ref{thm-shapeofSsigma} 
gives an important restriction on where to look for suitable $Z$.
\end{myremark}

\begin{myremark}
In~\cite[Section~5]{KLR2} a generalization of the above results is described, in particular  
to ideals generated by polynomials whose linear parts are binomials.
\end{myremark}

The following example illustrates the results of this section. 
For some explicit  computation see \cocoa-Example~\ref{coex-PropTrivBas}.
It was first mentioned in~\cite[Remark~7.5.3]{Hui1}. 

\begin{example}\label{ex-PropTrivBas}
In $P = \QQ[x,y]$, consider the order ideal $\OO = \{t_1,\dots,t_8\}$ given by 
$t_1=1$, $t_2=y$, $t_3=x$, $t_4=y^2$, $t_5=xy$, $t_6=x^2$, $t_7=y^3$, and $t_8 = xy^2$.
Then, its border is $\partial\OO = \{b_1,\dots,b_5\}$, where 
$b_1=x^2y$, $b_2=x^3$, $b_3=y^4$, $b_4 = xy^3$, and $b_5=x^2y^2$.

\begin{center}
\begin{minipage}[c]{0.45\textwidth}
Thus, $\QQ[C] = \QQ[c_{11}, \dots, c_{85}]$ is a polynomial ring in $40$ indeterminates. 
Note that the dimension of $\BO$ is $\dim(\BO) =n \mu  = 16$ and that there are 
32 natural generators of the ideal $I(\BO)$.
\end{minipage}%
\qquad\qquad
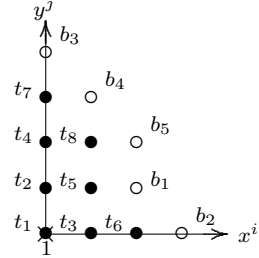
\begin{minipage}[c]{0.40\textwidth}
\centering
\beginpicture
        \setcoordinatesystem units <0.6cm,0.6cm>
        \setplotarea x from 0 to 3.5, y from 0 to 4.5
        \axis left /
        \axis bottom /
        \arrow <2mm> [.2,.67] from 3.5 0 to 4 0
        \arrow <2mm> [.2,.67] from 0 4.2 to 0 4.7
        \put {$\scriptstyle x^i$} [lt] <0.5mm,0.8mm> at 4.1 0.1
        \put {$\scriptstyle y^j$} [rb] <1.7mm,0.7mm> at 0 4.6
        \put {$\bullet$} at 0 0
        \put {$\bullet$} at 1 0
        \put {$\bullet$} at 0 1
        \put {$\bullet$} at 1 1
        \put {$\bullet$} at 1 2
        \put {$\bullet$} at 0 2
        \put {$\bullet$} at 0 3
        \put {$\bullet$} at 2 0
        \put {$\scriptstyle 1$} [lt] <-1mm,-1mm> at 0 0
        \put {$\scriptstyle t_1$} [rb] <-1.3mm,0.4mm> at 0 0
        \put {$\scriptstyle t_3$} [rb] <-1.3mm,0.4mm> at 1 0
        \put {$\scriptstyle t_2$} [rb] <-1.3mm,0mm> at 0 1
        \put {$\scriptstyle t_4$} [rb] <-1.3mm,0mm> at 0 2
        \put {$\scriptstyle t_5$} [rb] <-1.3mm,0mm> at 1 1
        \put {$\scriptstyle t_6$} [rb] <-1.3mm,0.4mm> at 2 0
        \put {$\scriptstyle t_7$} [rb] <-1.3mm,0mm> at 0 3
        \put {$\scriptstyle t_8$} [rb] <-1.3mm,0mm> at 1 2
        \put {$\scriptstyle b_1$} [lb] <1.5mm,0mm> at 2 1
        \put {$\scriptstyle b_2$} [lb] <1.5mm,1.2mm> at 3 0
        \put {$\scriptstyle b_3$} [lb] <1.5mm,1.2mm> at 0 4
        \put {$\scriptstyle b_4$} [lb] <1.5mm,1.2mm> at 1 3
        \put {$\scriptstyle b_5$} [lb] <1.5mm,1.2mm> at 2 2
        \put {$\times$} at 0 0

        \put {$\circ$} at 3 0
        \put {$\circ$} at 2 2
        \put {$\circ$} at 2 1
        \put {$\circ$} at 1 3
        \put {$\circ$} at 0 4
\endpicture 
\end{minipage}
\end{center}

The linear parts of these generators are 
$$
\begin{array}{l}
c_{65},\;  c_{51} - c_{85},\;  c_{45},\;   c_{44},\;  c_{55},\;  c_{43} - c_{54},\;   
c_{42},\;  c_{41} - c_{75},\;  c_{52} - c_{75},\;  c_{35},  \cr 
c_{34},\;  c_{33},\;  c_{31},\;  c_{25},\;  c_{24},\;  c_{23},\;  c_{22},\;  c_{21},\;  
c_{32},\;  c_{15},\;  c_{14},\;  c_{13},\;  c_{12},\;   c_{11}  
\end{array}
$$
If we let $U$ be the union of the supports of these elements, we get 
$$
C\setminus U = \{c_{53},\,  c_{61},\,  c_{62},\,  c_{63},\,  c_{64},\,  c_{71},\,  
c_{72},\,  c_{73},\,  c_{74},\,  c_{81},\,  c_{82},\,  c_{83},\,  c_{84} \}
$$
which is exactly the set of basic indeterminates by Lemma~\ref{lem-equivclasses}\,(2).
For the trivial cotangent equivalence class $E_0$ and the 
proper cotangent equivalence classes $E_1,\dots,E_q$, we get
\begin{align*}
E_0 = & \{ c_{11},  c_{12},  c_{13},  c_{14},  c_{15},  c_{21},  c_{22},  c_{23},  
c_{24},  c_{25},  c_{31},  c_{32},  c_{33},  c_{34},  c_{35}, c_{42},  c_{44},  c_{45},  c_{55},  c_{65} \}\\
E_1 = & \{c_{51},  c_{85}\},\qquad  E_2 = \{ c_{43}, c_{54}\},\qquad
E_3 = \{c_{41}, c_{52}, c_{75}\}
\end{align*}
Using $\sigma= \tt DegRevLex$, we obtain $E_1^\sigma = \{c_{51}\}$, $E_2^\sigma = \{c_{43}\}$,
and $E_3^\sigma = \{c_{41},  c_{52} \}$.
Next, we compute the set $S_\sigma$ of the minimal generators of the leading term ideal 
of $\Lin(I(\BO))$ and get
$$
\begin{array}{l}
S_\sigma = \{  c_{11},  c_{12},  c_{13},  c_{14},  c_{15},  c_{21},  c_{22},  c_{23},  
c_{24},  c_{25},  c_{31},  c_{32},  c_{33},  c_{34},  c_{35},  \\
\quad\quad\   c_{41}, c_{42},  c_{43},  c_{44},  c_{45},  c_{51},  c_{52},  c_{55},  c_{65}
  \}
\end{array}
$$
which coincides with $E_0\cup E_1^\sigma \cup E_2^\sigma \cup E_3^\sigma$ 
(see Theorem~\ref{thm-shapeofSsigma}\,(1)).

In view of Remark~\ref{rem-LinLT}, we compute all sets of the form 
$S_\sigma = E_0 \cup E_1^\ast \cup \cdots \cup E_q^\ast$.\\
We obtain the following 12 sets
\begin{alignat*}{3}
	Z_1  &= E_0\cup \{ c_{51} \} \cup \{ c_{43} \} \cup \{ c_{41},\ c_{52}\} \qquad \quad & Z_2  &= E_0\cup \{ c_{51} \} \cup \{ c_{43} \} \cup \{ c_{41},\ c_{75}\} \\
	Z_3  &= E_0\cup \{ c_{85} \} \cup \{ c_{43} \} \cup \{ c_{41},\ c_{52}\} \qquad \quad & Z_4  &= E_0\cup \{ c_{85} \} \cup \{ c_{43} \} \cup \{ c_{41},\ c_{75}\} \\
	Z_5  &= E_0\cup \{ c_{51} \} \cup \{ c_{54} \} \cup \{ c_{41},\ c_{52}\} \qquad \quad & Z_6  &= E_0\cup \{ c_{51} \} \cup \{ c_{54} \} \cup \{ c_{41},\ c_{75}\} \\
	Z_7  &= E_0\cup \{ c_{85} \} \cup \{ c_{54} \} \cup \{ c_{41},\ c_{52}\} \qquad \quad & Z_8  &= E_0\cup \{ c_{85} \} \cup \{ c_{54} \} \cup \{ c_{41},\ c_{75}\} \\
	Z_9  &= E_0\cup \{ c_{51} \} \cup \{ c_{43} \} \cup \{ c_{52},\ c_{75}\} \qquad \quad & Z_{10} &= E_0\cup \{ c_{85} \} \cup \{ c_{43} \} \cup \{ c_{52},\ c_{75}\} \\
	Z_{11} &= E_0\cup \{ c_{51} \} \cup \{ c_{54} \} \cup \{ c_{52},\ c_{75}\} \qquad \quad & Z_{12} &= E_0\cup \{ c_{85} \} \cup \{ c_{54} \} \cup \{ c_{52},\ c_{75}\}
\end{alignat*}
The corresponding complementary sets $C\setminus Z_i$ for $i =1, \dots, 12$ can then be determined.
For instance, we have 
\begin{equation*}
Y=C\setminus Z_1 = \{ c_{53}, c_{54}, c_{61}, c_{62}, c_{63}, c_{64}, c_{71}, c_{72}, 
c_{73}, c_{74}, c_{75}, c_{81}, c_{82}, c_{83}, c_{84}, c_{85} \}.
\end{equation*}
An explicit computation shows that $I(\BO)$ is $Z_1$-separating and there exists an isomorphism
$K[C]/I(\BO) \To \QQ[Y]$. In particular, this shows that $\BO$ is an affine cell.
The same conclusion follows for every choice of $Z_i$.
\end{example}

\medskip
\subsection{Interior, Rim, and Exposed Indeterminates}
\label{subsec-IntRimExp}

\begin{definition}\label{def-rim-and-interior}
Let $\OO = \{t_1, \dots, t_\mu\}$ be an order ideal in~$\mathbb{T}^n$, and let
$\partial\OO = \{b_1, \dots, b_\nu\}$ be its border.
\begin{myenumerate}
\item For $i\in \{1,\dots, \mu\}$, the term~$t_i$ is called a {\bf rim term}
in~$\OO$ if there exists an indeterminate~$x_k$ such that $x_k t_i \in \partial\OO$.
Otherwise, the term~$t_i$ is called an {\bf interior term} of~$\OO$.
The set of all rim terms is denoted by $\OOrim$, and the set of all interior terms
is denoted by $\OOint$.

\item For $i\in \{1,\dots, \mu\}$ and $j\in \{1,\dots, \nu\}$, the 
indeterminate $c_{ij}$ is called a {\bf rim indeterminate} if~$t_i$ is a rim term
in~$\OO$. The set of all rim indeterminates will be denoted by~$\Crim$.

\item For $i\in \{1,\dots, \mu\}$ and $j\in \{1,\dots, \nu\}$, the 
indeterminate $c_{ij}$ is called an {\bf interior indeterminate} if~$t_i$ is an
interior term of~$\OO$. The set of all interior indeterminates will be denoted
by~$\Cint$.

\end{myenumerate}
\end{definition}

Note that the set $C\setminus U$, described in Example~\ref {ex-PropTrivBas},
 is contained in the set of rim indeterminates.
 
Next we provide some additional terminology which is related to the structure of the
natural set of generators of~$I(\BO)$ (see~Definition~\ref{def-neighbors}). 
It was first introduced in~\cite[Section~4.1]{Hui1}.

\begin{definition}\label{def-exposedIndets}
Let $\OO = \{t_1, \dots, t_\mu\}$ be an order ideal in~$\mathbb{T}^n$, and let
$\partial\OO = \{b_1, \dots, b_\nu\}$ be its border.
\begin{myenumerate}
\item For $\ell \in \{1,\dots,n\}$ and $j\in \{1,\dots,\nu\}$,
the term~$b_j$ is called {\bf $x_\ell$-exposed} if it is of 
the form $b_j = x_\ell t_i$ with $t_i\in\OO$.

\item For $\ell\in \{1,\dots,n\}$, for $i\in \{1,\dots, \mu\}$, and for 
$j\in \{1,\dots, \nu\}$, we say that the 
indeterminate $c_{ij} \in C$ is {\bf $x_\ell$-exposed} if
$x_\ell t_i \in \partial \OO$ and if
there exists an index $j'\in \{1,\dots, \nu\}$ such that $(b_j, b_{j'})$
is a  next-door pair with $b_{j'} = x_\ell b_j$ or an across-the-street
pair with $x_k b_{j'} = x_\ell b_j$ for some $k \in \{1,\dots, n\}$.

\item The subtuple of~$C$ consisting of all indeterminates $c_{ij}$
which are $x_\ell$-exposed for some $\ell\in\{1,\dots,n\}$,
is called the tuple of all {\bf exposed indeterminates} in~$C$ 
and denoted by~$\Cexp$. The indeterminates in $C \setminus \Cexp$ 
are called {\bf non-exposed}.

\end{myenumerate}
\end{definition}

The next proposition provides two characterizations of exposed
indeterminates which motivate the above definition.

\begin{proposition}{\bf (Characterizing Exposed Indeterminates)}\label{prop-CharExposed}

Let $\OO = \{t_1,\dots,t_\mu\}$ be an order ideal in~$\mathbb{T}^n$, let $\partial\OO = 
\{b_1,\dots,b_\nu\}$ be its border, and
let $G = \{g_1,\dots,  g_\nu\}$ be the generic $\OO$-border 
prebasis, where $g_j \;= \; b_j - \sum_{i=1}^\mu c_{ij} t_i$ for $j=1,\dots, \nu$.
Then, for all $i\in \{1,\dots, \mu\}$ and $j\in \{1,\dots, \nu\}$, 
the following conditions are equivalent.
\begin{myenumerate}
\item There exists an index $\ell \in \{1,\dots, n\}$ such that
$c_{ij} \in C$ is $x_\ell$-exposed.

\item 
There exists an index $j'\in \{1,\dots, \nu\}$ and an index $\ell\in\{1,\dots,n\}$ such that $(b_j, b_{j'})$
is a next-door pair with $b_{j'} = x_\ell b_{j}$ or an across-the-street
pair with $x_k b_{j'} = x_\ell b_j$ for some $k\in\{1,\dots,n\}$,
and such that if we write
\begin{equation*}
x_\ell g_j \;=\; x_\ell b_j  + c_{i_1 j} g_{j_1} + \cdots + c_{i_r j} g_{j_r} 
+ \tsum_{\lambda=1}^\mu f_\lambda t_\lambda
\end{equation*}
where $i_1,\dots,  i_r\in \{1,\dots, \mu\}$, where $j_1,\dots, j_r\in\{1,\dots, \nu\}$, 
and where $f_\lambda \in K[C]$, then we have $i \in \{i_1, \dots,  i_r \}$.

\item 
There exists an index $j'\in \{1,\dots, \nu\}$ and an index $\ell\in\{1,\dots,n\}$ such that $(b_j, b_{j'})$
is a next-door pair with $b_{j'} = x_\ell b_j$ or an across-the-street
pair with $x_k b_{j'} = x_\ell b_j$ for some $k\in\{1,\dots,n\}$,
and such that the lifting of the neighbor syzygy $e_{j'} - x_\ell e_j$ resp.
$x_k e_{j'} - x_\ell e_j$ is of the form
\begin{equation*}
x_k^\epsilon e_{j'} - x_\ell e_j - \tsum_{\kappa=1}^\nu \bar{f}_\kappa  e_{\kappa}
\in \Syz_{B_\OO}( \bar{g}_1, \dots,  \bar{g}_\nu )
\end{equation*}
where $\epsilon\in \{0,1\}$ and $\bar{f}_\kappa \in B_\OO$ are the residue classes 
of monomials or binomials in $K[C]$ such that one of their supports contains $c_{ij}$.

\end{myenumerate}
\end{proposition}

\begin{proof}
To show (1)$\Rightarrow$(2), we write $x_\ell t_i = b_{j''}$ with 
$j'' \in \{1,\dots, \nu\}$. By definition, there exists a pair $(b_j, b_{j'})$
such that $x_k^\epsilon b_{j'} = x_\ell b_j$ for some $j, j' \in \{1,\dots, \nu\}$
and $\epsilon \in \{0,1\}$. 
Then we have
$$
x_\ell g_j \;=\; x_\ell b_j - \tsum_{\lambda=1}^\mu  c_{\lambda j} x_\ell t_\lambda
$$
and if we denote those terms $x_\ell t_{\lambda}$ which are in $\partial\OO$ by
$x_\ell t_{i_1} = b_{j_1}$, $\dots$, $x_\ell t_{i_r} = b_{j_r}$ then
the fact that $x_\ell t_i \in \partial\OO$ shows $i \in \{i_1,\dots,  i_r\}$.

Conversely, the condition $f_\lambda \in K[C]$ implies that the terms in the
support of $x_\ell (g_j-b_j)$ 
which are contained in $K[C] \cdot \partial \OO$ are of the form $c_{i_\kappa j} b_{j_\kappa}$
with $\kappa \in \{1,\dots, r\}$. Therefore $b_{j_\kappa} = x_\ell t_{i_\kappa}$
for $\kappa = 1,\dots, r$ and $c_{ij} \in \{ c_{i_1 j}, \dots,  c_{i_r j} \}$ 
is $x_\ell$-exposed.

The equivalence of~(2) and~(3) follows from the fact that the representation of
$x_k^\epsilon g_{j'} - x_\ell g_j$, which follows from~(2), is exactly the lifting
of the neighbor syzygy $x_k^\epsilon e_{j'} - x_\ell e_j$ in 
$\Syz_{B_\OO}( \bar{g}_1, \dots,  \bar{g}_\nu )$ by~\cite{KR3}, Prop. 4.1.
\end{proof}

Notice that all interior indeterminates are non-exposed, i.e., that
all exposed indeterminates are rim indeterminates.

The following examples illustrate Definition~\ref{def-exposedIndets}.

\goodbreak

\begin{example}\label{ex-exposed}
(see  \cocoa-Example~\ref{coex-exposed}).
In $P = \QQ[x,y]$, consider the order ideal $\OO = \{t_1,\dots,t_5\}$ given by 
$t_1=1$, $t_2=y$, $t_3=x$, $t_4=y^2$, $t_5=y^3$.
Then, its border is $\partial\OO = \{b_1,\dots,b_5\}$, where 
$b_1=xy$, $b_2=x^2$, $b_3=xy^2$, $b_4 = y^4$, and $b_5=xy^3$.

\begin{center}
\begin{minipage}[c]{0.5\textwidth}
Thus, $\QQ[C] = \QQ[c_{11}, \dots, c_{55}]$ is a polynomial ring in $25$ indeterminates. 
Let us see what  the exposed indeterminates are in this case.
The indeterminates which are $x$-exposed  are
$c_{2,1},\,  c_{2,4},\,  c_{3,1}, \, c_{3,4}, \, c_{4,1},\,  c_{4,4},\,  c_{5,1},\,  c_{5,4}$.

The indeterminates which are $y$-exposed  are
$c_{3,1}, \, c_{3,2},\,  c_{3,3},\,  c_{3,5},\,  c_{5,1},\,  c_{5,2},\,  c_{5,3},\,  c_{5,5}$.
\end{minipage}%
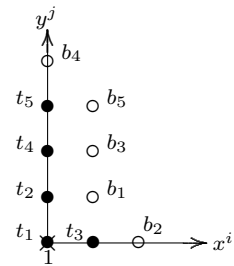
\qquad\qquad
\begin{minipage}[c]{0.40\textwidth}
\centering
\beginpicture
        \setcoordinatesystem units <0.6cm,0.6cm>
        \setplotarea x from 0 to 3, y from 0 to 4.5
        \axis left /
        \axis bottom /
        \arrow <2mm> [.2,.67] from 3 0 to 3.5 0
        \arrow <2mm> [.2,.67] from 0 4.2 to 0 4.7
        \put {$\scriptstyle x^i$} [lt] <0.5mm,0.8mm> at 3.5 0.1
        \put {$\scriptstyle y^j$} [rb] <1.7mm,0.7mm> at 0 4.6
        \put {$\bullet$} at 0 0
        \put {$\bullet$} at 1 0
        \put {$\bullet$} at 0 1
        \put {$\bullet$} at 0 2
        \put {$\bullet$} at 0 3
        \put {$\scriptstyle 1$} [lt] <-1mm,-1mm> at 0 0
        \put {$\scriptstyle t_1$} [rb] <-1.3mm,0.4mm> at 0 0
        \put {$\scriptstyle t_3$} [rb] <-1.3mm,0.4mm> at 1.1 0
        \put {$\scriptstyle t_2$} [rb] <-1.3mm,0mm> at 0 1
        \put {$\scriptstyle t_4$} [rb] <-1.3mm,0mm> at 0 2
        \put {$\scriptstyle t_5$} [rb] <-1.3mm,0mm> at 0 3
        \put {$\scriptstyle b_1$} [lb] <1.5mm,0mm> at 1 1
        \put {$\scriptstyle b_2$} [lb] <1.5mm,1.2mm> at 1.8 0
        \put {$\scriptstyle b_3$} [lb] <1.5mm,1.2mm> at 1 1.8
        \put {$\scriptstyle b_4$} [lb] <1.5mm,1.2mm> at 0 3.8
        \put {$\scriptstyle b_5$} [lb] <1.5mm,1.2mm> at 1 2.8
        \put {$\times$} at 0 0
        
        \put {$\circ$} at 2 0
        \put {$\circ$} at 1 1
        \put {$\circ$} at 1 2
       \put {$\circ$} at 1 3
        \put {$\circ$} at 0 4
\endpicture 
\end{minipage}
\end{center}

Let us do some checking.
The across-the-street neighbor pairs are   $(b_1, b_2)$ and $(b_4, b_5)$, 
while the next-door neighbor pairs are $(b_1, b_3)$ and $(b_3, b_5)$.   
\begin{myenumerate}
\item[]
The indeterminate $c_{21}$ is $x$-exposed because $x t_2 \in \partial\OO$, and $y b_2 = x b_1$.

\item[]
The indeterminate $c_{31}$ is $x$-exposed because  $x t_3 \in \partial\OO$, and $y b_2 = x b_1$.

\item[]
The indeterminate $c_{31}$ is also $y$-exposed because $y t_3 \in \partial\OO$, and $b_3 = y b_1$.

\item[]
The indeterminate $c_{43}$ is not $x$-exposed because $x t_4 \in \partial\OO$, but there is no 
neighbor pair  involving $x b_3$. It is also not $y$-exposed because $y t_4 \notin \partial\OO$.
\end{myenumerate}
\end{example}

The following example shows that, in order to find exposed indeterminates, 
it may be necessary to use all AS pairs, including those that are not AR pairs.

\begin{example}\label{ex-exposedNeedAcrossTheCorner}
In $P = \QQ[x,y]$, consider the order ideal $\OO = \{t_1,\dots,t_6\}$ given by 
$t_1=1$, $t_2=y$, $t_3=x$, $t_4=y^2$, $t_5=x^2$, $t_6 = x^3$.
Then, its border is $\partial\OO = \{b_1,\dots,b_6\}$, where 
$b_1=xy$, $b_2=y^3$, $b_3=xy^2$, $b_4 = x^2y$, and $b_5=x^3y$, $b_6 = x^4$.

\begin{center}
\begin{minipage}[c]{0.33\textwidth}
The AS pair $(b_3, b_4)$ is not an AR pair (see Remark~\ref{rem-a-t-s useful for exposed}). 
However it is essential to find the $x$-exposed indeterminate $c_{23}$ and the $y$-exposed
indeterminates $c_{34}$, $c_{54}$.
\end{minipage}%
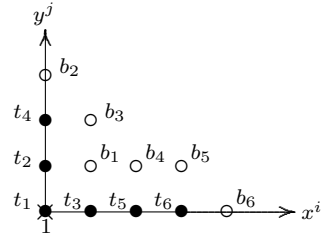
\qquad\qquad
\begin{minipage}[c]{0.40\textwidth}
\centering
\beginpicture
        \setcoordinatesystem units <0.6cm,0.6cm>
        \setplotarea x from 0 to 3, y from 0 to 4
        \axis left /
        \axis bottom /
        \arrow <2mm> [.2,.67] from 3 0 to 5.5 0
        \arrow <2mm> [.2,.67] from 0 3.8 to 0 4
        \put {$\scriptstyle x^i$} [lt] <0.5mm,0.8mm> at 5.5 0.1
        \put {$\scriptstyle y^j$} [rb] <1.7mm,0.7mm> at 0 4
        \put {$\bullet$} at 0 0
        \put {$\bullet$} at 1 0
        \put {$\bullet$} at 0 1
        \put {$\bullet$} at 0 2
         \put {$\bullet$} at 2 0
          \put {$\bullet$} at 3 0
        \put {$\scriptstyle 1$} [lt] <-1mm,-1mm> at 0 0
        \put {$\scriptstyle t_1$} [rb] <-1.3mm,0.4mm> at 0 0
        \put {$\scriptstyle t_3$} [rb] <-1.3mm,0.4mm> at 1.1 0
         \put {$\scriptstyle t_5$} [rb] <-1.3mm,0.4mm> at 2.1 0
         \put {$\scriptstyle t_6$} [rb] <-1.3mm,0.4mm> at 3.1 0
        \put {$\scriptstyle t_2$} [rb] <-1.3mm,0mm> at 0 1
        \put {$\scriptstyle t_4$} [rb] <-1.3mm,0mm> at 0 2

         \put {$\scriptstyle b_1$} [lb] <1.5mm,0mm> at 0.9 1.1
          \put {$\scriptstyle b_2$} [lb] <1.5mm,0mm> at 0 3
         \put {$\scriptstyle b_4$} [lb] <1.5mm,0mm> at 1.9 1.1
         \put {$\scriptstyle b_5$} [lb] <1.5mm,0mm> at 2.9 1.1
          \put {$\scriptstyle b_6$} [lb] <1.5mm,0mm> at 3.9 0.1

        \put {$\scriptstyle b_3$} [lb] <1.5mm,1.2mm> at 1 1.8
        \put {$\times$} at 0 0
        
        \put {$\circ$} at 0 3
        \put {$\circ$} at 1 1
         \put {$\circ$} at 2 1
          \put {$\circ$} at 3 1
        \put {$\circ$} at 1 2
        \put {$\circ$} at 4 0
\endpicture 
\end{minipage}
\end{center}
\end{example}

\medskip
\subsection{Planar Box Border Basis Schemes}
\label{subsec-PlanarBoxBBS}

Exposed indeterminates of planar border basis schemes play a fundamental 
role in this subsection. The reason lies in  the following result. 

\begin{proposition}\label{prop-elimNonExp}
Let $K$ be a field, let $P=K[x,y]$,
let $\OO = \{ t_1,\dots, t_\mu \}$ be an order ideal in~$\mathbb{T}^2$,
and let $Z$ be the tuple consisting of the non-exposed 
indeterminates $C\setminus \Cexp$.
Then $I(\BO)$ is coherently $Z$-separating.
\end{proposition}

\begin{proof}
The proof follows from \cite[Algorithm 6.3]{KR5} which assigns weights wt,
so that
\begin{myenumerate}
\item[-] all exposed indeterminates $c_{ij}\in \Crim$ satisfy
$\wt(c_{ij}) = 0$.

\item[-] all non-exposed indeterminates $c_{ij} \in \Cint$ satisfy
$\wt(c_{ij}) > 0$.
\end{myenumerate}
And for each non-exposed indeterminate $c_{ij}$ there
exists a natural generator~$f$ of the ideal~$I(\BO)$ for which $c_{ij}$
is the unique term of highest weight in~$\Supp(f)$.
\end{proof}

\goodbreak
This result can be used to show an important feature of box border basis schemes,

\begin{definition}\label{def-boxBBS}
Let $K$ be a field, let $P=K[x_1,\dots, x_n]$, and let $a_1,\dots, a_n\in\NN_+$.
\begin{myenumerate}
\item The order ideal 
$\OO = \mathbb{T}^n \setminus \langle x_1^{a_1}, \dots,  x_n^{a_n} \rangle =
\{ x_1^{i_1} \cdots x_n^{i_n} \mid 0\le i_j < a_j\;\hbox{for}\;
j=1,\dots, n\}$ is called the {\bf box} of type $(a_1,\dots, a_n)$.

\item If~$\OO$ is a box, the corresponding border basis
scheme is called a {\bf box border basis scheme}.

\item If $n=2$, the corresponding border basis
schemes are called {\bf planar box border basis schemes}.
\end{myenumerate}
\end{definition}

\goodbreak
\begin{definition}\label{def-planarboxBBS}
Let $K$ be a field, let $P=K[x, y]$, and let $\OO$ be a box of type $(a, b)$.

\begin{myenumerate}
\item The set $\{ x^{a-1}y^k \mid 0\le k\le b-1\}$ is called the {\bf right rim}
of~$\OO$.

\item The set $\{ x^k y^{b-1} \mid 0\le k\le a-1\}$ is called the {\bf upper rim}
of~$\OO$.

\item The set $\{ x^a y^k \mid 0\le k\le b-1\}$ is called the {\bf right border}
of~$\OO$.

\item[(d)] The set $\{x^k y^b \mid 0\le k\le a-1\}$ is called the {\bf upper border}
of~$\OO$.
\end{myenumerate}
\end{definition}

\begin{proposition}{\bf (Planar Box Border Basis Schemes 
are Affine Cells)}\label{prop-PlanarBox}\\
Let $a,b\in\mathbb{N}_+$, and let $\OO = \{ x^i y^j \mid
0\le i< a,\; 0\le j< b\}$ be the box of type $(a,b)$ in
$P=K[x,y]$.
\begin{myenumerate}
\item An indeterminate~$c_{ij}$ is $x$-exposed if and only if
$t_i$ is a term in the right rim and~$b_j$ is a term in the
upper border of~$\OO$.

In particular, the number of $x$-exposed indeterminates~$c_{ij}$ is~$ab$.

\item An indeterminate $c_{ij}$ is $y$-exposed if and only if
$t_i$ is a term in the upper rim and~$b_j$ is a term in the
right border of~$\OO$.

In particular, the number of $y$-exposed indeterminates~$c_{ij}$ is~$ab$.

\item The tuple~$Z$ of all non-exposed indeterminates is $I(\BO)$-separating.
The $Z$-sep\-a\-rating re-embedding yields an isomorphism
$\Phi \colon B_\OO \longrightarrow K[\Cexp]$.

In particular, every planar box border basis scheme is an affine cell.

\end{myenumerate}
\end{proposition}

\begin{proof}
To prove~(1), we consider the up-neighbors in $x$-direction, i.e.,
the next-door neighbor pairs $x b_j = b_{j'}$ for which $b_j, b_{j'}$
are in the upper border.
All terms $b_j$ in the upper border, except for $b_j=x^{a-1} y^b$,
have an up-neighbor in $x$-direction, i.e., there is a next-door neighbor pair
$x b_j = b_{j'}$. The term $b_j=x^{a-1} y^b$ has an across-the-street neighbor
$b_{j'} =  x^a y^{b-1}$, since $x b_j = y b_{j'}$.
Altogether, for each term~$b_j$ in the upper border and for each term~$t_i$ in 
the right rim we get an $x$-exposed indeterminate~$c_{ij}$. In this way 
we get all $x$-exposed indeterminates.

The proof of~(2) follows analogously. It remains to prove~(3).
Since there are no other exposed indeterminates, we have $\# \Cexp = 2ab$.
By Proposition~\ref{prop-elimNonExp}, we know that we can 
eliminate all non-exposed indeterminates. Therefore
$\dim(\BO) = 2 \mu = 2ab$ shows that $\BO \cong \mathbb{A}^{2ab}_K$,
and that~$B_\OO$ is a polynomial ring over~$K$ generated by the residue classes
of the indeterminates in $\Cexp$ described in~(1) and~(2).
\end{proof}

The following example illustrates this proposition.
\begin{example}\label{ex-24Box}
In $P = \QQ[x,y]$, consider the order ideal $\OO = \{t_1,\dots,t_6\}$ given by 
$t_1=1$, $t_2=y$, $t_3=x$, $t_4=y^2$, $t_5=xy$, $t_6 = xy^2$.
Then, its border is $\partial\OO = \{b_1,\dots,b_5\}$, where 
$b_1=x^2$, $b_2=y^3$, $b_3=x^2y$, $b_4 = xy^3$, and $b_5=x^2y^2$.
 \begin{center}
\begin{minipage}[c]{0.43\textwidth}
Thus, $\QQ[C] = \QQ[c_{11}, \dots, c_{65}]$ is a polynomial ring in $30$ indeterminates. 
Let us see what  the exposed indeterminates are in this case.\\
The $x$-exposed  indeterminates are\\
$c_{3,2},\,  c_{3,4},\,  c_{5,2}, \, c_{5,4}, \, c_{6,2},\,  c_{6,4}$

\smallskip
The~$y$-exposed  indeterminates are\\
$c_{4,1}, \, c_{4,3},\,  c_{4,5},\,  c_{6,1},\,  c_{6,3},\,  c_{6,5}$
\end{minipage}%
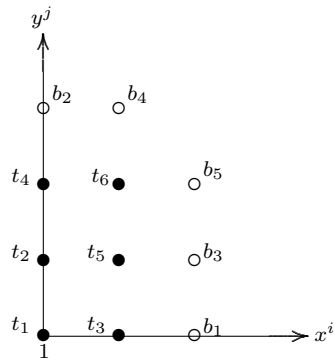
\qquad\qquad
\begin{minipage}[c]{0.40\textwidth}
\centering
\beginpicture
		\setcoordinatesystem units <1cm,1cm>
		\setplotarea x from 0 to 3, y from 0 to 3.5
		\axis left /
		\axis bottom /
		\arrow <2mm> [.2,.67] from  3 0  to 3.5  0
		\arrow <2mm> [.2,.67] from  0 3.5  to 0 4
		\put {$\scriptstyle x^i$} [lt] <0.5mm,0.8mm> at 3.5 0.1
		\put {$\scriptstyle y^j$} [rb] <1.7mm,0.7mm> at 0  4
		\put {$\bullet$} at 0 0
		\put {$\bullet$} at 1 0
		\put {$\bullet$} at 0 1
		\put {$\bullet$} at 1 1
	    \put {$\bullet$} at 1 2
		\put {$\bullet$} at 0 2
		\put {$\scriptstyle 1$} [lt] <-1mm,-1mm> at 0 0
		\put {$\scriptstyle t_1$} [rb] <-1.3mm,0.4mm> at 0 0
		\put {$\scriptstyle t_3$} [rb] <-1.3mm,0.4mm> at 1 0
		\put {$\scriptstyle t_2$} [rb] <-1.3mm,0mm> at 0 1
		\put {$\scriptstyle t_4$} [rb] <-1.3mm,0mm> at 0 2
		\put {$\scriptstyle t_5$} [rb] <-1.3mm,0mm> at 1 1
		\put {$\scriptstyle t_6$} [lb] <-4mm,0mm> at 1 2
		\put {$\scriptstyle b_3$} [lb] <0.8mm,0mm> at 2 1
		\put {$\scriptstyle b_1$} [lb] <0.8mm,0mm> at 2 0
		\put {$\scriptstyle b_2$} [lb] <0.8mm,0.8mm> at 0 3
	    \put {$\scriptstyle b_4$} [lb] <0.8mm,0.8mm> at 1 3
	    \put {$\scriptstyle b_5$} [lb] <0.8mm,0.8mm> at  2 2
	         \put {$\circ$} at 2 0
                 \put {$\circ$} at 0 3
		\put {$\circ$} at 2 2
		\put {$\circ$} at 2 1
		\put {$\circ$} at 1 3
\endpicture
\end{minipage}
\end{center}
The conclusion is that there is an isomorphism 
\begin{equation*}
K[C]/I(\BO) \To K[c_{3,2},  c_{3,4},  c_{5,2},  c_{5,4},  c_{6,2}, c_{6,4},
c_{4,1},  c_{4,3},  c_{4,5},  c_{6,1},  c_{6,3},  c_{6,5}]
\end{equation*}
hence $\BO$ is a 12-dimensional affine cell.
\end{example}

\newpage   
\section{Positive \texorpdfstring{$P_0$}{P0}-algebras}
\label{sec-Positive P0-algebras}

An important role in the description of special border basis schemes is played by 
positive $P_0$-algebras. However, they are also of independent interest, which motivates 
our current investigation.

In this section, we let $K$ be a field, and let $P = P_0[x_1, \dots, x_s]$ 
be a polynomial ring over~$P_0$, where $P_0 = K[a_1, \dots, a_m]$ 
is a polynomial ring over~$K$ which we consider as trivially graded, 
meaning that the degree of $a_i$ is 0 for all $i$.
The ring $P$ is equipped with an additional grading defined by a weight vector
$W = (w_1, \dots, w_s) \in \mathbb{Z}^s$.
We denote $(x_1, \dots, x_s)$ by~$X$. Finally, we set $n = m+s$, so 
that $P = P_0[X]$ is a polynomial ring in~$n$ indeterminates.

In the following, when we say that a polynomial or an ideal in $P$ is $W$-homogeneous,
we mean that it is homogeneous with respect to the $\mathbb{Z}$-grading defined by 
$\deg_W(a_i) = 0$ for $i = 1, \dots, m$ and $\deg_W(x_j) = w_j$ for $j = 1, \dots, s$.

\medskip
\subsection{Best and Optimal Re-embeddings of Positive \texorpdfstring{$P_0$}{P0}-algebras}
\label{subsec-Best and Optimal Re-embeddings of Positive P0-algebras}

\begin{definition}\label{def-posalg} 
Let $P_0$ and $P$ be as described above.
A graded $P_0$-algebra is a ring of the form $P/I$, where $I$ is a $W$-homogeneous 
ideal in $P$.
If $w_j > 0$ for $j = 1, \dots, s$ and $I \cap P_0 = \langle 0 \rangle$, we say that $P/I$ 
is a \textbf{positive $P_0$-algebra}.
In other words, we require that $P/I$ is a positively graded $P_0$-algebra 
with $(P/I)_0 = P_0$, so that the canonical $K$-algebra homomorphism $\phi \colon P_0 \To P/I$ 
is injective.

In the case $P_0 = K$, a positive $K$-algebra $P/I$ satisfies the condition 
$I \cap K = \langle 0 \rangle$ if and only if the $W$-homogeneous ideal $I$ is proper.
\end{definition}

The next result on positive $K$-algebras was first proved in~\cite[Proposition~3.4]{RoTe}, 
and was later reformulated as an algorithm in~\cite[Corollary~4.8]{KR6}.
Here we adopt a constructive approach which uses results from 
Subsection~\ref{subsec-Z-Separating Tuples and Term Orderings}, 
in particular Proposition~\ref{prop-ElimViaSubst}.
Both in~\cite{KR6} and in the following proposition, no Gr\"obner basis computation is needed.

Recall that $\Lin(f)$ denotes the linear part of a polynomial~$f$ with respect 
to the standard grading (see Definition~\ref{def-cotangent}).

\begin{proposition}{\bf (Optimal Re-Embeddings of Positively Graded Ideals)}
\label{prop-optimalgraded}

Let $P = K[x_1, \dots, x_s]$ be positively graded by~$W \in \mathbb{Z}^s$, 
let $G = (g_1, \dots, g_r)$ be a tuple of $W$-homogeneous polynomials in~$\M$, 
let $I = \langle G \rangle$, and let $d = \dim_K(\Lin(I))$.
Then there exist a $d$-tuple of indeterminates and a $d$-tuple 
of $W$-homogeneous polynomials defining an optimal separating re-embedding 
of~$I$. 
\end{proposition}

\begin{proof}
From Proposition~\ref{prop-cotancomp}\,(2), we have
$\Lin(I) = \langle \Lin(g_1), \dots, \Lin(g_r) \rangle_K$. 
After possibly permuting $G$, we may assume that 
$(\Lin(g_1), \dots, \Lin(g_d))$ is a $K$-basis of $\Lin(I)$. 
Consequently,  if we let  $G_d =(g_1, \dots, g_d)$, we have $\rk(\cmat(G_d)) = d$.
Let $Z = (z_1, \dots, z_d)$ be a tuple of indeterminates of top rank for $G_d$, 
and let $U_Z$ be the invertible submatrix of $\cmat(G_d)$ corresponding to the 
indeterminates in $Z$.
Multiplying $(g_1, \dots, g_d)$ on the right by~$U_Z^{-1}$, we obtain a new tuple 
$(g_1', \dots, g_d')$ of $W$-homogeneous polynomials in $I$ such that 
$g_i' = z_i - h_i'$ for $i = 1, \dots, d$. 
Their homogeneity forces them to be $z_i$-separating.
The interreduction of the tuple $(g_1', \dots, g_d')$, obtained by replacing $z_i$ with $h_i'$ 
in the polynomials $g_1', \dots, \widehat{g_i'}, \dots, g_d'$, produces a coherently 
$Z$-separating tuple $(f_1, \dots, f_d)$ with $f_i = z_i - h_i$. 
Moreover, Proposition~\ref{prop-ElimViaSubst} shows that such an interreduction applied to~$G$
produces a tuple $(\hat{g}_1, \dots, \hat{g}_r)$ which generates $I \cap K[X{\setminus}Z]$.
Using this, we obtain an isomorphism
\begin{equation*}
\Phi \colon P/I \To K[X{\setminus}Z] / \langle \hat{g}_1, \dots, \hat{g}_r \rangle
\end{equation*}
Its optimality is guaranteed by Theorem~\ref{thm-checkopt}, since $\#Z= \dim_K(\Lin(I)) = d$.
\end{proof}

The following example illustrates the proposition and shows the advantage of using 
the technique of \textit{elimination by substitution} (see Proposition~\ref{prop-ElimViaSubst}).

\begin{example}\label{ex-RoTe}
Let $P = \QQ[x, y, z, w]$ be graded by $W = (3, 2, 1, 1)$, and let 
$f_1 = x - yz$, $f_2 = y - w^2$, 
$f_3 = x^2y - z^8$, $f_4 = xz^4 + xy^2$.
Let $G = (f_1, f_2, f_3, f_4)$ and let $I = \langle G \rangle$.
It is clear that $(f_1, f_2)$ is coherently $Z$-separating with $Z = (x, y)$. 
The interreduction of $(f_3, f_4)$ obtained by substituting $x$ with $yz$ and $y$ with $w^2$ 
produces the pair $(z^2w^6 - z^8, \, z^5w^2 + zw^6)$. Using this, we obtain the optimal re-embedding
\begin{equation*}
\Phi \colon P/I \To \QQ[z, w]/J 
\end{equation*}
where $J = \langle z^8 -z^2w^6, \, z^5w^2 + zw^6 \rangle$.

By contrast, using an elimination term ordering $\sigma$, we obtain the tuple 
\begin{equation*}
(z^8 - z^2w^6,\, z^5w^2 + zw^6, \,  z^4w^6 + z^2w^8, \, z^3w^8 - zw^{10}, \, z^2w^{10}, \, zw^{12})
\end{equation*}
which is the reduced $\sigma$-Gr\"obner basis of $J$.
\end{example}

In Proposition~\ref{prop-optimalgraded}, we used the notion of $\Lin$ introduced 
and studied in Subsections~\ref{subsec-Cotangent Spaces and GFans}
and~\ref{subsec-Z-Separating Re-embeddings and Optimal Re-embeddings}. 
In the case of positive $P_0$-algebras, we extend this notion as follows.

\begin{definition}\label{def-linZ}
Let $P_0 = K[a_1, \dots, a_m]$ be a trivially graded polynomial ring over~$K$, 
let $P = P_0[x_1, \dots, x_s]$, let $I$ be an ideal in $P$ such that $P/I$ is a positive $P_0$-algebra, 
and let $X = (x_1, \dots, x_s)$.
\begin{myenumerate} 
\item Given a $W$-homogeneous polynomial $f \in \langle X \rangle$, write
$f = h_1 x_1 + \cdots + h_s x_s + g$ where $h_1, \dots, h_s \in P_0$  
and $g \in \langle X \rangle^2$. Then the $P_0$-linear 
form $\Lin_{P_0}(f) = h_1 x_1 + \cdots + h_s x_s$ is called the \textbf{$P_0$-linear part} of~$f$.

\item The $P_0$-module generated by the linear parts of the polynomials 
in $I$ is called the \textbf{$P_0$-linear part} of~$I$ and is denoted by $\Lin_{P_0}(I)$. 
In other words, we have 
$\Lin_{P_0}(I) = \langle \Lin_{P_0}(f) \mid f \in I \cap \langle X \rangle \rangle_{P_0}$.
\end{myenumerate}
\end{definition}

Next, we extend part of Proposition~\ref{prop-cotancomp} to positive $P_0$-algebras.

\begin{proposition}\label{prop-indepOfGen}
Let $P/I$ be a positive $P_0$-algebra and let $I = \langle f_1, \dots, f_r \rangle$.
Then the following claims hold true.
\begin{myenumerate}
\item $\Lin_{P_0}(I) = \langle X \rangle_{P_0} \cap (I + \langle X \rangle^2)$.

\item $\Lin_{P_0}(I) = \langle \Lin_{P_0}(f_1), \dots, \Lin_{P_0}(f_r) \rangle_{P_0}$.
\end{myenumerate}
\end{proposition}

\begin{proof}
To prove (1), we start with the containment $\subseteq$.
Since~$P/I$ is a positive $P_0$-algebra, we have $I \subseteq \langle X \rangle$.
Hence, if we let $\langle X \rangle_{P_0} = P_0 x_1 + \cdots + P_0 x_s$, then 
an element $f \in I$ can be written uniquely in the form $f = \ell + g$ with
$\ell \in \langle X \rangle_{P_0} $ and $g \in \langle X \rangle^2$.
Then $\Lin_{P_0}(f) = \ell \in \langle X \rangle_{P_0}$ and 
$\ell = f - g \in I + \langle X \rangle^2$ imply the claim.   

Conversely, let $\ell \in \langle X \rangle_{P_0}$ be of the form 
$\ell = f + g$ with $f \in I$ and $g \in \langle X \rangle^2$. Then $f = \ell - g$ 
shows that $\ell = \Lin_{P_0}(f) \in \Lin_{P_0}(I)$, and this concludes the proof of~(1).

To prove~(2), for each $j = 1, \dots, r$, we write $f_j = \ell_j + g_j$ 
with $\ell_j = \Lin_{P_0}(f_j)$ and $g_j \in \langle X \rangle^2$. 
It follows that $I + \langle X \rangle^2 = \langle \ell_1, \dots, \ell_r \rangle_{P_0} + \langle X \rangle^2$.
Using~(1), we obtain
\begin{equation*}
\Lin_{P_0}(I) = \langle X \rangle_{P_0} \cap (I + \langle X \rangle^2)
= \langle X \rangle_{P_0} \cap 
(\langle \Lin_{P_0}(f_1), \dots, \Lin_{P_0}(f_r) \rangle_{P_0} + \langle X \rangle^2)
\end{equation*}
Since $\langle \Lin_{P_0}(f_1), \dots, \Lin_{P_0}(f_r) \rangle_{P_0} \subseteq \langle X \rangle_{P_0}$, 
by the modular law for $P_0$-modules and the equality $\langle X \rangle_{P_0} \cap \langle X \rangle^2 = \langle 0 \rangle$, we get
the equality $\Lin_{P_0}(I) = \langle \Lin_{P_0}(f_1), \dots, \Lin_{P_0}(f_r) \rangle_{P_0}$.   
\end{proof}

\begin{definition}\label{def-P0coeffmat}
Let $P = P_0[X]$ be a polynomial ring over~$P_0$ such that $P/I$ is a positive $P_0$-algebra, 
and let $Z = (z_1, \dots, z_t)$ be a subtuple of $X = (x_1, \dots, x_s)$.
\begin{myenumerate}
\item Polynomials in $P_0$ are called \textbf{$P_0$-constants}.
We say that a polynomial $f \in P$ is a \textbf{$P_0$-monomial} if $f = h \prod_{j=1}^s x_j^{\alpha_j}$ 
with $h \in P_0$; in this case, the $P_0$-constant $h$ is called the \textbf{$P_0$-coefficient} of $f$.

\item Let $f$ be a $W$-homogeneous polynomial.
If $h_1 z_1 + \cdots + h_t z_t$ is the portion of~$\Lin_{P_0}(f)$ contained in 
$\langle Z \rangle_{P_0} = P_0 z_1 + \cdots + P_0 z_t$, 
we denote it by $\Lin_{P_0,Z}(f)$, and the $t$-tuple $(h_1, \dots, h_t) \in P_0^t$
of its $P_0$-coefficients is denoted by $\cvec_{P_0,Z}(f)$.

\item Let $F=(f_1, \dots, f_r)$ be a tuple of $W$-homogeneous polynomials in $\langle X\rangle$.
Then the matrix in $\Mat_{r, t}(P_0)$ whose rows are $\cvec_{P_0, Z}(f_1), \dots,  \cvec_{P_0, Z}(f_r)$
is denoted by~$\cmat_{P_0, Z}(F)$.
\end{myenumerate} 
\end{definition}

Next, we establish analogies with  Proposition~\ref{prop-sepZ}.

\begin{theorem}{\bf (Existence of $W$-homogeneous $Z$-separating tuples)}
\label{thm-existZsep}

Let $P = P_0[X]$ and let $P/I$ be a positive $P_0$-algebra graded by $W$.
Let $G = (g_1, \dots, g_r)$ be a tuple of $W$-homogeneous polynomials which generates $I$, and
let $Z = (z_1, \dots, z_t)$ be a subtuple of $X$.  Assume that there exists a tuple  
$F = (f_1, \dots, f_t)$ of $W$-homogeneous polynomials in~$I$
such that $z_i = \Lin_{P_0,Z}(f_i)$ for $i =1, \dots, t$. Then the following conditions hold true.

\begin{myenumerate}
\item We have $\langle Z \rangle_{P_0} = 
\langle \Lin_{P_0,Z}(g_1), \dots, \Lin_{P_0,Z}(g_r) \rangle_{P_0}$.

\item There exists a $Z$-separating term ordering $\sigma$ for $I$ such that 
$F$ is $Z$-separating and the reduced $\sigma$-Gr\"obner basis 
of $\langle F \rangle$ is  coherently $Z$-separating.

\item For every elimination term ordering $\sigma$ for $Z$, 
the tuple $F$ is $Z$-separating and the reduced $\sigma$-Gr\"obner basis 
of $\langle F \rangle$ is  coherently $Z$-separating.

\end{myenumerate}
\end{theorem}

\begin{proof}
To prove (1), we notice that the inclusion $\supseteq$ is clear, so let us prove 
the inclusion~$\subseteq$.
Since $g_1, \dots, g_r$ generate~$I$, we can write each $f_j$ as 
$f_j = \sum_{\ell=1}^r q_{j\ell} g_\ell$ with $q_{j\ell} \in P$.
Applying the operator $\Lin_{P_0,Z}$ to both sides, we observe that, in the product $q_{j\ell}g_\ell$,
only the $P_0$-constants $h_{j\ell}$ of $q_{j\ell}$ contribute to the $P_0$-monomials
of the form $h z$ with $z \in Z$.
Thus, we obtain the relation $z_j = \Lin_{P_0,Z}(f_j) = \Lin_{P_0,Z} \left( \sum_{\ell=1}^r q_{j\ell} g_\ell \right)
= \sum_{\ell=1}^r h_{j\ell} \Lin_{P_0,Z}(g_\ell)$.

To prove (2), we produce a term ordering $\sigma$ on $P$ such that $\LT_\sigma(f_j) = z_j$ 
for every~$j=1,\dots, t$.
Recall that $n$ is the number of indeterminates in $P$. We define $\sigma = \ord(U)$, 
where the matrix $U \in \Mat_{n, n}(\mathbb{Z})$ is constructed as follows.
For each $j = 1, \dots, t$, the first two rows of $U$ have entries
$\deg_W(z_j)$ and $\deg_W(z_j)-1$, respectively, in the columns corresponding 
to the indeterminate $z_j$, and zero in the columns corresponding to the other indeterminates of $X$.
The remaining $n-2$ rows of $U$ are chosen so that the total matrix defines a term ordering on $P$.
If $\deg_W(z_j) = 1$, then no $P_0$-monomial in $f_j - z_j$ is divisible by any $z_k$ with $k \ne j$, 
hence $\LT_\sigma(f_j) = z_j$.
If $\deg_W(z_j) > 1$, it may happen that the polynomial $f_j - z_j$ contains $P_0$-monomials of
the form $h \prod_{i=1}^m z_{\ell_i}$, where $z_{\ell_1}, \dots, z_{\ell_m} \in Z \setminus \{z_j\}$ and $m > 1$.
In this case, the homogeneity of $f_j$ yields $\deg_W(z_j) = \sum_{i=1}^m \deg_W(z_{\ell_i})$, 
while the evaluation through the second row of $U$ gives 
$\deg_W(z_j) - 1 > \sum_{i=1}^m (\deg_W(z_{\ell_i}) - 1) = \sum_{i=1}^m \deg_W(z_{\ell_i}) - m$.
Therefore, in this case as well, the leading term is forced to be $\LT_\sigma(f_j) = z_j$.
Moreover, the reduced $\sigma$-Gr\"obner basis of $\langle F \rangle$ is 
coherently $Z$-separating, 
as already noted in Remark~\ref{rem-SepCohSep}.

Finally, claim (3) follows from Proposition~\ref{prop-sepandcoherentlysep}.
\end{proof}

\begin{example}\label{ex-positiveZelim}
Let $P_0 = \mathbb{Q}[a, b]$ and $P = P_0[x, y, z, w]$ be graded by $W = (0, 0, 2, 2, 2, 1)$. 
Consider the polynomials 
$g_1 = ax + by + z - aw^2$, $g_2 = (1-a)x - bz - 2w^2$, and $g_3 = y - 3bz$. 
Let $I = \langle g_1, g_2, g_3 \rangle$ and let $Z = (x, y)$. 

The ideal $I$ is $W$-homogeneous and contained in $\langle x, y, z, w \rangle$, 
so that $P/I$ is a positive $P_0$-algebra. 
Let $\sigma$ be an elimination term ordering for $Z$. 
The reduced $\sigma$-Gr\"obner basis of $I$ is $\{f_1, f_2, f_3\}$, where
\begin{align*}
f_1 &= x + 3b^2z - bz + z - (a + 2)w^2, \\
f_2 &= y - 3bz, \\
f_3 &= (3ab^2 - 3b^2 + ab + a - 1)z - (a^2 + a)w^2.
\end{align*}
The pair $(f_1, f_2)$ is coherently $Z$-separating for $I$.

Indeed, the $P_0$-linear parts of the generators satisfy the condition of Theorem~\ref{thm-existZsep} via:
\begin{equation*}
x = \Lin_{P_0,Z}(g_1) + \Lin_{P_0,Z}(g_2) - b\,\Lin_{P_0,Z}(g_3) \quad \text{and} 
\quad y = \Lin_{P_0,Z}(g_3).
\end{equation*}

Moreover, we obtain the re-embedding of~$I$ via the isomorphism
\begin{equation*}
\Phi \colon P/I \To \mathbb{Q}[a, b, z, w] / \langle (3ab^2 - 3b^2 + ab + a - 1)z - (a^2 + a)w^2 \rangle.
\end{equation*}
\end{example}

An extensive use of Theorem~\ref{thm-existZsep} produces several algorithms 
(see~\cite[Algorithms~3.4, 3.6, 3.7, 4.2,  and 4.6]{KR6})
which allow us to obtain best separating tuples of $W$-homogeneous polynomials, 
in the sense of Definition~\ref{def-edim}.

\medskip
\subsection{The Unimodular Matrix Problem}
\label{subsec-UMP}

In this subsection, we present a classical method for producing re-embeddings that cannot 
be obtained with the techniques described so far.

\begin{definition}\label{def-UMP}
Let $P$ be a polynomial ring over a field~$K$, and let $r,t$ be natural numbers such that 
$1 \le r < t$.

\begin{myenumerate}

\item A matrix $A \in \Mat_{r,t}(P)$ is said to be \textbf{unimodular} 
if its maximal minors generate the unit ideal.
 
\item A matrix $B \in \Mat_t(P)$ is said to solve the 
\textbf{unimodular matrix problem (UMP)} for~$A$
if $\det(B)=1$ and $A \cdot B = (I_r \mid 0)$, where $I_r$ denotes the identity matrix of size~$r$.

\end{myenumerate}
\end{definition}

\begin{myremark}\label{rem-Bequalinverse}
In the case where  $r = t$, the unimodularity of $A$ means that $A$ is invertible, hence 
$B = A^{-1}$.
\end{myremark}

\begin{myremark}\label{rem-UMC}
Given~$B$ which solves the UMP for~$A$, 
we have $A = (I_r \mid 0) \cdot B^{-1}$. In other words, the first~$r$ rows of
the matrices~$A$ and $B^{-1}$ coincide. For this reason, the matrix $B^{-1}$ is
called a {\bf unimodular matrix completion} of~$A$.
\end{myremark}

A good reference and an algorithm for solving the UMP can be found in~\cite{LS}.

We now state the main result of this subsection.

\begin{theorem}{\bf (Automorphism via UMP)}\label{thm-GoodIso}

Let $P = P_0[X]$ be a polynomial ring over~$P_0$ such that $P/I$ 
is a positive $P_0$-algebra, where $X=(x_1,\dots, x_s)$.
Let $k \le s$, let $G = (g_1,\dots, g_k)$ be a tuple of 
$W$-homogeneous polynomials in $I$ such that the matrix 
$\cmat_{P_0, X}(G)$ is unimodular, and let $X_k = (x_1,\dots, x_k)$.
\vskip -.01cm
\begin{myenumerate}
\item There exists a $P_0$-graded automorphism $\phi$ of~$P$ 
such that $\phi(G)$ is $X_k$-separating.

\item Let $\sigma$ be an elimination term ordering for $X_k$.
Then the reduced Gr\"obner basis of~$\langle \phi(G) \rangle$ is 
coherently $X_k$-separating.
\end{myenumerate}
\end{theorem}

\begin{proof} 
To prove claim (1),
let $B \in \Mat_s(P_0)$ solve the UMP for $\cmat_{P_0, X}(G)$, and 
let $\phi \colon P \To P$ 
be defined by  $\phi(a_i) = a_i$ for $i=1,\dots, m$, and  $\phi(X\tr) =  B\cdot X\tr$. 
From the invertibility of $B$ we deduce that $\phi$ is an automorphism.  

We have 
$G\tr = \cmat_{P_0, X}(G) X\tr + \widetilde{G}\tr$,
where $\widetilde{G} = (\widetilde{g}_1, \dots, \widetilde{g}_k)$ and
$\Lin_{P_0,X}(\widetilde g_j)=0$ for $j = 1,\dots, k$.
By assumption, we have $\cmat_{P_0, X}(G) \cdot B =   (I_k \mid 0)$,
hence we deduce the following equalities.
\begin{equation*}
\begin{aligned}
\phi(G\tr) 
&= \cmat_{P_0, X}(G) \phi(X\tr) + \phi(\widetilde{G}\tr) \\
&= \cmat_{P_0, X}(G)\cdot B\cdot X\tr + \phi(\widetilde{G}\tr) \\
&= (I_k \mid 0) \cdot X\tr + \phi(\widetilde{G}\tr)\\
&= X_k\tr + \phi(\widetilde{G}\tr).
\end{aligned}
\end{equation*}
Therefore, we have 
\begin{equation*}
\langle X_k \rangle_{P_0} = 
\langle \Lin_{P_0, X}(\phi(g_1)), \dots, \Lin_{P_0, X}(\phi(g_k)) \rangle_{P_0}.
\end{equation*}
It follows from Theorem~\ref{thm-existZsep} that $\phi(G)$ is $X_k$-separating.

Claim (2) follows from Theorem~\ref{thm-existZsep}(3).
\end{proof}

\begin{myremark}\label{rem-block}
Every matrix of the form $\cmat_{P_0, X}(G)$ is a  diagonal block matrix 
with one block for each $W\!$-degree of the elements of $G$.
Hence, the UMP can be split into several  mutually independent problems. 
\end{myremark}

We conclude this subsection with an example which illustrates the previous theorem and remark.
The actual computation is given in \cocoa-Example~\ref{coex-UMP}.

\goodbreak
\begin{example}\label{ex-UMP}
Let $P_0= \QQ[a, b]$, and $P=P_0 [x, y, z, w]$ graded by $W = (1, 1, 2, 2)$,
let $g_1 = (1+a^2)x  + a^3y$, $g_2 = az +(1-ab)w  -x^2$, let $G = (g_1, g_2)$, and
let $I=\langle G \rangle$.
The polynomials $g_1$, $g_2$ are not separating with respect to any indeterminate,
hence we have $\sepdim(P/I)=6$.
If we let  $X = (x, y, z, w)$, we have 
\begin{equation*}
\cmat_{P_0, X}(G) = 
\begin{bmatrix} 
1 +a^2 & a^3 & 0 & 0 \\
0 & 0 & a & 1 -ab
\end{bmatrix}
\end{equation*}

A matrix  which solves the UMP for $\cmat_{P_0, X}(G)$ is 
\begin{equation*}
B = \begin{bmatrix} 
-a^2 +1 & -a^3 &\ 0\ &\  0 \\
a & a^2 + 1 &\ 0\ &\  0\\
0 & 0 &\ b\  &\  ab -1\\
0 & 0 &\ 1\ &\ a
\end{bmatrix}
\quad \text{since \quad }
\cmat_{P_0, X}(G)\cdot B = \begin{bmatrix}  1 & 0  & 0 & 0\\ 0 & 0 & 1 & 0 \end{bmatrix}
\end{equation*}

Using the entries of 
\begin{equation*}
B\cdot \begin{bmatrix} x & y & z  & w \end{bmatrix}\tr =\begin{bmatrix}
-a^3y -a^2x +x  \\
a^2y +ax +y \\
abw +bz -w\\
aw +z
\end{bmatrix}
\end{equation*} 
we define the $\QQ$-algebra  automorphism $\phi \colon P \to P$ by
$\phi(a)=a,\quad  \phi(b) = b, \quad 
\phi(x)= -a^3y -a^2x +x ,\quad 
\phi(y)= a^2y +ax +y , \quad 
\phi(z)= abw +bz -w, \quad
\phi(w) = aw +z$, 
and get 
\begin{equation*}
\phi(g_1)= x, \quad   \phi(g_2) =  -a^6y^2 -2a^5xy -a^4x^2 +2a^3xy +2a^2x^2 -x^2 +z
\end{equation*}
The map $\phi$ induces a $\QQ$-algebra isomorphism 
$\Phi \colon P/I  \To P/\langle \phi(g_1), \phi(g_2) \rangle $.

The pair $(\phi(g_1), \phi(g_2))$ is $(x, z)$-separating. Setting $x =0$ in $\phi(g_2)$
produces the pair $(x,\  z - a^6y^2)$ which is coherently $(x, z)$-separating.

We use this pair to define a $\QQ$-algebra homomorphism 
$\psi \colon P \To \QQ[a, b, y, w]$ by $\psi(a) = a$, $\psi(b) =b$, $\psi(x) = 0$, $\psi(z) = a^6y^2$.
The map $\psi$  induces a $\QQ$-algebra  isomorphism  
$\Psi \colon P/\langle \phi(g_1), \phi(g_2) \rangle \to \QQ[a, b, y, w]$.

Altogether, we have a re-embedding $\Psi\circ \Phi \colon P/I \To \QQ[a, b, y, w]$.
Consequently, we have  $\edim(P/I)=  4 < \sepdim(P/I)=6$. Moreover, $P/I$ is 
a free $\mathbb{Q}$-algebra, though it cannot be obtained as a $Z$-separating algebra,
a situation already encountered in Example~\ref{ex-isomorphtoK[x]}.
\end{example}

\medskip
\subsection{Fibers of Positive \texorpdfstring{$P_0$}{P0}-algebras}
\label{subsec-Fibers}

In this subsection, we describe general and special fibers of positive $P_0$-algebras.
We keep using the notation introduced before; in particular, we let $P_0 = K[a_1, \dots, a_m]$ 
be trivially graded, let $P = P_0[x_1, \dots, x_s]$ be 
equipped with a $\mathbb{Z}$-grading defined by a positive weight vector
$W = (w_1, \dots, w_s)$, and let $I$ be an ideal in $P$ such that $P/I$ is 
a positive $P_0$-algebra. As usual, we denote $(x_1, \dots, x_s)$ by $X$.

\begin{definition}\label{def-fibers}
Let $P_0 = K[a_1, \dots, a_m]$, and let $\phi \colon P_0 \To P/I$ be the canonical 
injective $K$-algebra homomorphism. Furthermore, let $L = K(a_1, \dots, a_m)$ be the 
field of fractions of $P_0$, let $\Gamma = (c_1, \dots, c_m) \in K^m$, 
and let $\m_\Gamma = \langle a_1-c_1, \dots, a_m-c_m \rangle$ be the 
corresponding maximal ideal of~$P_0$.

\begin{myenumerate}
\item The ideal $I_L = I\, L[X]$ is called the \textbf{generic fiber ideal of $\phi$}, and
the $L$-algebra $L[X]/I_L$ is called the \textbf{generic fiber} of~$\phi$.

\item The ideal $I_\Gamma = I \cdot (P_0 / \m_\Gamma)[X] \cong I\, K[X]$ is called the
\textbf{(special) fiber ideal} of~$\phi$ over~$\Gamma$, and the $K$-algebra
$K[X]/I_\Gamma$ is called the \textbf{(special) fiber} of~$\phi$ over~$\Gamma$.
\end{myenumerate}
\end{definition}

\begin{myremark}\label{rem-maptofibers}
We know that $P = P_0[X]$, hence there are canonical surjective homomorphisms 
$\theta_L \colon P/I \To L[X]/I_L$ and $\theta_\Gamma \colon P/I \To K[X]/I_\Gamma$.
Moreover, if $Z$ is a $t$-subtuple of $X$ and $F = (f_1, \dots, f_t) \in P^t$ is 
$Z$-separating, then $\theta_L(F)$ and $\theta_\Gamma(F)$ are easily seen to be 
$Z$-separating as well.
\end{myremark}

The following important result represents a nice application 
of Proposition~\ref{prop-optimalgraded}.

\begin{proposition}{\bf (Optimal Re-Embeddings of Fibers)}\label{prop-optFibers}

Let $P_0$, $P$, and~$I$ be such that $P/I$ is a positive $P_0$-algebra, and 
let $\phi \colon P_0 \To P/I$ be the canonical injective $K$-algebra homomorphism.

\begin{myenumerate}
\item There exists a subtuple~$Z$ of~$X$ which yields an 
optimal separating re-embedding 
$\Phi_L \colon L[X] / I_L \To L[X {\setminus} Z] / (I_L \cap L[X {\setminus} Z])$
of the generic fiber ideal of $\phi$.

\item Let $\Gamma = (c_1, \dots, c_m) \in K^m$. There exists a subtuple~$Z$ of~$X$  
which yields an optimal separating re-em\-bed\-ding 
$\Phi_\Gamma \colon K[X] / I_\Gamma \To K[X{\setminus}Z] / (I_\Gamma \cap K[X{\setminus}Z])$
of the special fiber ideal of $\phi$ over $\Gamma$.
\end{myenumerate}
\end{proposition}

\begin{proof}
Both~$I_L$ and~$I_\Gamma$ are proper $W$-homogeneous ideals in positively graded rings. 
Hence the claim follows from Proposition~\ref{prop-optimalgraded}. 
\end{proof}

\begin{example}\label{ex-posFib}
We use the data from Example~\ref{ex-positiveZelim}.
Remember that we have $P/I \cong Q/J$ where 
$Q = \mathbb{Q}[a, b, z, w]$ and $J = \langle (ab^2 -\tfrac{1}{3}ab -b^2 
+\tfrac{1}{3}a -\tfrac{1}{3} ) z - \tfrac{1}{3}(a^2+a)w^2 \rangle$. 

Let $\psi \colon P_0 \To Q/J$ be the induced injective $\mathbb{Q}$-algebra homomorphism, 
let $L = \mathbb{Q}(a, b)$, let $\Gamma_1 = (1, 0)$, and let $\Gamma_2 = (1, -3)$.
The generic fiber over $\psi$ is $L[z, w]/ \langle z - \tfrac{h_1}{h_2} w^2 \rangle$ where 
$h_1 = - \tfrac{1}{3}(a^2+a)$ and $h_2 = ab^2 -\tfrac{1}{3}ab -b^2 +\tfrac{1}{3}a -\tfrac{1}{3}$,
hence it is isomorphic to $L[w]$. In other words, the generic fiber of $\psi$ is a free $L$-algebra
(see Definition~\ref{def-edim}\,(8)).

The special fiber of $\psi$ over $\Gamma_1$ is $\mathbb{Q}[z, w]/ \langle w^2 \rangle$ 
while the special fiber of $\psi$ over $\Gamma_2$ is 
$\mathbb{Q}[z, w] / \langle z - \tfrac{2}{3}w^2 \rangle \cong \mathbb{Q}[w]$.
According to Proposition~\ref{prop-optFibers}, both special fiber ideals have an optimal 
separating re-embedding.
In particular, the special fiber ideal of $\phi$ over $\Gamma_2$ turns out to be a 
free $\mathbb{Q}$-algebra.
\end{example}

Next, we discuss fibers in more detail. It is worth noting that, 
when we introduced the linear part $\Lin(f)$ and the cotangent space $\Cot$, 
the maximal ideal was omitted from the notation since it was intended 
to be $\langle x_1, \dots, x_n\rangle$. Here, we need to specify the maximal 
ideal explicitly, as different maximal ideals will appear within the same formula.

\begin{lemma}\label{lem-Lin=Lin}
Let $P_0$, $P$, and~$I$ be such that $P/I$ is a positive $P_0$-algebra, 
and let $\Gamma = (c_1, \dots, c_m) \in K^m$. 
Let $\alpha \colon P \To K[X]$ be the $K$-algebra homomorphism defined by 
$\alpha(a_i)=c_i$ for $i=1,\dots, m$ and $\alpha(x_j)=x_j$ for $j=1,\dots, s$, 
and let $\M_\Gamma$ be the maximal ideal of~$P$ given by 
$\M_\Gamma =\langle a_1-c_1, \dots, a_m-c_m,\, x_1, \dots, x_k \rangle$. 
Then the following claims hold.
\begin{myenumerate}
\item There is an isomorphism 
$ \alpha_\Gamma \colon \Lin_{\M_\Gamma}(I) \cong \Lin_{\langle X \rangle}(I_\Gamma)$
of $K$-vector spaces induced by the homomorphism~$\alpha$.

\item We have $\dim_K ( \Cot_{\langle X\rangle}(K[X]/I_\Gamma) ) =
\dim_K (\Cot_{\M_\Gamma}(P/I) ) - m$.
\end{myenumerate}
\end{lemma}

\begin{proof}
To prove (1) we use the fact that every polynomial $f \in I$ can be written as
$f = h_1 x_1 + \cdots + h_s x_s + g$ with $h_1, \dots, h_s \in P_0$ 
and $g \in \langle X \rangle^2$. Consequently, the $\M_\Gamma$-linear part of $f$ is
\begin{equation*}
\Lin_{\M_\Gamma}\big( \sum_{i=1}^s h_ix_i \big) 
= \sum_{i=1}^s h_i(c_1,\dots,c_m)x_i.
\end{equation*} 
To obtain the fiber ideal~$I_\Gamma$, we have to substitute $a_i \mapsto c_i$ for $i=1,\dots,m$
in the polynomials in~$I$. The $\langle X \rangle$-linear parts of the resulting polynomials are precisely the polynomials
$\Lin_{\M_\Gamma}(f)$, viewed as polynomials in $K[X]$. This observation implies claim~(1).

Next, we prove~(2) using  $s = n - m$ and   
$\dim_K ( \Cot_{\M_\Gamma}(P/I)) = n - \dim_K ( \Lin_{\M_\Gamma}(I) )$
\begin{align*}
\dim_K ( \Cot_{\langle X \rangle} (K[X]/I_\Gamma) ) 
& = n - m - \dim_K (\Lin_{\langle X \rangle}(I_\Gamma)) \\
& =  n - m - \dim_K ( \Lin_{\M_\Gamma}(I) )\\ 
& = \dim_K ( \Cot_{\M_\Gamma}(P/I)) - m.
\end{align*}
where the second equality follows from~(1).
\end{proof}

\begin{theorem}{\bf (Regularity and Freeness of Fibers)}\label{thm-fibers}

Let $P_0$, $P$, and~$I$ be such that $P/I$ is a positive $P_0$-algebra,
let $\Gamma = (c_1, \dots, c_m) \in K^m$.  Let $\phi \colon P_0 \To P/I$ be the 
canonical homomorphism, let $L[X]/I_L$ be the generic fiber of~$\phi$, 
and let $K[X]/I_\Gamma$ be the special fiber of~$\phi$ over~$\Gamma$.
\begin{myenumerate}
\item If the local ring $(L[X]/I_L)_{\langle X \rangle/I_L}$ is regular, 
then $L[X]/I_L$ is a free $L$-algebra.

\item If the local ring $(K[X]/I_\Gamma)_{\langle X \rangle/I_\Gamma}$ is regular, 
then $K[X]/I_\Gamma$ is a free $K$-algebra.

\item Assume that the local ring $(P/I)_{\M_\Gamma/I}$ is regular.

\smallskip
\begin{myenumerate}
\item[(a)] The local ring $(K[X]/I_\Gamma)_{\langle X \rangle/I_\Gamma}$ of the fiber of~$\phi$
over~$\Gamma$ is regular.

\item[(b)] We have $\dim(K[X]/I_\Gamma) = \dim((P/I)_{\M_\Gamma/I}) - m$.
\end{myenumerate}

\end{myenumerate}
\end{theorem}

\begin{proof}
To prove claim (1), we note that $L[X]/I_L$ is positively graded,
hence we deduce from Proposition~\ref{prop-optimalgraded} that, for 
$t = \dim_L(\Lin_{\langle X \rangle}(I_L))$,
there exist a $t$-tuple of indeterminates and a $t$-tuple 
of $W$-homogeneous polynomials defining an optimal separating re-embedding 
of $I_L$.
Consequently, the conclusion follows from Theorem~\ref{thm-checkopt}\,(2).

Claim (2) follows in the same way, so let us prove (3).
Let $S=(K[X]/I_\Gamma)_{\langle X \rangle/I_\Gamma}$.
We claim that there is the following chain of relations.
\begin{align*}
\dim((P/I)_{\M_\Gamma/I} ) - m & \le \dim(S)  \le
\dim_K ( \Cot_{\langle X \rangle}(K[X]/I_\Gamma) ) \\
& = \dim_K ( \Cot_{\M_\Gamma}(P/I) ) - m =\dim((P/I)_{\M_\Gamma/I} ) - m.
\end{align*}
The first inequality follows from the fact that~$S$ is obtained from $(P/I)_{\M_\Gamma/I}$ 
by modding out the ideal~$\langle a_1-c_1, \dots, a_m-c_m \rangle$ which is generated by~$m$ elements. 
The second inequality follows from the fact that the dimension of a local ring is bounded above by the 
dimension of the cotangent space at its maximal ideal (see~\cite[V.4.5]{Kun}). 
The first equality follows from Lemma~\ref{lem-Lin=Lin}(2), and the second equality is a consequence of the assumption that $(P/I)_{\M_\Gamma/I}$ is regular. It follows that all relations are in fact equalities. 
In particular, we have 
$\dim((K[X]/I_\Gamma)_{\langle X \rangle/I_\Gamma}) = \dim_K (\Cot_{\langle X \rangle} (K[X]/I_\Gamma))$, 
and hence~$(K[X]/I_\Gamma)_{\langle X \rangle/I_\Gamma}$ is a regular local ring, 
which proves (a). 
Moreover, we have $\dim((K[X]/I_\Gamma)_{\langle X \rangle/I_\Gamma}) = \dim((P/I)_{\M_\Gamma/I} ) - m$,
and $\dim((K[X]/I_\Gamma)_{\langle X \rangle/I_\Gamma}) = \dim(K[X]/I_\Gamma)$ follows 
from the fact that $K[X]/I_\Gamma$ is a positively graded algebra. This concludes the proof of~(b).
\end{proof}

\begin{myremark}\label{rem-WPS}
We have already seen that the generic fiber ideal and the special fiber ideals are $W$-graded.
Consequently, $L[X]/I_L$ and $K[X]/I_\Gamma$ are $W$-graded algebras, so they can be viewed 
as the homogeneous coordinate rings of weighted projective schemes (see~\cite{BR}).  
This remark inspires the proof of the following theorem.
\end{myremark}

\begin{theorem}{\bf (Connectedness)}\label{thm-connected}

Let $P_0 = K[a_1, \dots, a_m]$, let $P = P_0[X]$ 
with $X = (x_1, \dots, x_s)$, and $n =  m +s$, be 
equipped with a $\mathbb{Z}$-grading defined by a positive weight vector
$W = (w_1, \dots, w_s)$, and let $P/I$ be a positive $P_0$-algebra.

\begin{myenumerate}
\item The fiber $F_\Gamma$  is connected.

\item The scheme $\Spec(P/I)$ is connected.

\end{myenumerate}
\end{theorem}

\begin{proof}
Given any $K$-algebra~$S$, the canonical morphism 
$\Spec(S \otimes_K \overline{K}) \to \Spec(S)$ is surjective, by the faithfulness
of $K \to \overline{K}$, hence in  both claims we may assume that $K$ is algebraically closed.

To prove claim~(1), let $\Gamma=(c_1,\dots, c_m) \in K^m$, 
$\Gamma_0 = (c_1,\dots, c_m, 0, \dots, 0) \in K^n$,
$\Gamma_\pi=(c_1,\dots, c_m, \pi_1, \dots, \pi_s)$ be $K$-rational points of the 
fiber $F_\Gamma = \Spec(K[X]/I_\Gamma)$, and let  $\psi \colon K[X] \longrightarrow K[t]$ be the
$K$-algebra homomorphism defined by $\psi(x_i) = \pi_i\, t^{w_i}$ for $i=1, \dots, s$.
Since $\Ker(\psi)$ is a prime ideal, its associated curve $C_{\Gamma_0, \Gamma_\pi}$ is irreducible,
and from the preimages of $\langle t-1\rangle$ 
and $\langle t\rangle$ under~$\psi$ we see that~$\Gamma_0$ and $\Gamma_\pi$
are contained in~$C_{\Gamma_0, \Gamma_\pi}$. 

Next, we show that $C_{\Gamma_0, \Gamma_\pi} \subseteq F_\Gamma$. 
The ideal  $I_\Gamma$ is $W$-homogeneous, hence it suffices to show that for 
every $W$-homogeneous polynomial $f \in I_\Gamma$, we have $f \in \Ker(\psi)$.
Indeed, if $f \in I_\Gamma$ is a $W$-homogeneous polynomial of $W$-degree $d$, then 
 $f(\pi_1 t^{w_1}, \dots \pi_s t^{w_s}) = t^d\cdot f(\pi_1,\dots,\pi_s) =0$.

Next, let $\Gamma_\pi, \Gamma_{\pi'}$ be two closed points in $F_\Gamma$.
The irreducible curve which connects~$\Gamma_\pi$ and~$\Gamma_0$
is contained in an irreducible component of~$F_\Gamma$. For the point $\Gamma_{\pi'}$,
the corresponding curve also connects it to~$\Gamma_0$. Hence all irreducible 
components of~$F_\Gamma$ contain~$\Gamma_0$, and thus~$F_\Gamma$ is connected.

Finally, we prove (2). To show connectedness, we  show
that any two closed points $\Gamma_\pi,  \Delta_\rho$ of $\Spec(P/I)$ are in the same 
connected component. 
By~(1), $\Gamma_\pi$ is in the same connected component as the  point 
$\Gamma_0 = (c_1,\dots, c_m, 0, \dots, 0)$ in its fiber. 
Similarly, $\Delta_\rho$ is in the same connected component
as the  point $\Delta_0 = (d_1,\dots, d_m, 0, \dots, 0)$ in its fiber. 
The points  $\Gamma=(c_1,\dots, c_m)$ and $\Delta = (d_1, \dots, d_m)$ lie on the 
affine space $\mathbb A_K^m$, hence they can be connected by a straight line.
The map $\zeta \colon \mathbb{A}_K^m \To \Spec(P/I)$ which is induced by the 
canonical $K$-algebra epimorphism 
$P/I \longrightarrow P/\langle x_1,\dots, x_s \rangle \cong P_0$
allows us to map the straight line to an irreducible
curve which contains both $\Gamma_0$ and $\Delta_0$. 
Consequently, $\Gamma_\pi$ and  $\Delta_\rho$
lie on the same connected component of $\Spec(P/I)$, which concludes the proof.
\end{proof}

\medskip
\subsection{Freeness of Regular Positive \texorpdfstring{$P_0$}{P0}-Algebras}
\label{subsec-Freeness of Regular Positive P0-Algebras}

The main result in this subsection is Theorem~\ref{thm-Free}.
To reach the target, we need  a few preliminary results.
In the following proposition we collect several classical results in commutative
algebra which will be used in the remaining part of this subsection. The reader can find detailed
proofs in standard commutative algebra textbooks such as~\cite{Ei, Kun, Mats}.

\begin{proposition}\label{prop-standard}
Let $P/I$ be an affine $K$-algebra, and $\overline{K}$ the algebraic closure of~$K$.
\begin{myenumerate}

\item If $K = \overline{K}$ and for every pair of closed points in $\Spec(P/I)$ there exists 
an irreducible curve which contains both, then $\Spec(P/I)$ is connected.

\item If $P/I$ is regular and $\Spec(P/I)$ is connected, then $P/I$ is a regular integral domain.

\item If $K$ is perfect,  $P/I$ is regular, and $\Spec((P/I) \otimes_K \overline{K})$ is connected, 
then the ring $(P/I)\otimes_K \overline{K}$ is a regular integral domain. 
Moreover, we have the equality $\dim(P/I) = \dim((P/I)\otimes_K \overline{K})$.
\end{myenumerate}
\end{proposition}

\begin{proof}[Sketch of proof]
Claim (1) follows from the fact that if $K = \overline{K}$, then closed points are dense, and if each 
pair of them lies on a connected curve, then all of them are in the same connected component.

The regularity of $P/I$ implies that $P/I$ is normal, hence a finite direct product of integral domains. 
Therefore, the connectedness of $\Spec(P/I)$ implies that $P/I$ is an integral domain, which proves (2).

To prove claim (3), we note that if $K$ is perfect and $R$ is regular, 
then $P/I$ is geometrically regular over $K$,
hence $(P/I) \otimes_K \overline{K}$ is regular, and the conclusion follows from (2). The last claim follows from the faithfulness  of $P/I \to (P/I) \otimes_K \overline{K}$.
\end{proof}

\begin{lemma}\label{lem-unitJ}
Let $K$  be a perfect field, let $P_0 = K[a_1, \dots, a_m]$, and let $P = P_0[X]$ 
with $X = (x_1, \dots, x_s)$ and $n = m + s$. 
Let $P/I$ be a regular positive $P_0$-algebra, and let $G=\{g_1,\dots, g_r\}$ be 
a $W$-homogeneous system 
of generators of $I$, ordered in non-decreasing $W$-degree. Let $k = n - \dim(P/I)$,
let $U$ be the $P_0$-module generated by the rows of $\cmat_{P_0,X}(G)$,
and let $J$ be the ideal generated by the $k$-minors of $\cmat_{P_0, X}(G)$. 
Assume that the indeterminates $X$ are also ordered  by non-decreasing $W$-degree.

\begin{myenumerate}
\item  The module $U$ is  the direct sum of the  modules generated by
the rows of~$\cmat_{P_0,X}(G)$ which are coefficients of  $P_0$-linear forms of the  same $W\!$-degree.

\item The matrix $\cmat_{P_0, X}(G)$ is a  block diagonal matrix.

\item We have $\rk(\cmat_{P_0, X}(G )) = k$.

\item  We have $J = \langle 1 \rangle$.
\end{myenumerate}
\end{lemma}

\begin{proof}
To prove (1), it suffices to note that the module $\langle \Lin_{P_0, X}(G)\rangle_{P_0}$ is the direct sum of the 
modules generated by $P_0$-linear forms of elements of $G$ with the same $W\!$-degree. The reason is that 
 their support contains only indeterminates of the same $W\!$-degree.
 
Since claim (2) follows directly from (1) and the ordering of the indeterminates $X$, let us prove (3).
Let $\overline{K}$ be the algebraic closure of~$K$, let $\overline{P}_0 = \overline{K}[a_1, \dots, a_m]$ and
$\overline{P} = \overline{P}_0[X]$ where $X = (x_1, \dots, x_s)$.
We know from Theorem~\ref{thm-connected} that
$\Spec(P/I \otimes_K \overline{K})$ is connected, hence Proposition~\ref{prop-standard}(3) implies that 
$\overline{P} / I\, \overline{P}$ is a regular integral domain, and  that 
$\dim(P/I) = \dim(\overline{P}/I\, \overline{P})$. Consequently, 
 for every $\Gamma = (c_1,\dots, c_m) \in \overline{K}^m$ and  the corresponding 
 maximal ideal $\M_\Gamma =  \langle a_1-c_1,\dots,a_m-c_m,x_1,\dots,x_s\rangle$ in~$\overline{P}$, 
 we have
\begin{equation*}
k= n - \dim(P/I) = n - \dim(\overline{P}/I\, \overline{P}) = n -\dim ((\overline{P}/I\, \overline{P})_{\M_\Gamma/I\overline{P}}).
\end{equation*}
The local ring $(\overline{P}/I\,\overline{P})_{\M_\Gamma / I\,\overline{P}}$ is a regular 
integral domain, hence the Jacobian criterion (see~\cite{Kun}, VI.1.5) implies that 
the rank of the Jacobian matrix of $G$ at~$(c_1, \dots, c_m, 0, \dots, 0)$  
is $k$ for every~$\Gamma$.
Moreover, we note that the Jacobian matrix of $G$ evaluated at $(c_1, \dots, c_m, 0, \dots, 0)$ 
is equal to the matrix
$\cmat_{P_0,X}(G)_{\Gamma}$, i.e., the matrix $\cmat_{P_0,X}(G)$ evaluated at $\Gamma$,
because the non-linear terms in $X$ vanish under the partial derivatives evaluated at $X=0$. 
Hence all $(k+1)$-minors of $\cmat_{P_0,X}(G)$ vanish identically on $\overline{K}^m$, while at least one 
$k$-minor is non-zero at each point. Therefore, $\rk(\cmat_{P_0,X}(G))=k$, which proves (3).

Consequently, $J$ has no zero in~$\overline K^m$, and by Hilbert's Nullstellensatz,
the ideal~$J$ is the unit ideal in $\overline{P}_0$. Since $J$ is generated by polynomials 
already contained in $P_0$, it must be the unit ideal in $P_0$ as well, which proves (4).
\end{proof}

Finally, we arrive at  the main theorem of this section.

\begin{theorem}{\bf (Freeness of Regular Non-Negatively Graded Algebras)}\label{thm-Free}

Let $K$ be a perfect field, let $P_0 = K[a_1, \dots, a_m]$, and let $P = P_0[X]$ 
with $X = (x_1, \dots, x_s)$ and $n = m + s$. 
Let $I$ be an ideal of $P$ such that $P/I$ is a positive $P_0$-algebra, and let $k = n - \dim(P/I)$.
If $P/I$ is a regular ring, then the following claims hold.
\begin{myenumerate}
\item The ideal~$I$ is generated by a $W$-homogeneous regular sequence
$F = (f_1, \dots, f_k)$.

\item The matrix $\cmat_{P_0,X}(F)$ is unimodular.

\item The ring $P/I$ is isomorphic to a polynomial ring in $\dim(P/I)$ indeterminates 
and thus it is a free $K$-algebra.

\end{myenumerate}
\end{theorem}

\begin{proof}
Let $G=\{g_1,\dots, g_r\}$ be a $W$-homogeneous system of generators of $I$, 
ordered in non-decreasing $W$-degree, and let $U$ be the $P_0$-submodule of $P_0^s$ 
generated by the vectors $\cvec_{P_0, X}(g_i)$ for $i = 1,\dots, r$, i.e., the rows of $\cmat_{P_0,X}(G)$.
Using~\cite[Proposition~20.8]{Ei} and Lemma~\ref{lem-unitJ}, we conclude that $U$ is a locally free $P_0$-module of rank $k$. 
Then the Quillen-Suslin theorem implies that $U$ is a free $P_0$-module of rank $k$. 

Using claim (1) of Lemma~\ref{lem-unitJ}, we can choose a free $P_0$-basis $\{v_1,\dots, v_k\}$ of $U$ where 
$v_i = \sum_{j=1}^r u_{ij}\cdot \cvec_{P_0, X}(g_j)$ with $u_{ij}\in P_0$ for $i=1,\dots,k$, and such that for each $v_i$, 
the corresponding sum involves only $\cvec_{P_0, X}(g_j)$ with $g_j$ of the same $W$-degree. 
Thus, the polynomials $f_i= \sum_{j=1}^r u_{ij}\, g_j \in I$ are $W$-homogeneous and, 
letting $F = (f_1, \dots, f_k)$,  the $W$-homogeneous ideal $\langle F \rangle$ is 
contained in $I$ by construction.

Moreover, we have $\cvec_{P_0, X}(f_i) = v_i$ for $i = 1,\dots, k$, hence these vectors 
form a $P_0$-basis of the row space of $\cmat_{P_0, X}(F)$.
Consequently, $\cmat_{P_0, X}(F)$ is unimodular, and the Jacobian criterion implies 
that $P/\langle F \rangle$ is a regular $W$-homogeneous complete intersection 
of dimension $n-k = \dim(P/I)$.
Since the local rings of this complete intersection are integral domains (cf.~\cite[V.5.15.a]{Kun}),
the corresponding local rings of~$P/I$, which are residue class rings of the same
dimension, must coincide with them. By the local-global principle (cf.~\cite[IV.1.4)]{Kun}, 
it follows that $I = \langle f_1,\dots, f_k\rangle$. This completes the proof of claims (1) and (2).

To show (3), we use the unimodularity of $\cmat_{P_0, X}(F)$ and Theorem~\ref{thm-GoodIso} applied to $F$. 
In particular, we deduce that there exists a $P_0$-graded automorphism $\phi$ of $P$ such that, if~$\sigma$ 
is an elimination term ordering for $X_k = (x_1, \dots, x_k)$, then the reduced Gr\"obner basis of $\langle \phi(F) \rangle$  
is coherently $X_k$-separating. From Theorem~\ref{thm-GoodIso}(1) we obtain an
isomorphism $P/I \cong P/\langle \phi(I) \rangle$, and hence Theorem~\ref{thm-checkopt}\,(2)
yields an isomorphism 
\begin{equation*}
P/I \cong K[a_1, \dots, a_m, x_{k+1}, \dots, x_s],
\end{equation*} 
where the number of these remaining indeterminates is $m + s - k = n - k = \dim(P/I)$, 
and the proof is complete.
\end{proof}

We conclude this subsection by noting that 
the existence of a non-negative grading for which~$I$
is a $W$-homogeneous ideal is a crucial hypothesis in the theorem. 
The next example demonstrates this strikingly.

\begin{example}{\bf (The Circle)}\label{ex-circle}

Let $P = \QQ[x,y]$, and let $I = \langle x^2+y^2-2x \rangle$.
The ring $P/I$ is regular and a (global) complete intersection.
However, the ideal~$I$ is not homogeneous with respect to any non-trivial non-negative $\ZZ$-grading.
\begin{myenumerate}
\item There is no isomorphism between
$P/I$ and a univariate polynomial ring over~$\QQ$. To verify this, it suffices to note
that  $x^2+y^2-2x= x(x-2) +y^2$, hence we have  $ \bar{x} ( 2 -\bar{x}) = \bar{y}^2$  in $P/I$.
On the other hand,   $\bar{x}$, $2-\bar{x}$, and $\bar{y}$ are irreducible, but not prime elements in~$P/I$. 
Therefore, $P/I$ is not a factorial domain.

\item  It is also interesting to note that, if we enlarge the base field to $K=\QQ(i)$,
there is still no isomorphism between $P/I$ and a univariate polynomial ring $K[t]$, but for a different reason.
In this case the $\QQ$-algebra isomorphism $\alpha \colon K[x,y] \longrightarrow K[u,v]$ defined by 
$\alpha(x) = {\tfrac{u+v}{2}} +1$, and $\alpha(y) = \tfrac{u-v}{2i}$ identifies~$I$ with
$\langle uv-1\rangle$. It induces an isomorphism $K[x,y]/I \cong K[u,u^{-1}]$, which is a factorial
domain, but $K[u,u^{-1}]$ is not isomorphic to a univariate polynomial ring over~$K$,
because it has too many invertible elements.
\end{myenumerate}
\end{example}

\newpage
\section{Special Border Basis Schemes}
\label{sec-SpecialBBS}

In this section we analyze several types of border basis schemes.

\medskip
\subsection{Homogeneous and MaxDeg Border Basis Schemes}
\label{subsec-HomogeneousBBS}

If we restrict our attention to zero-dimensional ideals which have 
an $\mathcal{O}$-border basis and are homogeneous
with respect to the standard grading, we obtain the following
subscheme of the border basis scheme.

\begin{definition}\label{def-homgBBS}
Let $\OO = \{t_1,\dots,  t_\mu\}$ be an order ideal with border 
$\partial\OO = \{b_1,\dots,  b_\nu\}$, let $I(\BO) \subseteq K[C]$ be the ideal defining the $\OO$-border basis scheme $\BO$, 
and let $K[C_0]$ be the K-subalgebra of K[C] generated by 
the indeterminates of total arrow degree 0.

\begin{myenumerate}
\item For  $j=1,\dots,\nu$, let 
$h_j = b_j - \sum_{\substack{i=1,\dots,\mu \\ \deg(t_i)=\deg(b_j)}} c_{ij}t_i \in K[C,X]$.

The set 
$\mathsf{H}=\{h_1,\dots, h_\nu\}$ 
is called the \textbf{generic homogeneous 
$\mathcal{O}$-border prebasis}.

\item For $k=1,\dots,n$, let $\mathcal{H}^{\mathstrut}_k \in\Mat_{\mu}(K[c_{ij}])$ be the
$k^{\rm th}$ formal multiplication matrix associated to~$\mathsf{H}$ i.e.\ 
 the matrix obtained from the generic multiplication matrix $\mathcal{M}_k$ by 
setting $c_{ij} \mapsto 0$ whenever $\deg(t_i)\ne\deg(b_j)$.
Then,  $\mathcal{H}_k$ is called the \textbf{$k^{\rm th}$ generic 
homogeneous multiplication matrix}
with respect to~$\mathcal{O}$.

\item Let $I(\BOhom)$ be the ideal generated by the entries of
the matrices $\mathcal{H}_k \mathcal{H}_\ell -\mathcal{H}_\ell \mathcal{H}_k$ with
$1\le k<\ell\le n$.
The affine scheme $\BOhom \subseteq \mathbb{A}^{\mu\nu}_K$,
defined by $I(\BOhom)$ is called the \textbf{homogeneous $\OO$-border basis scheme}.
Its coordinate ring $K[C_0]/I(\BOhom)$ is denoted by $B_\OO\hom$.
\end{myenumerate}
\end{definition}

\begin{myremark}\label{rem-deg0}
From the definitions we deduce the following facts.
\begin{myenumerate}
\item The ideal $I(\BOhom)$ is contained in $K[C_0]$.

\item The scheme $\BOhom$  coincides with 
$\mathbb{B}_{\mathcal{O}} \cap  
\mathcal{Z}(\langle c_{ij}\mid \deg(t_i)\ne\deg(b_j)\rangle)$. Consequently,
there is a canonical embedding $\BOhom \longmono \BO$.
\end{myenumerate}
\end{myremark}

\bigskip
The following definition introduces a special kind of border basis schemes.

\begin{definition}\label{def-maxdegBBS}
Let $\OO = \{t_1,\dots,  t_\mu\}$ be an order ideal with border $\partial\OO = 
\{b_1,\dots,  b_\nu\}$, and let $I(\BO) \subseteq K[C]$ be the ideal defining the $\OO$-border
basis scheme $\BO$.  We say that~$\OO$ has a \textbf{MaxDeg border} (or simply that $\OO$ is \textbf{MaxDeg}) and  $\BO$ is called a \textbf{MaxDeg border basis scheme},
if the following equivalent conditions are satisfied.

\begin{myenumerate}
\item We have $\deg(b_j) \ge \deg(t_i)$ for $i=1, \dots, \mu$ and $j=1, \dots, \nu$.

\item The total arrow grading on~$K[C]$ is non-negative.
\end{myenumerate}
\end{definition}

One of the main features of MaxDeg border basis schemes is the following.

\begin{theorem}{\bf (Maxdeg and Generic Homogeneous Multiplication Matrices)}
\label{thm-homcommute}

Let $\OO$ be MaxDeg. 
Then, the generic homogeneous multiplication matrices commute.
\end{theorem}

\begin{proof}
For $i=1,\dots,n$, let $\mathcal{H}_i$ be the
generic homogeneous multiplication matrix with respect to~$\mathcal{O}$.
Let $k, \ell \in \{1, \dots, n\}$, and write $\mathcal{H}_k=(r_{ij})$ and $\mathcal{H}_\ell=(s_{ij})$.
We examine the entry $\sum_{\gamma=1}^\mu r_{\alpha\gamma}s_{\gamma\beta}$ 
of the product $\mathcal{H}_k \mathcal{H}_\ell$ at position $(\alpha,\beta)$.
To have $r_{\alpha\gamma} s_{\gamma\beta} \ne 0$ we need:

\begin{myenumerate}
\item[(a)] $r_{\alpha \gamma}\ne 0$, hence the term~$t_\alpha$ is contained in the support of
the representation of $x_k\,t_\gamma$ with respect to the basis~$\mathcal{O}$.

\item[(b)] $s_{\gamma\beta} \ne 0$, hence the term~$t_\gamma$ is contained in the support of
the representation of $x_\ell\,t_\beta$ with respect to the basis~$\mathcal{O}$.
\end{myenumerate}

By Proposition~\ref{prop-arrowhomog}\,(2), the ideal $\langle \mathsf{G} \rangle =
\langle g_1,\dots, g_\nu \rangle$ is $\overline{W}$-homogeneous, hence 
case~(a) implies $\deg(t_\alpha)=\deg(x_k  t_\gamma) > \deg(t_\gamma)$,
and case~(b) implies $\deg(t_\gamma)=\deg(x_\ell t_\beta) > \deg(t_\beta)$.
Thus, we have $\deg(t_\alpha)=\deg(x_k  x_\ell t_\beta) > \deg(x_\ell t_\beta)$.

The assumption that $\mathcal{O}$ is MaxDeg implies 
$x_\ell t_\beta \in \mathcal{O}$. From case (b), we deduce that the only non-zero term in the sum occurs when $t_\gamma = x_\ell t_\beta$. Since $x_\ell t_\beta \in \OO$, there is a unique index, say $\rho$, such that $t_{\rho} = x_\ell t_\beta$. This forces $s_{\rho \beta} = 1$ and $s_{\gamma \beta} = 0$ for all $\gamma \ne \rho$. 
Consequently, the sum collapses to a single term:
\begin{equation*}
\sum_{\gamma=1}^\mu r_{\alpha\gamma}s_{\gamma\beta} = r_{\alpha \rho}
\end{equation*}

If $r_{\alpha \rho} \ne 0$, there are two possibilities.
Either we have $t_\alpha = x_k t_{\rho} = x_k x_\ell t_\beta$ and thus $r_{\alpha \rho} = 1$,
or there exists an index $j$ such that $t_\alpha \in \Supp(b_j - g_j)$ and 
$b_j = x_k t_{\rho} =x_k x_\ell t_\beta$. In this case, $r_{\alpha \rho} = c_{\alpha j}$.

Using the same arguments to examine the entry 
$\sum_{\gamma=1}^\mu s_{\alpha\gamma} r_{\gamma\beta}$ 
of the product $\mathcal{H}_\ell \mathcal{H}_k$ at position $(\alpha,\beta)$,
we find that it evaluates to $s_{\alpha \delta}$, where $\delta$ is the unique index such that $t_{\delta} = x_k t_\beta$.
Again, if $s_{\alpha \delta}\ne 0$, either we have 
$t_\alpha = x_\ell t_{\delta} = x_\ell x_k t_\beta$ and thus $s_{\alpha \delta} = 1$,
or there exists an index $j$ such that $t_\alpha \in \Supp(b_j - g_j)$ and 
$b_j=x_\ell t_{\delta} =x_\ell x_k t_\beta$. In this case, $s_{\alpha \delta} = c_{\alpha j}$.

When we examine $\mathcal{H}_\ell \mathcal{H}_k$, 
the conditions on $t_\alpha$ and $b_j$ are that they must be equal to $x_\ell x_k t_\beta$. 
The equality $x_k x_\ell = x_\ell x_k$ 
shows that the conditions coincide; hence, we obtain the desired equality  
$\mathcal{H}_k \mathcal{H}_\ell =\mathcal{H}_\ell \mathcal{H}_k$.
\end{proof}

\begin{proposition}\label{prop-homcommute}
Let $\OO$ have  a MaxDeg border.

\begin{myenumerate}
\item We have $I(\BO) \cap K[\Cnull] = \langle 0 \rangle$.

\item The ring $K[\Cnull]$ is the homogeneous component of degree
zero of the non-negatively graded $K$-algebra $B_\OO = K[C]/I(\BO)$.

 \item Let $d = \max\{\deg(t) \mid t \in \OO\}$, and let 
  \begin{equation*}
   r = \# \{t \in \OO \mid \deg(t) = d\}, \quad s = \# \{b \in \partial\OO \mid \deg(b) = d\}. 
   \end{equation*}
    Then $K[C_0]$ is a polynomial ring in $r{\cdot}s$ indeterminates, and the 
    homogeneous border basis scheme $\BOhom$ is an affine space of dimension $r{\cdot}s$.
    
    \item The ring $B_\OO$ is a positive $K[C_0]$-algebra.

\item The scheme $\BO$ is connected.
\end{myenumerate}
\end{proposition}

\begin{proof} 
To prove (1), we first observe that 
$ I(\BO)_0 = I(\BO) \cap K[C_0] = I(\BOhom)$ by definition.
The ideal $I(\BOhom)$ is generated by the entries of the commutator 
matrices  $\mathcal{H}_k \mathcal{H}_\ell -\mathcal{H}_\ell \mathcal{H}_k$ with
$1\le k<\ell\le n$. Since $\OO$ is MaxDeg, from Theorem~\ref{thm-homcommute} we know that 
these matrices are zero-matrices, implying $I(\BOhom) = \langle 0 \rangle$. 
Consequently, all generators of~$I(\BO)$ have positive arrow degrees.

The second claim is an immediate consequence of (1).

To prove (3), notice that  we have $\deg(b_j) \ge \deg(t_i)$
for all $i,j$.  Under the MaxDeg assumption, 
the only way to get $\deg_W(c_{ij})= 0$ is from
$\deg(t_i)=\deg(b_j)=d$, and the claim follows.

Claim (4) follows from the definition  of MaxDeg and claim (1).

Claim (5) follows from (4) and Theorem~\ref{thm-connected}.
\end{proof}

The following example shows why the assumption in Proposition~\ref{prop-homcommute} 
that $\OO$ is MaxDeg is essential.

\begin{example}\label{ex-madeg-essence}
For the actual computation see  \cocoa-Example~\ref{coex-maxdeg-essential}. 
Let $P = \QQ[x,y]$ and let $\OO=\{1,\, y,\, x,\, y^2,\, xy,\, y^3 \}$. Its border is
$\partial\OO = \{x^2,\,  xy^2,\,  x^2y,\,  y^4, \, xy^3\}$.  We have  $B_\OO = K[C]/I(B_\OO)$ 
where $K[C]$ is a polynomial ring with $6 \times 5=30$ indeterminates all of them of 
non-negative total arrow degree, \textit{except} for~$c_{61}$ whose degree is 
\begin{equation*}
\deg_W(c_{61})= \deg(b_1) - \deg(t_6) = \deg(x^2) -\deg(y^3) =-1
\end{equation*}
We have $B_\OO\hom= \QQ[c_{41}, c_{51}, c_{62}, c_{63}]/\langle f \rangle$ where
$f = c_{41} -c_{63}  +c_{51}c_{62}$. On the other hand, we have  
$I(\BO) \cap K[\Cnull]= I(\BO)_0=\langle c_{41} -c_{63}  +c_{61}c_{64} +c_{51}c_{62} \rangle $.
Notice that $\deg_W(c_{61}) =-1$ and $\deg_W(c_{64}) =1$ so that $\deg_W(c_{61}c_{64})= 0$.

Consequently, this example shows that, without the assumption of MaxDeg, we may have 
$I(\BO) \cap K[\Cnull]  \ne \{0\}$ and $B_\OO\hom \ne (B_\OO)_0$.
\end{example}

We recall from Proposition~\ref{prop-irreducible}
that $\BO$ is irreducible for any $n$ and $\mu \le 7$.
Long ago, in the paper~\cite{Iar}, Iarrobino  proved  that Hilbert schemes need not be irreducible in general.
This result was further explored in~\cite[Example~5.6]{KR3}, where Iarrobino's example was explained 
using the framework of homogeneous border basis schemes.

Here is another example.

\begin{example}\label{ex-reducible}
Let $P=\QQ[x,y,z,w]$ and let $\OO \subset \mathbb{T}^4$ be an order ideal consisting of all terms of 
degree $\le 2$ and $10$ terms of degree $3$. Thus, $\mu = 1 + 4 + 10 + 10 = 25$. 
In this case, we have $d=3$ and $r=s=10$. By Proposition~\ref{prop-homcommute}, the scheme $\BOhom$, whose coordinate ring is $(B_\OO)_0$, is isomorphic to $\mathbb{A}^{100}_\QQ$. 

By Remark~\ref{rem-nmu},  the principal component of $\BO$ also has dimension $n \cdot \mu = 4 \cdot 25 = 100$. However, this component contains points representing non-homogeneous ideals, such as the distractions of the monomial ideal $\langle \partial\OO \rangle$. 

Since $\BOhom$ and the principal component of $\BO$ are distinct irreducible closed subschemes of the same dimension, it follows that $\BO$ must have at least two irreducible components passing through the origin. Consequently, the scheme $\BO$ is reducible, and the origin is a singular point of the border basis scheme.
\end{example}

The following result follows directly from Theorem~\ref{thm-Free}.
It is stated as  a theorem due to its relevance.

\begin{theorem}{\bf (Planar MaxDeg Border Basis Schemes)}\label{thm-bivariateaffinecells}

Let $K$ be a perfect field. Then every planar MaxDeg  border basis  scheme $\BO$
defined over~$K$ is an affine cell of dimension $2\mu$.
\end{theorem}

\begin{proof}
By Proposition~\ref{prop-irreducible}\,(2), the scheme  $\BO$ is regular and irreducible.
Hence it coincides with its principal component whose dimension 
is~$2\mu$ (see Remark~\ref{rem-nmu}). The conclusion follows from 
Theorem~\ref{thm-Free}\,(3).
\end{proof}

To fully appreciate the relevance of this result, we discuss an intriguing
example: a special planar border basis scheme which we will investigate in great detail.

\begin{center}
\begin{minipage}[c]{0.37\textwidth}
\begin{definition}\label{def-Lshape}
The order ideal $\OO = \{1,\, y,\, x,\, y^2,\, x^2\}$ in~$\TT^2$
is called the \textbf{L-shape order ideal}.
For the L-shape order ideal~$\OO$, the corresponding border basis scheme~$\BO$
is called the \textbf{L-shape border basis scheme}.
\end{definition}
\end{minipage}
\hspace{1.2cm} 
\begin{minipage}[c]{0.60\textwidth}
\centering
\refstepcounter{figure}\label{fig:L-shape}
\beginpicture
        \setcoordinatesystem units <0.7cm,0.7cm>
        \setplotarea x from 0 to 4, y from 0 to 3.5
        \axis left /
        \axis bottom /
        \arrow <2mm> [.2,.67] from 3.5 0 to 4 0
        \arrow <2mm> [.2,.67] from 0 3 to 0 3.5
        \put {$\scriptstyle x^i$} [lt] <0.5mm,0.8mm> at 4.1 0.1
        \put {$\scriptstyle y^j$} [rb] <1.7mm,0.7mm> at 0 3.6
        \put {$\bullet$} at 0 0
        \put {$\bullet$} at 1 0
        \put {$\bullet$} at 0 1
        \put {$\bullet$} at 0 2
        \put {$\bullet$} at 2 0
        \put {$\scriptstyle 1$} [lt] <-1mm,-1mm> at 0 0
        \put {$\scriptstyle t_1$} [rb] <-1.3mm,0.4mm> at 0 0
        \put {$\scriptstyle t_3$} [rb] <-1.3mm,0.4mm> at 1 0
        \put {$\scriptstyle t_2$} [rb] <-1.3mm,0mm> at 0 1
        \put {$\scriptstyle t_4$} [rb] <-1.3mm,0mm> at 0 2
        \put {$\scriptstyle t_5$} [lb] <0.8mm,0.4mm> at 2 0
        \put {$\scriptstyle b_1$} [rb] <4.3mm,0mm> at 1 1
        \put {$\scriptstyle b_2$} [lb] <0.8mm,0mm> at 0 3
        \put {$\scriptstyle b_3$} [lb] <0.8mm,0mm> at 1 2
        \put {$\scriptstyle b_4$} [lb] <0.8mm,0mm> at 2 1
        \put {$\scriptstyle b_5$} [lb] <0.8mm,0.7mm> at 3 0
        \put {$\times$} at 0 0
        \put {$\circ$} at 0 3
        \put {$\circ$} at 2 1
        \put {$\circ$} at 1 2
        \put {$\circ$} at 1 1
        \put {$\circ$} at 3 0
        
        \put {\small \textbf{Figure~\thefigure.} The L-shape order ideal} at 1.5 -1.2
        \put {\small and its border} at 1.5 -1.7
\endpicture
\end{minipage}
\end{center}

\begin{example}{\bf (The L-shape)}\label{ex-Lshape}

For the full computations see  \cocoa-Example~\ref{coex-L-shape}.\ 
In the case of the L-shape order ideal we have $\mu=\nu=5$, hence the corresponding border 
basis scheme is naturally embedded into the affine space $\AA_K^{25}$. 
Its vanishing ideal is generated by the 20 nonzero entries of the
commutator $\mathcal{M}_1 \mathcal{M}_2 - \mathcal{M}_2 \mathcal{M}_1$ of the two
generic multiplication matrices. As we know from Proposition~\ref{prop-arrowhomog}, 
the ideal $I(\BO)$ is homogeneous
with respect to the total arrow grading which assigns to the tuple of indeterminates
$C=(c_{11}, c_{12}, \dots, c_{55})$ the weights given by
\begin{equation*}
W = (2,\; 3,\; 3,\; 3,\; 3,\; 1,\; 2,\; 2,\; 2,\; 2,\; 1,\; 2,\; 2,\; 2,\; 2,\; 0,\; 1,\; 1,\; 1,\; 1,\; 0,\; 1,\; 1,\; 1,\; 1)
\end{equation*}
In particular, $\OO$ has a MaxDeg border and the total arrow grading is
non-negative. We deduce from Theorem~\ref{thm-bivariateaffinecells} that $\BO$ 
is an affine cell of dimension~10, in other words that its coordinate ring $B_\OO =
K[c_{11},\dots, c_{55}]/I(\BO)$ is isomorphic to a polynomial ring in 10 indeterminates.
End of the story? Not really, since  a natural challenge for computer algebra experts 
is to explicitly construct such an isomorphism.

It is doable, but reaching this conclusion requires a non-trivial sequence of steps. 
Our experience shows that a direct attack using Theorem~\ref{thm-Free} is feasible,
but computationally too costly. Therefore, we adopt an indirect approach which exploits the UMP-technique
suggested by the theorem, preceded by a suitable application
of the technique of $Z$-separating re-embeddings. It is worth noting that a significant amount of
\cocoa-code is required to achieve this result. We describe the process as a sequence of 
steps. For the computations, see  \cocoa-Example~\ref{coex-L-shape}.

Let $K$ be a perfect field,
and let $B_\OO =K[C]/I(\BO)$ be the coordinate ring of the L-shape border basis scheme,
where $C = (c_{11},\dots, c_{15},\dots, c_{51},\dots, c_{55})$.

\begin{myenumerate}
\item  We use~\cite[Remark~4.3 and R2 Algorithm~4.6]{KR6}
to compute \textit{all} best tuples of separating indeterminates, and find
36 such tuples, each consisting of 13 elements. Any of these tuples yields
a best $Z$-separating re-embedding of~$I(\BO)$. 
In other words, each tuple provides an isomorphism from~$B_\OO$ to a polynomial
ring in $25-13=12$ indeterminates modulo an ideal generated by two polynomials. 

\item 
Among these tuples, we select one that minimizes the cardinality of the union of its support.
Our choice is 
$Z = (c_{11},\, c_{12},\, c_{13},\, c_{14},\, c_{15},\, c_{23},\, c_{24},\, c_{25},\, 
c_{31},\, c_{32},\, c_{34},\, c_{44},\, c_{53})$
so that we obtain
$C \setminus Z = (c_{21},\, c_{22},\, c_{33},\, c_{35},\, c_{41},\, c_{42},\, c_{43},\, c_{45},\, c_{51},\, c_{52},\, c_{54},\, c_{55})$
whose weights are  $W = (1, \, 2, \, 2, \, 2, \, 0,\, 1,\, 1, \, 1,\, 0,\, 1,\, 1,\, 1)$

\item We let $S = \QQ[c_{21},\, c_{22},\, c_{33},\, c_{35},\, c_{41},\, c_{42},\, c_{43},\, c_{45},\, c_{51},\, c_{52},\, c_{54},\, c_{55}]$. 
The re-embedding of $I(\BO)$ is a graded isomorphism $\Phi \colon \QQ[C]/I(\BO) \To S/\langle f_1, f_2 \rangle$, where 
\begin{align*}
f_1 ={} & c_{21}c_{41}^2c_{51}^2 + c_{41}^2c_{43}c_{51}^2 + c_{41}c_{45}c_{51}^3 + c_{41}^2c_{42}c_{51} - c_{41}^3c_{52} - c_{41}^2c_{51}c_{54} \\
& + c_{21}c_{41}c_{51} + c_{45}c_{51}^2 - c_{41}c_{51}c_{55} + c_{41}c_{42} + c_{41}c_{54} + c_{21} - c_{43} \\[2mm]
f_2 ={} & -c_{21}^2c_{41}^3c_{51}^5 - 2c_{21}c_{41}^3c_{43}c_{51}^5 - c_{41}^3c_{43}^2c_{51}^5 - 2c_{21}c_{41}^2c_{45}c_{51}^6 - 2c_{41}^2c_{43}c_{45}c_{51}^6 \\
& - c_{41}c_{45}^2c_{51}^7 - c_{21}c_{41}^3c_{42}c_{51}^4 - c_{41}^3c_{42}c_{43}c_{51}^4 - c_{41}^2c_{42}c_{45}c_{51}^5 + c_{21}c_{41}^3c_{51}^4c_{54} \\
& + c_{41}^3c_{43}c_{51}^4c_{54} + c_{41}^2c_{45}c_{51}^5c_{54} - c_{41}^4c_{42}c_{51}^2c_{52} + c_{41}^5c_{51}c_{52}^2 + c_{41}^4c_{51}^2c_{52}c_{54} \\
& - c_{41}^2c_{42}c_{43}c_{51}^3 - c_{21}c_{41}^3c_{51}^2c_{52} + c_{41}^3c_{43}c_{51}^2c_{52} - c_{41}^2c_{45}c_{51}^3c_{52} \\
& - 2c_{21}c_{41}^2c_{51}^3c_{54} - c_{41}^2c_{43}c_{51}^3c_{54} - 2c_{41}c_{45}c_{51}^4c_{54} - c_{41}^2c_{42}c_{51}^3c_{55} \\
& + 2c_{41}^3c_{51}^2c_{52}c_{55} + c_{41}^2c_{51}^3c_{54}c_{55} + 2c_{21}c_{41}c_{43}c_{51}^3 + 3c_{41}c_{43}^2c_{51}^3 + 2c_{43}c_{45}c_{51}^4 \\
& - c_{41}^4c_{52}^2 - 2c_{41}^3c_{51}c_{52}c_{54} + 2c_{41}c_{43}c_{51}^3c_{55} + c_{41}c_{51}^3c_{55}^2 + c_{41}c_{42}c_{43}c_{51}^2 \\
& - c_{42}c_{45}c_{51}^3 + c_{21}c_{41}^2c_{51}c_{52} + 3c_{41}c_{45}c_{51}^2c_{52} - 3c_{41}c_{43}c_{51}^2c_{54} + c_{45}c_{51}^3c_{54} \\
& + c_{41}c_{42}c_{51}^2c_{55} - 3c_{41}^2c_{51}c_{52}c_{55} - 3c_{41}c_{51}^2c_{54}c_{55} + c_{22}c_{41}c_{51} - 2c_{33}c_{41}c_{51} \\
& + c_{21}^2c_{51}^2 - c_{35}c_{51}^2 - c_{43}^2c_{51}^2 + c_{41}^2c_{42}c_{52} + c_{41}^2c_{52}c_{54} + c_{41}c_{51}c_{54}^2 \\
& - 2c_{43}c_{51}^2c_{55} - c_{51}^2c_{55}^2 - c_{41}c_{43}c_{52} - c_{45}c_{51}c_{52} + 2c_{43}c_{51}c_{54} + c_{41}c_{52}c_{55} \\
& + 2c_{51}c_{54}c_{55} + c_{33} - c_{54}^2
\end{align*}
The isomorphism $\Phi$ is induced by a surjective homomorphism $\phi \colon K[C] \To S$ which
is non-trivial in the sense that the images of the indeterminates in~$C$
are polynomials in~$S$ whose supports have sizes
\begin{equation*}
55,\, 282,\, 298, \, 277, \, 193, \, 1, \, 1, \, 30, \, 62,\, 71,\, 9, \, 23, \, 1, \, 
51, \, 1, \, 1, \, 1,\, 1, \, 3,\, 1, \, 1, \, 1, \, 10, \, 1,\, 1
\end{equation*}
respectively.

\item Using $Z$-separating re-embeddings, this is the furthest we can proceed. However, we have another resource at our disposal: searching for an automorphism via the UMP-technique, as described in Theorem~\ref{thm-GoodIso}.

\item We note that 
\begin{align*}
C_0 &= ( c_{ij} \mid \deg_W(c_{ij}) = 1) = (c_{41},\, c_{51})\\
C_1 &= ( c_{ij} \mid \deg_W(c_{ij}) = 1) =(c_{21},  c_{42},  c_{43},  c_{45},  c_{52},  c_{54},  c_{55})\\
C_2 &= ( c_{ij} \mid \deg_W(c_{ij}) = 2) =(c_{22},  c_{33},  c_{35})
\end{align*}

Since $\deg_W(f_1) = 1$ and $\deg_W(f_2) = 2$, we can apply Remark~\ref{rem-block} 
to split the problem.  We obtain
\begin{align*}
\cmat_{K[C_0], C_1}(f_1)\tr
=
\begin{bmatrix}
c_{41}^2c_{51}^2 + c_{41}c_{51} + 1 \\
c_{41}^2c_{51} + c_{41} \\
c_{41}^2c_{51}^2 - 1 \\
c_{41}c_{51}^3 + c_{51}^2 \\
-c_{41}^3 \\
-c_{41}^2c_{51} + c_{41} \\
-c_{41}c_{51}
\end{bmatrix}
\end{align*}
A solution of its UMP is 
\begin{equation*}
B_1 = 
\left[ \begin{smallmatrix}
-1    &\ c_{41}^2c_{51} +c_{41}   & \  -c_{41}^2c_{51}^2 +1    &     \  -c_{41}c_{51}^3 -c_{51}^2        
         &\   c_{41}^3    &\  c_{41}^2c_{51} -c_{41} &\    {c_{41}c_{51}}_{\mathstrut}   \cr
c_{51}&\  -c_{41}^2c_{41}^2 -c_{41}c_{51} -1 &\  c_{41}^2c_{51}^3 -c_{51}       &\ c_{41}c_{51}^4 +c_{51}^3    
         &\ {-c_{41}^3c_{51}}^{\mathstrut}  &\  -c_{41}^2c_{51}^2 +c_{41}c_{51} &\ -c_{41}c_{51}^2 \cr
0&\  0&\      1& \      0&\      0&\        0&\    0 \cr
0&\  0&\      0& \      1&\      0&\        0&\    0 \cr
0&\  0&\      0& \      0&\      1&\        0&\    0\cr
0&\  0&\      0& \      0&\      0&\        1&\    0 \cr
0&\  0&\      0& \      0&\      0&\        0& \    1
\end{smallmatrix}\right]
\end{equation*}

\goodbreak
Moreover, we have 
\begin{align*}
\cmat_{K[C_0], C_2}(f_2) = \begin{bmatrix}
c_{41}c_{51} & -2c_{41}c_{51} + 1 & -c_{51}^2
\end{bmatrix}
\end{align*}
A solution of its UMP is 
\begin{equation*}
B_2 = \begin{bmatrix}
  2 \     & 2c_{41}c_{51}-1  \  & 2c_{51}^2  \cr
  1  \    &  c_{41} c_{51}    \   & c_{51}^2  \cr
  0  \   &       0                   \   &  1
\end{bmatrix}
\end{equation*}

\item 
The two matrices  $B_1$ and $B_2$ are used to define a
graded $K[C_0]$-algebra automorphism of $K[C]/I(\LBO)$ which maps $f_1$ to $c_{21}$ 
and $f_2$ to $f_2'$ (see Theorem~\ref{thm-GoodIso} and Remark~\ref{rem-block}). 
Next, we  substitute $c_{21}$ with $0$ in $f_2'$, and we obtain  a coherently 
$(c_{21}, c_{22})$-separating pair $(c_{21}, f)$, where

\goodbreak
\begin{align*}
f ={} & c_{22} + c_{41}^5c_{43}^2c_{51}^7 + 2c_{41}^4c_{43}c_{45}c_{51}^8 + c_{41}^3c_{45}^2c_{51}^9 
- 2c_{41}^5c_{42}c_{43}c_{51}^6 - 2c_{41}^4c_{42}c_{45}c_{51}^7 \\
& - 2c_{41}^6c_{43}c_{51}^5c_{52} - 2c_{41}^5c_{45}c_{51}^6c_{52} - 2c_{41}^5c_{43}c_{51}^6c_{54} - 2c_{41}^4c_{45}c_{51}^7c_{54} \\ 
& + c_{41}^5c_{42}^2c_{51}^5 + 2c_{41}^3c_{43}c_{45}c_{51}^7 + 2c_{41}^2c_{45}^2c_{51}^8 + 2c_{41}^6c_{42}c_{51}^4c_{52} \\
& + c_{41}^7c_{51}^3c_{52}^2 + 2c_{41}^5c_{42}c_{51}^5c_{54} + 2c_{41}^6c_{51}^4c_{52}c_{54} + c_{41}^5c_{51}^5c_{54}^2 \\
& - 2c_{41}^4c_{43}c_{51}^6c_{55} - 2c_{41}^3c_{45}c_{51}^7c_{55} - 2c_{41}^4c_{42}c_{43}c_{51}^5 - 4c_{41}^3c_{42}c_{45}c_{51}^6 \\
& - 2c_{41}^4c_{45}c_{51}^5c_{52} + 2c_{41}^4c_{43}c_{51}^5c_{54} + 2c_{41}^4c_{42}c_{51}^5c_{55} + 2c_{41}^5c_{51}^4c_{52}c_{55} \\
& + 2c_{41}^4c_{51}^5c_{54}c_{55} + 2c_{41}^4c_{42}^2c_{51}^4 - 3c_{41}^3c_{43}^2c_{51}^5 - 4c_{41}^2c_{43}c_{45}c_{51}^6 \\
& + 2c_{41}^5c_{42}c_{51}^3c_{52} - 2c_{41}^5c_{51}^3c_{52}c_{54} - 2c_{41}^4c_{51}^4c_{54}^2 - 2c_{41}^2c_{45}c_{51}^6c_{55} \\
& + c_{41}^3c_{51}^5c_{55}^2 + 2c_{41}^3c_{42}c_{43}c_{51}^4 - 2c_{41}^2c_{42}c_{45}c_{51}^5 + 2c_{41}^4c_{43}c_{51}^3c_{52} \\
& + 4c_{41}^3c_{43}c_{51}^4c_{54} + 4c_{41}^2c_{45}c_{51}^5c_{54} + 2c_{41}^3c_{42}c_{51}^4c_{55} - 2c_{41}^3c_{51}^4c_{54}c_{55} \\
& + 2c_{41}^3c_{42}^2c_{51}^3 - c_{41}^2c_{43}^2c_{51}^4 - 3c_{41}c_{43}c_{45}c_{51}^5 + 2c_{41}^4c_{42}c_{51}^2c_{52} + c_{41}^5c_{51}c_{52}^2 \\
& - 2c_{41}^3c_{42}c_{51}^3c_{54} + 2c_{41}^2c_{43}c_{51}^4c_{55} + 4c_{41}^2c_{42}c_{43}c_{51}^3 - c_{41}c_{42}c_{45}c_{51}^4 \\
& + 2c_{41}^3c_{43}c_{51}^2c_{52} - c_{41}^2c_{45}c_{51}^3c_{52} - 2c_{41}^2c_{43}c_{51}^3c_{54} - c_{41}c_{45}c_{51}^4c_{54} \\
& + 2c_{41}^2c_{42}c_{51}^3c_{55} + 2c_{41}^3c_{51}^2c_{52}c_{55} + c_{41}^2c_{42}^2c_{51}^2 + 4c_{41}c_{43}^2c_{51}^3 \\
& + c_{43}c_{45}c_{51}^4 - c_{41}^4c_{52}^2 - 2c_{41}^2c_{42}c_{51}^2c_{54} - 2c_{41}^3c_{51}c_{52}c_{54} + c_{41}^2c_{51}^2c_{54}^2 \\
& + 3c_{41}c_{43}c_{51}^3c_{55} + c_{41}c_{51}^3c_{55}^2 + c_{41}c_{42}c_{43}c_{51}^2 + c_{42}c_{45}c_{51}^3 \\
& + 3c_{41}c_{45}c_{51}^2c_{52} - 5c_{41}c_{43}c_{51}^2c_{54} + c_{45}c_{51}^3c_{54} - c_{41}c_{42}c_{51}^2c_{55} \\
& - 3c_{41}^2c_{51}c_{52}c_{55} - 3c_{41}^2c_{51}^2c_{54}c_{55} - c_{41}^2c_{42}c_{52} + c_{41}^2c_{52}c_{54} \\
& + c_{41}c_{51}c_{54}^2 - 2c_{43}c_{51}^2c_{55} - c_{51}^2c_{55}^2 - c_{41}c_{43}c_{52} - c_{45}c_{51}c_{52} \\
& + 2c_{43}c_{51}c_{54} + c_{41}c_{52}c_{55} + 2c_{51}c_{54}c_{55} - c_{54}^2
\end{align*}

\item By using the pair $(c_{21}, f)$, more precisely by sending $c_{21} \mapsto 0$
and $c_{22} \mapsto c_{22}-f$, we obtain a map $\Psi \colon S \To T$,
where $T = \QQ[c_{33},\, c_{35},\, c_{41},\, c_{42},\, c_{43},\, c_{45},\, c_{51},\, c_{52},\, c_{54},\, c_{55}]$,
which is an isomorphism of graded algebras.

\item In conclusion, we obtain an isomorphism $\Psi\circ \Phi \colon B_\OO \To T$ induced by 
a surjective homomorphism $\eta \colon K[C] \To T$.
It is remarkable that the $\eta$-images of the 25 indeterminates in $K[C]$ are 
polynomials whose supports have  cardinalities
\begin{equation*}
78, \mkern1mu 329, \mkern1mu 375,\mkern1mu 372,\mkern1mu 419,\mkern1mu 10,\mkern1mu 
 87,\mkern1mu 87,\mkern1mu 95,\mkern1mu 109,\mkern1mu 8,\mkern1mu 90,\mkern1mu 
86,\mkern1mu 99,\mkern1mu 1,\, 1,\, 11,\, 1,\, 11,\, 1,\, 1,\, 1,\, 9,\, 1,\, 1
\end{equation*}
\end{myenumerate}
\end{example}

\medskip
\subsection{Simplicial Border Basis Schemes}
\label{subsec-Simplicial BBS}

We start with the definition of  very special  MaxDeg border basis schemes.

\begin{definition}\label{def-simplicial}
An order ideal $\OO \subset \TT^n$ is called \textbf{simplicial} (of \textbf{type} $d$) if we have 
$\OO = \{t \in \TT^n \mid \deg(t)\le d\}$, in other words, if we call $\OO_i$ the set of terms 
of standard degree $i$, we have $\OO = \cup_{i=0}^d \OO_i$.
\end{definition}

\begin{proposition}\label{prop-postotgrad}
Let $\OO \subset \TT^n$,  let $B_\OO = K[C]/I(\BO)$, and let $\deg_W$ 
denote the total arrow grading on~$K[C]$.
\begin{myenumerate}
\item We have $\partial\OO = \OO_{d+1}$

\item The following are equivalent.
\begin{enumerate}

\item[(2a)] The order ideal $\OO$ is simplicial.

\item[(2b)] The total arrow grading $\deg_W$ is a positive grading on~$K[C]$.
\end{enumerate}

\end{myenumerate}
\end{proposition}

\begin{proof}
Since (1) follows directly from the definition, let us prove (2). 
If $\OO$ is simplicial, it follows from (1) that
$\deg_W(c_{ij}) = d+1 - \delta$ with $\delta \le d$. To show the opposite implication,
assume for contradiction that~$\OO$ is not simplicial and let $d = \max\{\deg(t)  \mid  t \in \OO\}$.
Then there exists at least one element $b_j \in \partial\OO$ with $\deg(b_j) = d$, hence there exists at least one
element $c_{ij}$ with $\deg_W(c_{ij}) = 0$, a contradiction.
\end{proof}

Let us collect some basic numerical information about simplicial order ideals 
which will be used in the following without further mention.

\begin{lemma}\label{lem-formulas}
Let $d\ge 1$, let $\OO = \{t_1, \dots, t_\mu\}$ be the simplicial order ideal of type~$d$
in~$\mathbb{T}^n$, and let $\partial\OO = \{ b_1, \dots, b_\nu\}$ be its border.
Then the following formulas hold.
\begin{myenumerate}
\item $\mu = \binom{d+n}{n}$

\item $\nu = \binom{d+n}{n-1}$

\item $\# \OOint = \binom{d+n-1}{n}$

\item $\# \OOrim = \binom{d+n-1}{n-1}$

\item$\# C = \binom{d+n}{n} \cdot \binom{d+n}{n-1}$

\item $\# \Cint = \binom{d+n-1}{n} \cdot \binom{d+n}{n-1}$

\item $\# \Crim = \binom{d+n-1}{n-1} \cdot \binom{d+n}{n-1}$
\end{myenumerate}
\end{lemma}

\begin{proof}
These equalities follow directly from the formulas for
the usual and the affine Hilbert functions of a polynomial ring 
given in~\cite[Propositions~5.1.13 and 5.6.3.c]{KR2}.
\end{proof}

\begin{proposition}\label{prop-simplicialOptimal}
Let $\OO$ be a simplicial order ideal, let $B_\OO = K[C]/I(\BO)$ 
be the coordinate ring of $\BO$, and let $r = \dim_K(\Lin_K(I(\BO)))$. Then
there exists an $r$-tuple of indeterminates $c_{ij}$ and an $r$-tuple of $W$-homogeneous
polynomials in $I(\BO)$ defining an optimal separating re-embedding of $I(\BO)$.
\end{proposition}

\begin{proof}
We know from Proposition~\ref{prop-postotgrad} that the $K$-algebra $B_\OO$ 
is positively graded. Hence the conclusion follows from Proposition~\ref{prop-optimalgraded}.
\end{proof}

A deeper insight into the structure of $I(\BO)$ allows us to make the claim 
of the above proposition more explicit.

\begin{proposition}{\bf (Optimal $\Cint$-separating Re-embedding)}\label{prop-optimalCint}

Let $\OO$ be a simplicial order ideal, let $B_\OO = K[C]/I(\BO)$ 
be the coordinate ring of $\BO$.
\begin{myenumerate}
\item There exists an optimal  $\Cint$-separating re-embedding of $I(\BO)$.

\item We have $\Lin_K(I(\BO)) = \langle \Cint \rangle_K$.

\item We have $\Cot_\m(B_\OO)= \langle \Crim \rangle_K$ 

\item We have $\edim(B_\OO) = \dim_K (\Cot_{\m}(B_\OO)) $.  
\end{myenumerate}
\end{proposition}

\begin{proof}
The proof of (1) is given in~\cite[Proposition 5.5]{KR5}. Let us prove  claim (2).
The proof of (1) shows that $\Cint \subseteq \Lin_K(I(\BO))$. On the other hand,
in~\cite[Corollary 2.8.b]{KSL} it is shown that there is no
linear part of one of the natural generators of~$I(\BO)$ which contains
an indeterminate of~$\Crim$ in its support, and hence $\Lin_K(I(\BO)) \subseteq \Cint$,
which concludes the proof.
To prove (3), we note that $C$ is  the disjoint union of $\Cint$ and $\Crim$. Moreover,  the equality 
$\Cot_\m(B_\OO) = \langle C \rangle_K/\Lin_K(I(\BO))$ follows from Proposition~\ref{prop-cotancomp}\,(3).
Consequently, the conclusion follows from (1).

The proof of (4) follows as a combination of (1) and (3) together with the disjoint union of 
$\Cint$ and $\Crim$ already mentioned above.
\end{proof}

An application of the proposition is that we can check  whether 
a simplicial border basis scheme is an affine cell. 

\begin{proposition}\label{prop-AffineOrSing}
Let $\OO$ be the simplicial order ideal of type~$d$, and
let $\Gamma_0=(0, \dots, 0)$ be the origin of~$\BO$ (see Definition~\ref{def-represent}).

\begin{myenumerate}
\item If $n=2$ then $\BO$ is isomorphic to $\mathbb{A}^{(d+1)(d+2)}$.

\item If $n\ge 3$ then $\Gamma_0$ is a singular point of~$\BO$.
\end{myenumerate}
\end{proposition}

\begin{proof}
For $n = 2$, it follows from claims (3) and (4) of Proposition~\ref{prop-optimalCint}  that
$\BO$ can be embedded into an affine space of dimension
\begin{equation*}
\dim_K (\Cot_\m(B_\OO)) =\dim_K(\langle \Crim \rangle_K) = \#\Crim 
=\tbinom{d+1}{1}\cdot \tbinom{d+2}{1} =   (d+1)(d+2)
\end{equation*}
We know from Proposition~\ref{prop-irreducible} that the  scheme $\BO$ is irreducible 
and from Remark~\ref{rem-nmu} that 
$\dim(\BO)=2 \mu = 2\binom{d+2}{2} = (d+1)(d+2)$.
The conclusion follows.

Let us prove (2). The dimension of the principal component of~$\BO$
is $n\cdot \#\OO = n\cdot \binom{d+n}{n}$,
and  from Proposition~\ref{prop-optimalCint}(3) we get
$\dim_K (\Cot_\m(B_\OO)) =\#\Crim= \binom{d+n-1}{n-1} \cdot \binom{d+n}{n-1}$.
Hence it suffices to show that for $n\ge 3$ we have
\begin{equation*}
n\cdot \tbinom{d+n}{n} < \tbinom{d+n-1}{n-1} \cdot \tbinom{d+n}{n-1}  \eqno{(1)}
\end{equation*}
To prove (1), we note that $n\cdot \binom{d+n}{n} {=}(d+n)\cdot \binom{d+n-1}{n-1}$.
Hence (1) is equivalent to
\begin{equation*}
d+n <  \tbinom{d+n}{n-1}   \eqno{(2)}
\end{equation*}
Therefore we need to show that for $n\ge 3$ we have
\begin{equation*}
1 <  \tfrac{ (d+n-1)\cdot(d+n-2)\cdots(d+2)}{(n-1)!}=
\tfrac{d+n-1}{n-1}\cdot \tfrac{d+n-2}{n-2}\cdots \tfrac{d+2}{2}
\end{equation*}
which is clear since $d\ge 1$.

Notice that for $n=2$ we have $d+n = \binom{d+n}{n-1}$. Thus we have equality  in (2),
and hence also equality in~(1). This shows $\dim_K(\Cot_\m(B_\OO)) = \dim(\BO)$ 
and implies that~$\Gamma_0$ is a smooth point of~$\BO$, in agreement with~(1).
\end{proof}

The following example illustrates this proposition. 

\begin{example}\label{ex-singBBS}
Let $P=\QQ[x,y,z]$, and let $\OO= \{1, z, y, x\}$. 
Then the border of~$\OO$ is given by
$\partial\OO=\{ z^2,\, yz,\, xz,\, y^2,\,  xy,\,  x^2\}$. 
This yields $\mu=4$ and $\nu=6$. Consequently, the polynomial ring $\QQ[C]$
has 24 indeterminates.
Proposition~\ref{prop-optimalCint}(1) implies that 
$\Lin_\M(I(\BO)) = \langle c_{11},\dots,c_{16}\rangle_{\QQ}$, and 
there exists a tuple $(f_1,\dots, f_6)$ of coherently $Z$-separating
polynomials  in~$I(\BO)$, where $Z=(c_{11},\dots,c_{16})$. 
\begin{align*}
f_1 \;=\; & c_{11} -c_{23}c_{41} -c_{33}c_{42}+c_{21}c_{43} -c_{43}^2 +c_{31}c_{45} +c_{41}c_{46},\\
f_2 \;=\; & c_{12} -c_{23}c_{42} +c_{22}c_{43} -c_{33}c_{44} +c_{32}c_{45} -c_{43}c_{45} +c_{42}c_{46},\\
f_3 \;=\; & c_{13} -c_{26}c_{41} -c_{36}c_{42} +c_{23}c_{43} +c_{33}c_{45},\\
f_4 \;=\; & c_{14} -c_{25}c_{42} +c_{24}c_{43} -c_{35}c_{44} +c_{34}c_{45} -c_{45}^2 +c_{44}c_{46},\\
f_5 \;=\; & c_{15} -c_{26}c_{42} +c_{25}c_{43} -c_{36}c_{44} +c_{35}c_{45},\\
f_6 \;=\; & c_{16} +c_{26}c_{32} -c_{25}c_{33} -c_{35}^2 +c_{34}c_{36} -c_{36}c_{45} +c_{35}c_{46}
\end{align*}
Consequently, the $Z$-separating re-embedding of the ideal~$I(\BO)$
is an isomorphism $\Phi \colon  B_\OO \rightarrow \widehat{P} / (I(\BO)\cap \widehat{P})$,
where $\widehat{P} = \QQ[c_{21},\dots,c_{26},c_{31},\dots,c_{36},c_{41},\dots,c_{46}]$
is a polynomial ring in~18 indeterminates. It turns out
that $I(\BO) \cap \widehat{P}$ is minimally generated by the 15 quadratic polynomials
$$
{\scriptstyle
\begin{array}{lll}
 h_1 &\!=\!&  c_{25}c_{32} {-}c_{24}c_{33} {+}c_{26}c_{42} {-}c_{25}c_{43} \\
 h_2 &\!=\!&   c_{25}c_{41} {-}c_{23}c_{42} {+}c_{35}c_{42} {-}c_{33}c_{44} \\
 h_3 &\!=\!& c_{25}c_{31} {-}c_{22}c_{33} {-}c_{36}c_{42} {+}c_{33}c_{45} \\
 h_4 &\!=\!&  c_{26}c_{31} {-}c_{23}c_{33} {-}c_{33}c_{35} {+}c_{32}c_{36} {-}c_{36}c_{43} {+}c_{33}c_{46} \\
 h_5 &\!=\!&    c_{22}c_{41} {-}c_{21}c_{42} {+}c_{32}c_{42} {+}c_{42}c_{43} {-}c_{31}c_{44} {-}c_{41}c_{45} \\
 h_6 &\!=\!&  c_{23}c_{24} {-}c_{22}c_{25} {+}c_{25}c_{34} {-}c_{24}c_{35} {+}c_{26}c_{44} {-}c_{25}c_{45} \\
 h_7 &\!=\!&    c_{23}c_{31} {-}c_{21}c_{33} {+}c_{32}c_{33} {-}c_{31}c_{35} {-}c_{36}c_{41} {+}c_{33}c_{43} \\
 h_8 &\!=\!&     c_{24}c_{41} {-}c_{22}c_{42} {+}c_{34}c_{42} {-}c_{32}c_{44} {+}c_{43}c_{44} {-}c_{42}c_{45} \\
 h_9 &\!=\!&    c_{23}c_{25} {-}c_{22}c_{26} {+}c_{25}c_{35} {-}c_{24}c_{36} {+}c_{26}c_{45} {-}c_{25}c_{46} \\
 h_{10} &\!=\!&   c_{23}c_{32} {-}c_{22}c_{33} {+}c_{33}c_{34} {-}c_{32}c_{35} {+}c_{26}c_{41} {-}c_{23}c_{43}
                 {+}c_{35}c_{43} {-}c_{33}c_{45} \\
 h_{11} &\!=\!&   c_{24}c_{31} {-}c_{22}c_{32} {-}c_{23}c_{42} {-}c_{35}c_{42} {+}c_{22}c_{43} {+}c_{32}c_{45}
                  {-}c_{43}c_{45} {+}c_{42}c_{46} \\
 h_{12} &\!=\!&   c_{22}c_{23} {-}c_{21}c_{25} {+}c_{24}c_{33} {-}c_{22}c_{35} {-}c_{26}c_{42} {+}2c_{25}c_{43}
                  {-}c_{36}c_{44} {-}c_{23}c_{45}  {+}c_{35}c_{45} \\
 h_{13} &\!=\!&   c_{22}^2 {-}c_{21}c_{24} {+}c_{24}c_{32} {-}c_{22}c_{34} {+}c_{24}c_{43} {-}c_{23}c_{44} 
                  {-}c_{35}c_{44} {+}c_{34}c_{45} {-}c_{45}^2 {+}c_{44}c_{46} \\
 h_{14} &\!=\!&   c_{22}c_{31} {-}c_{21}c_{32} {+}c_{32}^2 {-}c_{31}c_{34} {-}c_{23}c_{41} {-}c_{35}c_{41} 
                  {+}c_{21}c_{43} {-}c_{43}^2 {+}c_{31}c_{45} {+}c_{41}c_{46} \\
 h_{15} &\!=\!&   c_{23}^2 {-}c_{21}c_{26} {+}c_{26}c_{32} {-}c_{35}^2 {-}c_{22}c_{36} {+}c_{34}c_{36} 
                  {+}c_{26}c_{43} {-}c_{36}c_{45} {-}c_{23}c_{46} {+}c_{35}c_{46} 
\end{array}
}
$$
Summarizing, since $\dim_\QQ(\Lin_\QQ(I(\BO))) = 6 = \#Z$, 
we know from  Claims (3) and (4) of Proposition~\ref{prop-optimalCint} that $\edim(B_\OO)=18$, that~$\Phi$ 
yields an optimal embedding of the 12-dimensional scheme $\BO$ into $\AA^{18}_{\QQ}$, 
and that the vanishing ideal of the image is minimally generated by~15 quadrics.
\end{example}

\newpage
\section{Subschemes of  Border Basis Schemes}
\label{sec-SubBBS}

In this section we show actual and potential extensions of the material treated so far.
The main goal suggested here is to describe and compute subschemes 
of a border basis scheme whose closed points represent zero-dimensional schemes which
share a property. This is a wide open area of research, although some results 
are already available. 

\medskip
\subsection{The Subscheme  \texorpdfstring{$\BO\df$}{BOdf} }
\label{subsec-BOdf}

An important subscheme of a border basis scheme is  the 
Degree Filtered Border Basis Scheme $\BOdf$. To introduce it, we need 
a key ingredient, the degree filtration.

\begin{definition}\label{def-degfiltration}
Let $R=P/I$ be a zero-dimensional affine $K$-algebra.

 We equip~$P$ with the
(standard) {\bf degree filtration} $\widetilde{\mathcal{F}}
= (F_iP)_{i\in\mathbb{Z}}$, where
\begin{equation*}
F_iP \;=\;  \{f\in P \mid \deg(f) \le i\} \;\cup\; \{0\}
\end{equation*}
This is an increasing filtration which satisfies
$F_iP=\{0\}$ for $i<0$ and $F_0P=K$.
For every $i\in\mathbb{Z}$, let $F_iI=F_iP\cap I$, and let
$F_iR = F_iP/F_iI$. Then the family $(F_iI)_{i\in\mathbb{Z}}$
is called the {\bf induced filtration} on~$I$, and the family
$\mathcal{F}= (F_iR)_{i\in\mathbb{Z}}$
is a $\mathbb{Z}$-filtration on~$R$
which is called the {\bf degree filtration} on~$R$.

It is  \textbf{exhaustive}  and \textbf{orderly} in the
sense that every element $f\in R \setminus \{0\}$ has an
\textbf{order} $\ord_\F(f)=\min\{i\in\ZZ \mid f\in F_i R
\setminus F_{i-1}R\}$.
\end{definition}

\begin{definition}\label{def-degfiltrbasis}
Let $R=P/I$ be a zero-dimensional affine $K$-algebra equipped with a 
degree filtration $\widetilde{\mathcal{F}}$.
\begin{myenumerate}
\item  A tuple $B=(\bar{t}_1,\dots,\bar{t}_\mu) \in R^\mu$ 
is called a \textbf{degree filtered $K$-basis} of~$R$ if
the set $F_iB = B\cap F_i R$ is a $K$-basis of~$F_i R$
for every $i\in\ZZ$ and if $\ord_\F(\bar{t}_1) \le \cdots 
\le \ord_\F(\bar{t}_\mu)$.

\item Given an order ideal $\OO=\{t_1,\dots,t_\mu\}$
in~$\mathbb{T}^n$ and a zero-dimensional ideal $I \subset P$
such that $(\bar{t}_1,\dots,\bar{t}_\mu)$ is a degree-filtered $K$-basis of $P/I$,
we say that $P/I$ has  a \textbf{degree-filtered $\OO$-border basis}.

\end{myenumerate}
\end{definition}

The following example shows that not all monomial bases of $R$ are 
necessarily degree-filtered.

\begin{example}\label{nondegreefiltered}
Let $K=\mathbb{Q}$, let $P = K[x,y]$,
let $I$ be the vanishing ideal of the affine set
of eight points given by
$p_1=(1,-1)$, $p_2 =(0,2)$, $p_3=(1,1)$,
$p_4 =(1,2)$, $p_5=(0,1)$, $p_6=(1,3)$,
$p_7 = (2,4)$, and $p_8 =(3,4)$, and let $R=P/I$.
The reduced Gr\"{o}bner basis of~$I$ with respect to
\texttt{DegRevLex} is
\begin{gather*}
\{ \; x^{2}y -4x^{2} - xy + 4x,\quad 
x^{3} + xy^{2} -6x^{2} -3xy - y^{2} + 7x + 3y -2,\\
y^{4} -10xy^{2} -5y^{3} + 15x^{2} + 30xy + 15y^{2} -35x -25y + 14,\\
xy^{3} -7xy^{2} - y^{3} + 14xy + 7y^{2} -8x -14y + 8\;\}
\end{gather*}
Since this term ordering is degree compatible,
the residue classes of the elements in the tuple
$(1, y, x, y^2, xy, x^2, y^3, xy^2)$
form a degree-filtered $K$-basis of~$R$
with order tuple $(0,1,1,2,2,2,3,3)$.
On the other hand, the reduced Gr\"{o}bner basis of~$I$ 
with respect to~\texttt{Lex} is
\begin{gather*}
\{\, x^{2} -\tfrac{2}{3}xy^{2} + 2xy -\tfrac{7}{3}x +
\tfrac{1}{15}y^{4} -\tfrac{1}{3}y^{3}
+ y^{2} -\tfrac{5}{3}y + \tfrac{14}{15},\\
xy^{3} -7xy^{2} + 14xy -8x - y^{3} + 7y^{2} -14y + 8,\,
y^{5} -9y^{4} + 25y^{3} -15y^{2} -26y + 24 \,\}
\end{gather*}
So, the residue classes of the elements in the tuple
$B=(1, y, x,  y^2, xy, y^3, xy^2, y^4)$ form a $K$-basis of~$R$.
Since $\bar{y}^4 = 10\bar{x}\bar{y}^2
+ 5\bar{y}^3 - 15\bar{x}^2 - 30\bar{x}\bar{y}
- 15\bar{y}^2 + 35\bar{x} + 25\bar{y} - 14$, 
we have $\ord_\Fbar(\bar{y}^4)=3$.
Altogether, we see that~$B$ is not a degree-filtered basis, since
its order tuple is $(0,1,1,2,2,3,3,3)$.
\end{example}

The degree filtered border basis scheme
is a subscheme of~$\BO$ which
parametrizes all zero-dimensional ideals in~$P$
which have a degree filtered $\OO$-border basis.
The following proposition provides an explicit description
of this subscheme.

\begin{proposition}\label{CharBOdf}
Let $\OO=\{t_1, \dots, t_\mu\}$ be an order ideal and let
$\partial\OO=\{b_1,\dots,b_\nu\}$ be the border of~$\OO$. 

\begin{myenumerate}
\item For a $K$-rational point $\Gamma=(\gamma_{ij})$ of~$\BO$,
the zero-dimensional scheme $\X_\Gamma$ represented by~$\Gamma$
has a degree-filtered $\OO$-border basis if and only if $\gamma_{ij}=0$
for all $i\in\{1,\dots,\mu\}$ and $j\in\{1,\dots,\nu\}$ such that
$\deg(t_i)>\deg(b_j)$.

\item Let $I_\OO\df$ be the ideal in~$K[C]$
generated by all  $c_{ij}$ such that $\deg(t_i)>
\deg(b_j)$. The $K$-rational points of the closed subscheme~$\BOdf$ 
of~$\BO$ defined by the ideal $I(\BO)+I_\OO\df$ 
represent the zero-dimensional schemes $\X$ in~$\AA^n_K$ 
which have a degree-filtered $\OO$-border basis.
\end{myenumerate}
\end{proposition}

\begin{proof}
Claim~(1) follows  from~\cite[Proposition 5.3(d)]{KLR05}, 
and~(2) is a consequence of~(1).
\end{proof}

Part~(2) of this proposition 
gives rise to the following definition.

\begin{definition}\label{def-degfiltrScheme}
Let $\OO=\{t_1,\dots,t_\mu\}$ be an order ideal in~$\mathbb{T}^n$,
let $\partial\OO=\{b_1,\dots,b_\nu\}$ be the border of~$\OO$, and
let $I_\OO\df$ be the ideal in~$K[C]$ generated by all 
indeterminates $c_{ij}$ such that $\deg(t_i)>\deg(b_j)$.

\begin{myenumerate}
\item The closed subscheme~$\BOdf$ of~$\BO$ defined by 
$I(\BOdf) = I(\BO)+I_\OO\df$ is called the
\textbf{degree-filtered $\OO$-border basis scheme}.
Its affine coordinate ring is denoted by $B_\OO\df =
K[C]/I(\BOdf)$.

\item The set of polynomials $G\df = \{g_1\df,\dots,
g_\nu\df\}$ in~$K[C][x_1,\dots,x_n]$ given by the formula
$g_j\df = b_j - \sum_{\{i\mid \deg(t_i)\le\deg(b_j)\}}
c_{ij}\, t_i$  for $j=1,\dots,\nu$ 
is called the \textbf{generic degree-filtered $\OO$-border prebasis}.

\item  The canonical $B_\OO\df$-algebra homomorphism
$$
\Phi\df \colon B_\OO\df \longrightarrow U_\OO\df :=
B_\OO\df[x_1,\dots,x_n]/ \langle G\df\rangle
$$ 
is called the \textbf{universal degree-filtered $\OO$-border basis family}.
\end{myenumerate}
\end{definition}

\begin{myremark}\label{rem-BOdfnotPositive}
The ring $B_\OO\df$ is non-negatively graded but, in general, it is not a positive $P_0$-algebra
since it may happen that $I(\BOdf)  \cap (B_\OO\df)_0 \ne \langle 0\rangle$, 
as the following example illustrates.
\end{myremark}

\begin{example}\label{ex-BOdfnotPositive}
Let $P:=\QQ[x,y]$ and let $\OO=\{1, y, x, y^2, x^2, y^3 \}$. Its border is 
$\partial\OO = \{xy,  xy^2,  x^2y,  x^3,  y^4,  xy^3\}$. The scheme $\BO$ is not MaxDeg
since the indeterminate  $c_{61}$ has degree $-1$ with respect to the total arrow grading.

We have  $B_\OO^{\hom}= \QQ[C]/I$ where $C=(c_{41}, c_{51}, c_{62}, c_{63}, c_{64})$ 
and $I=\langle f_1, f_2 \rangle$ where  $f_1 = c_{41} -c_{62} +c_{51}c_{63}$,
and $f_2 = c_{63} -c_{41}c_{62} - c_{51}c_{64}$. There is no $Z$ of  length~2 
such that the pair $(f_1, f_2)$ is  coherently $Z$-separating, so that we do not even know 
if $\BOhom$ is an affine cell.

A further remark about this example is that  the two polynomials which define~$\BOhom$
are ``parts'' of the following  polynomials in~$I(\BO)$
\begin{equation*}
g_1=c_{41} -c_{62} +c_{51}c_{63}  +c_{31}c_{61} +c_{61}c_{65}, \quad
g_2=c_{63} -c_{41}c_{62} - c_{51}c_{64} -c_{21}c_{61} - c_{61}c_{66}
\end{equation*}
in the sense that $f_i $ is obtained by setting $c_{61}=0$ in  $g_i$ for $i=1,2$.

Moreover, $g_1$ and $g_2$  are among the natural generators of $I(\BO)$, 
and are  homogeneous of degree 0 
since the total arrow degree of $c_{61}$ is  -1
while the total arrow degree of the indeterminates $c_{31}, c_{65}, c_{21}, c_{66}$ is 1.

In particular, in this example we have  
\begin{myenumerate}

\item The ring $B_\OO\hom$ is not a polynomial ring.

\item The ring $B_\OO\hom$ is  the degree zero part of~$B_\OO\df$,  
but  not the degree zero part of~$B_\OO$.
\end{myenumerate}

\smallskip
Using $Z$-re-embedding and UMP it was not possible to decide if $\BOdf$ is an affine cell. 
The length of the support of some  of the remaining polynomials is huge.
\end{example}

To get a nice property of $\BOdf$, we need the following result. We recall that in our setting,
for a given $\OO=\{t_1,\dots,t_\mu\}$  and $\partial\OO=\{b_1,\dots,b_\nu\}$ we have 
$\deg(t_1)\le \cdots \le \deg(t_\mu)$ and $\deg(b_1)\le \cdots \le\deg(b_\nu)$.

\begin{lemma}
Let $\OO=\{t_1,\dots,t_\mu\}$ be an order ideal
with border $\partial\OO=\{b_1,\dots,b_\nu\}$.
Assume that $\deg(t_\mu)>\deg(b_1)$, and let $\Gamma=(\gamma_{ij})\in K^{\mu\nu}$
with $\gamma_{\mu 1}=1$ and $\gamma_{ij}=0$ for $(i,j)\ne (\mu, 1)$.
Then~$\Gamma$ represents a 0-dimensional scheme $\X_\Gamma$ in~$\BO$.
\end{lemma}

\begin{proof} To show that the $\OO$-border prebasis
$G_\Gamma = (g_1,\dots,g_\nu) = (b_1-t_\mu, b_2,\dots, b_\nu)$ 
is an $\OO$-border basis,
it suffices to show that the residue classes of the 
elements of~$\OO$ form a $K$-basis of $P/I_\Gamma$, 
where $I_\Gamma=\langle G_\Gamma \rangle$. 

Since~$\OO$ is an order ideal,
we have $b_1 \nmid t_\mu$. Therefore there exists a term ordering~$\sigma$
such that $b_1 >_\sigma t_\mu$. Now we use Buchberger's Criterion
to verify that $G_\Gamma = (b_1-t_\mu, b_2,\dots,b_\nu)$ is a 
$\sigma$-Gr\"obner basis of $I_\Gamma$. 

For this it suffices to show that, for $k=2,\dots,\nu$, the 
$S$-polynomial~$S_{1k}$ satisfies $S_{1k} \stackrel{G}{\rightarrow} 0$.
Using
$$
S_{1k} = \tfrac{{\lcm}(b_1,b_k)}{b_1}(b_1-t_\mu) 
-  \tfrac{{\lcm}(b_1,b_k)}{b_k} b_k
= -\tfrac{{\lcm}(b_1,b_k)}{b_1}t_\mu
$$
and the fact that $\deg(b_1)\le \deg(b_k)$ implies $b_k \nmid b_1$,
we see that $S_{1k}$ is a proper multiple of~$t_\mu$.
Therefore the claim follows from $x_i t_\mu\in \partial\OO\setminus\{b_1\}
\subset G_\Gamma$. 
Now we conclude from Macaulay's Basis Theorem 
that $\OO=\OO_\sigma(I_\Gamma)=\mathbb{T}^n \setminus \LT_\sigma(G_\Gamma)$
represents a $K$-basis of $P/I_\Gamma$, and the proof is complete.
\end{proof}

At this point we are ready to prove the promised nice property of $\BOdf$.

\begin{proposition}\label{prop-CharMaxdeg}
For an order ideal $\OO=\{t_1,\dots,t_\mu\}$,
the following conditions are equivalent.

\begin{myenumerate}
\item We have $\BO=\BOdf$.

\item The scheme $\BO$ is MaxDeg.
\end{myenumerate}
\end{proposition}

\begin{proof}
First we show that~(1) implies~(2). By hypothesis, 
every scheme $\X_\Gamma$ represented by a $K$-rational point~$\Gamma$
of~$\BO$ has an $\OO$-border basis which is degree filtered. Suppose 
that there exist terms $t_i\in\OO$ and $b_j\in\partial\OO$
such that $\deg(t_i)>\deg(b_j)$. Then the lemma yields
a 0-dimensional scheme $\X_\Gamma$ which has an $\OO$-border
basis that is not degree filtered, in contradiction
to the hypothesis.

That~(2) implies~(1) follows from Proposition~\ref{CharBOdf}\,(2), since 
for a MaxDeg border basis scheme $\BO$  we have $I_\OO\df =\langle 0\rangle$.
\end{proof}

\medskip
\subsection{The Locally Gorenstein Locus and Other Loci}
\label{subsec-LocGor}

In this subsection we briefly report on some results concerning the computation 
of special loci inside a border basis scheme. This topic was studied in~\cite{KLR-1},
using special techniques described in~\cite{KLR-2, KR25}.

Let us recap the main contribution to this problem given in the book~\cite{KR25}.
Let $V$ be a finite-dimensional $K$-vector space and let $\mathcal{F}$ be a finite family 
of endomorphisms of $V$. Then $V$ acquires the structure of an $\mathcal{F}$-module 
via the operation $\phi\cdot v = \phi(v)$. Moreover, if $V$ is generated as 
an $\mathcal{F}$-module by a single element, we say that $V$ is \textbf{$\mathcal{F}$-cyclic}.
Then, the  \textbf{Cyclicity Test} (see~\cite[Algorithm 3.1.4]{KR25}) checks whether 
$V$ is $\mathcal{F}$-cyclic by simply computing the determinant of a suitable matrix.

How can this information be used in the context of border bases? A key ingredient is the following.

\begin{definition}\label{def-canonicalmodule}
Let $R=P/I$ be a zero-dimensional affine $K$-algebra.
The \textbf{canonical module}
of~$R$ is  the $K$-vector space
$\omega_R = \Hom_K(R,K)$ equipped with the $R$-module structure
defined by $f\cdot \phi (g)=\phi(fg)$ for $f, g\in R$.
\end{definition}

\begin{definition}\label{def-LocGor}
Let $R$ be a 0-dimensional affine $K$-algebra. If $(R, \m)$ is a local ring, 
we say that $R$ is Gorenstein if its socle, i.e.,\ the annihilator of~$\m$, is a principal ideal.
If $R$ is  semi-local and $\q_1,\dots,\q_s$ are the primary components
of the zero ideal of~$R$,  we say that~$R$ is
\textbf{locally Gorenstein} if $R/\q_i$ is a Gorenstein local ring for $i=1,\dots,s$. 
\end{definition}

Here is the promised link to the Cyclicity Test.

\begin{theorem}{\bf (The Canonical Module of 
a Locally Gorenstein Ring)}\label{thm-CharGor}

Let $K$ be a field, and let~$R$ be a zero-dimensional affine $K$-algebra.
Then the following conditions are equivalent.

\begin{myenumerate}
\item The ring $R$ is locally Gorenstein.

\item The canonical module $\omega_R$ is a cyclic $R$-module.

\item There exists an element $\phi\in \omega_R$ such that 
$\Ann_R(\phi) = \{0\}$.
\end{myenumerate}
\end{theorem}

\begin{proof}
See~\cite[Theorem 5.2]{KLR-2}.
\end{proof}

\begin{corollary}\label{cor-cycBB}
Let $P=K[x_1, \dots, x_n]$, 
let $\OO = \{t_1, \dots, t_\mu\}$ be an order ideal,
let $\partial\OO = \{b_1, \dots, b_\nu\}$ be the border of $\OO$, 
let  $\{g_1, \dots, g_\nu\}$ be an $\OO$-border basis, 
and let $I = \langle g_1, \dots, g_\nu  \rangle$. Moreover, let $M_{x_1}, \dots, M_{x_n}$ be the 
formal multiplication matrices and let $D \in \Mat_ \mu(K[z_1, \dots, z_\mu])$ be the 
matrix whose $j$-th column is
\begin{equation*}
t_j(M_{x_1}\tr,\dots, M_{x_n}\tr) \cdot (z_1,\dots, z_\mu)\tr.
\end{equation*}
Then the following conditions are equivalent.
\begin{myenumerate}
\item We have $\det(D) \ne 0$.

\item The ring $P/I$ is locally Gorenstein.
\end{myenumerate}
\end{corollary}

\begin{proof}
The transposed matrices $M_{x_i}\tr$ represent the endomorphisms of multiplication by~$x_i$ on the canonical module $\omega_{P/I}$, and  in our case the matrix $D$ is exactly the matrix used in the cyclicity test. The conclusion follows from the equivalence of (1) and (2) in the theorem.
\end{proof}

Finally, we have the solution to the problem of computing the subscheme $\LGor(\OO)$
of a border basis scheme $\BO$
which parametrizes zero-dimensional ideals $I$ in $P$ such that~$\OO$ represents a basis of $P/I$ and $P/I$ is locally Gorenstein.

\begin{theorem} {\bf  (The Locally Gorenstein Locus)}
\label{thm-LocGorLocus}

Let $\OO=\{t_1,\dots,t_\mu\}$ be an order ideal
in~$\mathbb{T}^n$, let $\mathcal{M}_{x_1}, \dots, \mathcal{M}_{x_n}$ be the 
generic multiplication matrices, and let $B_\OO = K[C]/I(\BO)$ be the 
coordinate ring of $\BO$. Then let 
$D \in \Mat_ \mu(K[C][z_1, \dots, z_\mu])$ be the 
matrix whose $ j$-th column is
\begin{equation*}
t_j(\mathcal{M}_{x_1}\tr,\dots, \mathcal{M}_{x_n}\tr) \cdot (z_1,\dots, z_\mu)\tr
\end{equation*}
for $j = 1,\dots, \mu$, and let $J$ be the ideal in $K[C]$ generated by the coefficients of $\det(D)$ 
with respect to the indeterminates $z_1,\dots, z_\mu$.

Then the ideal $I(\BO) {+} J$  defines a closed subscheme 
$\NonLGor(\OO)$ of~$\BO$
such that the set of $K$-rational points of the complement
$\BO\setminus \NonLGor(\OO)$ is precisely the set $\LGor(\OO)$.
\end{theorem}

\begin{proof}
The claim follows from Corollary~\ref{cor-cycBB}.
\end{proof}

The following example illustrates the theorem.
For the complete and detailed  computations see  \cocoa-Example~\ref{coex-LocGor121}.

\begin{example}\label{ex-LocGor121}
Let $P = \QQ[x,y]$ and let $\OO=\{1,\, y,\, x,\, xy\}$. Its border is
$\partial\OO = \{ y^2,\,  x^2,\, xy^2,\, x^2y\}$.  We have  $B_\OO = \QQ[C]/I(\BO)$ 
where $K[C]$ is a polynomial ring with $4 \times 4=16$ indeterminates all of them of 
non-negative total arrow degree. The ideal $I(\BO)$ is homogeneous
with respect to the total arrow grading which assigns to the tuple of indeterminates
$C=(c_{11}, c_{12}, \dots, c_{44})$ the weights given by
\begin{equation*}
W=(2,\, 2,\, 3,\, 3,\, 1,\, 1,\, 2,\, 2,\, 1,\, 1,\, 2,\, 2,\, 0,\, 0,\, 1,\, 1).
\end{equation*}
The transposes of the generic multiplication matrices are
\begin{equation*}
\mathcal{M}_x\tr =\begin{bmatrix}
0 & 0 & 1 & 0 \\
0 & 0 & 0 & 1 \\
c_{12} & c_{22} & c_{32} & c_{42} \\
c_{14} & c_{24} & c_{34} & c_{44}
\end{bmatrix}
\qquad
\mathcal{M}_y\tr =\begin{bmatrix}
0 & 1 & 0 & 0 \\
c_{11} & c_{21} & c_{31} & c_{41} \\
0 & 0 & 0 & 1 \\
c_{13} & c_{23} & c_{33} & c_{43}
\end{bmatrix}
\end{equation*}
Then, we use the column matrix $Z  = \begin{bmatrix} z_1 & z_2 & z_3 & z_4 \end{bmatrix}\tr$
and construct a square matrix $DZ$ whose four columns are 
$Z,\; \mathcal{M}_x\tr\cdot Z, \;  \mathcal{M}_y\tr\cdot Z, \;  \mathcal{M}_{xy}\tr\cdot Z$. 
Its determinant $D$ is 
a polynomial in $\QQ[C][z_1,z_2,z_3,z_4]$ whose support consists of 35 elements.
Next, we construct the ideal $J$ in $\QQ[C]$ generated by the generators of 
$I(\BO)$ together with the coefficients of $D$.  Finally we compute $\dim(\QQ[C]/J)$,
and obtain $\dim(\QQ[C]/J)= 4$.  The conclusion is that in this case  $\LGor(\OO)$ is the complement  of the  4-dimensional subscheme $\NonLGor(\OO)$ of the 8-dimensional 
scheme $\BO$.

From this result, we deduce that non-Gorenstein zero-dimensional schemes such that  $\OO$ 
represents a $\QQ$-basis of their coordinate rings are rare. Let us detect one of them.
Let $I_1 = \langle x +1,\, y +1\rangle$ and $I_2 = \langle x^2, \, xy, \, y^2\rangle$, and
let $I= I_1\cap I_2$. Since $I_2$ is not Gorenstein, the ideal $I$ is not locally Gorenstein.
On the other hand, the ideal $I$ is generated by the $\OO$-border 
basis $(y^2 - xy, \, x^2 - xy,\, xy^2 + xy,\, x^2y + xy)$. Therefore, $I$ is represented by the rational point
$\Gamma=(0,0,0,1,\,  0,0,0,1,\, 0,0,0,-1,\, 0,0,0,-1)$ of $\BO$, which means that $\Gamma$ is a zero of 
$I(\BO)$. It is also a zero of the ideal  generated by the coefficients of $D$, hence it belongs to $\NonLGor(\OO)$.

\smallskip
It is interesting to note that the 4-dimensional subscheme $\NonLGor(\OO)$ misses the origin, i.e.,\ the point in $\BO$ which represents the monomial ideal  $I$ generated by $\partial\OO$.
The reason is that $I = \langle y^2, x^2, xy^2, x^2y \rangle$ is minimally 
generated by $\{y^2, x^2\}$, which is a regular sequence, hence $P/I$ is Gorenstein. 
We can check this directly, since one coefficient of $D$ is $c_{41}c_{42}-1$.
\end{example}

\medskip
We conclude this section by recalling that other loci are studied in~\cite{KLR-1}, in particular, the \textbf{Cayley-Bacharach locus} and the \textbf{strict complete intersection locus}.  
To achieve these results, the subscheme $\BOdf$ plays a fundamental role, and more 
tools are needed, such as the \textbf{affine Hilbert function} and suitable 
stratifications of the scheme $\BO$.

\newpage

\section{Questions and Problems}
\label{sec-Questions and Problems}   

The final part of this survey is devoted to several open questions and problems.

\medskip
\begin{myenumerate}
\item[\textbf{Q1.}] 
In Remark~\ref{rem-HilbConnected} we mentioned that the Hilbert schemes 
$\Hilb^\mu(\mathbb{A}^n_K)$ are connected. Moreover, 
Proposition~\ref{prop-homcommute}\,(5) shows that MaxDeg border 
basis schemes are connected. 
Are all border basis schemes connected?

\medskip
\item[\textbf{Q2.}] Are MaxDeg border basis schemes reduced? 

\medskip
\item[\textbf{Q3.}] What is the smallest value of $\mu$ for which the 
corresponding MaxDeg border basis scheme $\BO$ is reducible?
Recall that for large order ideals, some MaxDeg border basis schemes are known 
to be reducible (see~\cite{Iar, KR3}). 
By contrast, compare this with Proposition~\ref{prop-irreducible}.

\medskip
\item[\textbf{Q4.}]  In Subsection~\ref{subsec-Fibers} we studied the fibers of positive $P_0$-algebras.
In particular, Theorem~\ref{thm-fibers} examines properties related to regularity.
A natural next step would be to investigate the singularities of positive $P_0$-algebras.
A fist step is taken in~\cite{KR7}.

\medskip
\item[\textbf{Q5.}] Planar border basis schemes are smooth and irreducible
(Proposition~\ref{prop-irreducible}).
Moreover, if $K$ is a perfect field, planar MaxDeg border basis schemes are affine cells
(see Theorem~\ref{thm-bivariateaffinecells}).
Does Example~\ref{ex-BOdfnotPositive}  provide an instance of a planar 
border basis scheme that is not an affine cell?

\medskip
\item[\textbf{Q6.}]  For planar border basis schemes, we know that non-exposed 
indeterminates can be eliminated. What about the non-planar case?

\medskip
\item[\textbf{Q7.}]  It follows from Proposition~\ref{prop-irreducible}\,(3) that the 
$(2, 2, 2)$-Box is irreducible, hence it coincides with its principal component.
Thus, we have $\dim(\BO) = 3\times 8 = 24$. Moreover, a direct computation 
shows that the origin of $\BO$ is regular. 
Can we gain a deeper understanding of this scheme?

\medskip
\item[\textbf{Q8.}]  In Section~\ref{sec-Border Bases} 
we emphasized the suitability of border bases for handling symmetry.
However, if we have an order ideal $\OO$ which is invariant with respect to
a set of permutations of the indeterminates in $P$, the corresponding ideal
$I(\BO)$ is not necessarily invariant with respect to the indeterminates in $K[C]$.
Is there a natural way to modify $I(\BO)$ so as to obtain an invariant ideal?
 
 \medskip
\item[\textbf{Q9.}] In Section~\ref{sec-SubBBS} we have studied some important
subschemes of the border basis schemes. In Subsection~\ref{subsec-LocGor}
we described open, closed, and constructible subschemes parame-trizing
zero-dimensional schemes that share a given property.
It would be interesting to study further subschemes of this kind.
 
 \medskip
\item[\textbf{Q10.}] As far as I know, nobody has yet studied nested 
border basis schemes. Perhaps it is time to remedy this.

\end{myenumerate}

\newpage
\cocoasection{CoCoA Examples}
\label{sec-CoCoA Examples}
The following session shows how to use \cocoa\ to solve some computational problems.

\begin{example} \label{coex-bb-no-gb}
\par\noindent
\begin{lstlisting}[style=cocoaStyle]
Use P::=QQ[x,y]; -- polynomial ring over the rationals
OO:=[one(P), y, x, y^2, x^2]; OO; 
-- [1,  y,  x,  y^2,  x^2]
BB:=BBBorder(OO); BB;
-- [x*y,  y^3,  x*y^2,  x^2*y,  x^3]
I:=ideal(x*y - x^2 -y^2, x^3, x^2*y, x*y^2, y^3);
GF:=GroebnerFanIdeals(I); -- Groebner fan of I
U:=[GBasis(I) | I in GF];
indent(U);
-- [
--   [x^2 -x*y +y^2,  y^3,  x*y^2],
--   [y^2 -x*y +x^2,  x^3,  x^2*y]
-- ]
LL:=[[Gens(LT(I)), QuotientBasisSorted(I)] | I in GF];
indent(LL);
-- [
--   [[x^2,  y^3,  x*y^2],  [1,  y,  x,  y^2,  x*y]],
--   [[y^2,  x^3,  x^2*y],  [1,  x,  y,  x^2,  x*y]]
-- ]

---------------- formal multiplication matrices
-- M_x
x*OO[1] = OO[3];
x*OO[2] = x*y; -- hence  x*OO[2] --> OO[4] +OO[5]
x*OO[3] =OO[5];
x*OO[4] = BB[3]; -- hence x*OO[4] --> 0
x*OO[5] = BB[5]; -- hence x*OO[5] --> 0

Mx:=transposed(mat([ [0,0,1,0,0], [0,0,0,1,1], [0,0,0,0,1], [0,0,0,0,0], [0,0,0,0,0]   ])); Mx;
-- matrix(QQ,
--  [[0, 0, 0, 0, 0],
--   [0, 0, 0, 0, 0],
--   [1, 0, 0, 0, 0],
--   [0, 1, 0, 0, 0],
--   [0, 1, 1, 0, 0]])
-- /**/
------------------------------

-- My
y*OO[1] = OO[2];
y*OO[2] = OO[4];
y*OO[3] = x*y; -- hence  y*OO[3] --> OO[4] +OO[5]
y*OO[4] = BB[2]; -- hence y*OO[4] --> 0
y*OO[5] = BB[4]; -- hence y*OO[5] --> 0
My:=transposed(mat([ [0,1,0,0,0], [0,0,0,1,0], [0,0,0,1,1], [0,0,0,0,0], [0,0,0,0,0]   ])); My;

-- matrix(QQ,
--  [[0, 0, 0, 0, 0],
--   [1, 0, 0, 0, 0],
--   [0, 0, 0, 0, 0],
--   [0, 1, 1, 0, 0],
--   [0, 0, 1, 0, 0]])
-- /**/ 

Mx*My - My*Mx;
-- matrix(QQ,
--  [[0, 0, 0, 0, 0],
--   [0, 0, 0, 0, 0],
--   [0, 0, 0, 0, 0],
--   [0, 0, 0, 0, 0],
--   [0, 0, 0, 0, 0]])
-- /**/
------------------------------
\end{lstlisting}
\end{example}

\newpage
\begin{example}\label{coex-numstabil}
\par\noindent
\begin{lstlisting}[style=cocoaStyle]
A::=QQ[e]; -- add the indeterminate e to the field QQ
K:=NewFractionField(A); -- field of rational functions in e
Use P::=K[x,y];
f1:=(1/4)*x^2 +y^2 +e*x*y -1;
f2:=x^2 +(1/4)*y^2 +e*x*y -1;
J:=ideal(f1, f2);
GF:=GroebnerFanIdeals(J); -- Groebner fan of J
U:=[GBasis(I) | I in GF]; -- reduced Groebner bases
indent(U);
-- [
--   [x^2 -y^2,  x*y +(5/(4*e))*y^2 -1/e,  y^3 +((-16*e)/(16*e^2 -25))*x +(20/(16*e^2 -25))*y],
--   [y^2 -x^2,  x*y +(5/(4*e))*x^2 -1/e,  x^3 +((-16*e)/(16*e^2 -25))*y +(20/(16*e^2 -25))*x],
--   [y +((-16*e^2 +25)/(16*e))*x^3 +(-5/(4*e))*x,  x^4 +(40/(16*e^2 -25))*x^2 -16/(16*e^2 -25)],
--   [x +((-16*e^2 +25)/(16*e))*y^3 +(-5/(4*e))*y,  y^4 +(40/(16*e^2 -25))*y^2 -16/(16*e^2 -25)]
-- ]
LL:=[[Gens(LT(I)), QuotientBasisSorted(I)] | I in GF];
-- leading term ideals and quotient bases
indent(LL);
-- [
--   [[x^2,  x*y,  y^3],  [1,  y,  x,  y^2]],
--   [[y^2,  x*y,  x^3],  [1,  x,  y,  x^2]],
--   [[y,  x^4],          [1,  x,  x^2,  x^3]],
--   [[x,  y^4],          [1,  y,  y^2,  y^3]]
-- ]
\end{lstlisting}
\end{example}

\begin{example}\label{coex-22Box}
\par\noindent
\begin{lstlisting}[style=cocoaStyle]
Use P::=QQ[x,y];
OO:=[one(P), y, x, x*y]; OO;
-- [1,  y,  x,  x*y]
BO:=BBBorder(OO); BO;
-- [y^2,  x^2,  x*y^2,  x^2*y]
BBS:=BBRing(OO); Use BBS;
NumIndets(BBS);
-- 16
GMM:=GenMultMat(BBS,OO); -- generic multiplication matrices
Mx:=GMM[1];
My:=GMM[2];
Mx; My;
-- matrix( /*RingWithID(6, "QQ[c[1,1],c[1,2],c[1,3],c[1,4],c[2,1],c[2,2],c[2,3],c[2,4],c[3,1],c[3,2],c[3,3],c[3,4],c[4,1],c[4,2],c[4,3],c[4,4]]")*/
--  [[0, 0, c[1,2], c[1,4]],
--   [0, 0, c[2,2], c[2,4]],
--   [1, 0, c[3,2], c[3,4]],
--   [0, 1, c[4,2], c[4,4]]])
-- matrix( /*RingWithID(6, "QQ[c[1,1],c[1,2],c[1,3],c[1,4],c[2,1],c[2,2],c[2,3],c[2,4],c[3,1],c[3,2],c[3,3],c[3,4],c[4,1],c[4,2],c[4,3],c[4,4]]")*/
--  [[0, c[1,1], 0, c[1,3]],
--   [1, c[2,1], 0, c[2,3]],
--   [0, c[3,1], 0, c[3,3]],
--   [0, c[4,1], 1, c[4,3]]])

IBO:=NatIdealOfBBS(BBS,OO);
Ge:=Gens(IBO);
len(Ge);
-- 12
indent(Ge);
-- [
--   -c[1,2]*c[3,1] -c[1,4]*c[4,1] +c[1,3],
--   -c[2,2]*c[3,1] -c[2,4]*c[4,1] +c[2,3],
--   -c[3,1]*c[3,2] -c[3,4]*c[4,1] -c[1,1] +c[3,3],
--   -c[3,1]*c[4,2] -c[4,1]*c[4,4] -c[2,1] +c[4,3],
--   -c[1,1]*c[2,2] -c[1,3]*c[4,2] +c[1,4],
--   -c[2,1]*c[2,2] -c[2,3]*c[4,2] -c[1,2] +c[2,4],
--   -c[2,2]*c[3,1] -c[3,3]*c[4,2] +c[3,4],
--   -c[2,2]*c[4,1] -c[4,2]*c[4,3] -c[3,2] +c[4,4],
--   -c[1,1]*c[2,4] +c[1,2]*c[3,3] +c[1,4]*c[4,3] -c[1,3]*c[4,4],
--   -c[2,1]*c[2,4] +c[2,2]*c[3,3] +c[2,4]*c[4,3] -c[2,3]*c[4,4] -c[1,4],
--   -c[2,4]*c[3,1] +c[3,2]*c[3,3] +c[3,4]*c[4,3] -c[3,3]*c[4,4] +c[1,3],
--   -c[2,4]*c[4,1] +c[3,3]*c[4,2] +c[2,3] -c[3,4]
-- ]

A:=MDGrBBS(OO);
-- requires the packages NonNegGrading.cpkg5 and MaxDegBBS.cpkg5
MDBBS:=A.ring;   Use MDBBS;
GGe:=A.image;
WdInd:=[wdeg(X)[1] | X in indets(MDBBS)];
WdInd; -- total arrow degrees
-- [2,  2,  3,  3,  1,  1,  2,  2,  1,  1,  2,  2,  0,  0,  1,  1]

WdGGe:=[wdeg(X)[1] | X in GGe];
WdGGe; -- total arrow degrees
-- [3,  2,  2,  1,  3,  2,  2,  1,  4,  3,  3,  2]

[IsHomog(X) | X in GGe];
-- checking that the generators of IBO are homogeneous
-- [true,  true,  true,  true,  true,  true,
--  true,  true,  true,  true,  true,  true]
\end{lstlisting}
\end{example}

\bigskip
\begin{example}\label{coex-Zsepar}
\par\noindent
\begin{lstlisting}[style=cocoaStyle]
Use P::=QQ[x,y,z,w];
f_1:=x^3 -y +z;
f_2:=x*w^2 -x -y;
f_3:=x^2 -y*w;
L:=[f_1,f_2,f_3];
I:=ideal(L);
[RLF(P,X) | X in L]; -- linear parts
-- [-y +z,  -x -y,  0]
W:=mat([[0,1,2,0]]);
-- make (f_1,f_2) coherently Z-separating with Z=(y,z)

O:=MakeTermOrdMat(W);
R:=NewPolyRing(QQ,"x,y,z,w",O,1); -- P with the new term ordering
alpha:=PolyAlgebraHom(P,R,indets(R)); -- natural isomorphism

Use R;
J:=ideal(alpha(L));
GBasis(J);
-- [y -x*w^2 +x,  x*w^3 -x^2 -x*w,  z +x^3 -x*w^2 +x]
-- (f_1,f_2) is a coherently (y,z)-separating tuple

S::=QQ[x,w]; -- construct the re-embedding
beta:=PolyAlgebraHom(R,S, "x, x*w^2-x, -x^3 +x*w^2 -x, w");   -- beta is surjective
Use R;
beta(gens(J));
-- [0,  0,  -x*w^3 +x^2 +x*w]

Use S;
T:=S/ideal(x*w^3 -x^2 -x*w);
gamma:=CanonicalHom(S,T);
theta:=gamma(beta(alpha));
Use P;
ker(theta)=I;
-- true
-- hence we have constructed the isomorphism
-- P/I <--> QQ[x,w]/<x*w^3 -x^2 -x*w>
\end{lstlisting}
\end{example}

\bigskip
\begin{example}\label{coex-isotoK[x]}
\par\noindent
\begin{lstlisting}[style=cocoaStyle]
Use P::=QQ[x,y];
f:=x +2*x^8 +8*x^6*y +12*x^4*y^2 +8*x^2*y^3 +2*y^4;
I:=ideal(f);
Q::=QQ[x];
alpha:=PolyAlgebraHom(P,Q,"-2*x^4, x -4*x^8");
Use Q;
preimage0(alpha, x);
-- x^2 +y
ker(alpha)=I;
-- true      hence P/I is isomorphic to QQ[x]
\end{lstlisting}
\end{example}

\newpage
\begin{example}\label{coex-PropTrivBas}
\par\noindent
\begin{lstlisting}[style=cocoaStyle]
Use P::=QQ[x,y];
OO:=[one(P),y,x,y^2,x*y,x^2,y^3,x*y^2];
BO:=BBBorder(OO); BO;
-- [x^2*y,  x^3,  y^4,  x*y^3,  x^2*y^2]
BBS:=BBRing(OO); Use BBS; NumIndets(BBS);
-- 40
IBO:=NatIdealOfBBS(BBS,OO);
Ge:=Gens(IBO); len(Ge);
-- 32

CI:=ClassifyingIndets(BBS,OO);
-- same as CotEquivClasses, using CotEquivIndets

E0:=CI.TrivIndets; E0; len(E0);  -- trivial indeterminates
-- [c[1,1],  c[1,2],  c[1,3],  c[1,4],  c[1,5],  c[2,1],  c[2,2],  c[2,3],  c[2,4],  c[2,5],  c[3,1],  c[3,2],  c[3,3],  c[3,4],  c[3,5],  c[4,2],  c[4,4],  c[4,5],  c[5,5],  c[6,5]]
-- 20

Ep:=CI.ProperIndets; indent(Ep);  -- proper indeterminates
-- [c[5,1],  c[8,5]],
-- [c[4,3],  c[5,4]],
-- [c[4,1],  c[5,2],  c[7,5]]

Eb:=CI.BasicIndets; Eb; len(Eb);  -- basic indeterminates
-- 13

Z:=concat(E0,[c[5,1], c[4,3], c[4,1], c[5,2]]); Z;
-- [c[1,1],  c[1,2],  c[1,3],  c[1,4],  c[1,5],  c[2,1],  c[2,2],  c[2,3],  c[2,4],  c[2,5],  c[3,1],  c[3,2],  c[3,3],  c[3,4],  c[3,5],  c[4,2],  c[4,4],  c[4,5],  c[5,5],  c[6,5],  c[5,1],  c[4,3],  c[4,1],  c[5,2]]

Y:=diff(indets(BBS),Z); Y;
-- [c[5,3],  c[5,4],  c[6,1],  c[6,2],  c[6,3],  c[6,4],  c[7,1],  c[7,2],  c[7,3],  c[7,4],  c[7,5],  c[8,1],  c[8,2],  c[8,3],  c[8,4],  c[8,5]]
\end{lstlisting}
\end{example}

\bigskip
\begin{example}\label{coex-exposed}
\par\noindent
\begin{lstlisting}[style=cocoaStyle]
Use P::=QQ[x,y];
OO:=[one(P),y,x,y^2,y^3]; OO;
-- [1,  y,  x,  y^2,  y^3]
BO:=BBBorder(OO); BO;
-- [x*y,  x^2,  x*y^2,  y^4,  x*y^3]
BBS:=BBRing(OO); NumIndets(BBS);
-- 25
Use BBS;
IBO:=NatIdealOfBBS(BBS,OO);

PartialExposedIndets(BBS,OO,1); -- x-exposed indeterminates
-- [c[2,1],  c[2,4],  c[3,1],  c[3,4],  c[4,1],  c[4,4],  c[5,1],  c[5,4]]

PartialExposedIndets(BBS,OO,2); -- y-exposed indeterminates
-- [c[3,1],  c[3,2],  c[3,3],  c[3,5],  c[5,1],  c[5,2],  c[5,3],  c[5,5]]
\end{lstlisting}
\end{example}

\newpage
\begin{example}\label{coex-UMP}
\par\noindent
\begin{lstlisting}[style=cocoaStyle]
U:=mat([[0,0,1,1,2,2]]);
TeOr:=MakeTermOrdMat(U);
P:=NewPolyRing(QQ,"a,b,x,y,z,w",TeOr,1);
Use P;
g_1:=(1+a^2)*x +a^3*y;
g_2:=a*z +(1-a*b)*w -x^2;
I:=ideal(g_1,g_2);
IsHomog(I);
-- true

cmatG:=mat([[1+a^2,a^3,0,0],[0,0,a,1-a*b]]); cmatG;
-- matrix( /*RingWithID(31, "QQ[a,b,x,y,z,w]")*/
--  [[a^2 +1, a^3, 0, 0],
--   [0, 0, a, -a*b +1]])

I1:=ideal(1+a^2,a^3); GenRepr(1,I1);
-- [-a^2 +1,  a]

I2:=ideal(a,1-a*b); GenRepr(1,I2);
-- [b,  1]

IndP:=transposed(mat([[x,y,z,w]]));

B:=mat(
 [[-a^2 +1, -a^3,      0,         0],
  [a,       a^2 +1,    0,         0],
  [0,       0,         b,       a*b -1],
  [0,       0,         1,         a]]);
det(B);
-- 1

cmatG*B;
-- matrix( /*RingWithID(5, "QQ[a,b,x,y,z,w]")*/
--  [[1, 0, 0, 0],
--   [0, 0, 1, 0]])

B*IndP;
-- [[-a^3*y -a^2*x +x],
--  [a^2*y +a*x +y],
--  [a*b*w +b*z -w],
--  [a*w +z]]

cmatG*B*IndP;
-- matrix( /*RingWithID(5, "QQ[a,b,x,y,z,w]")*/
--  [[x],
--   [z]])

phi:=PolyAlgebraHom(P,P,
"a,b,-a^3*y -a^2*x +x,a^2*y +a*x +y,a*b*w +b*z -w,a*w +z");

phi(g_1);
-- x
phi(g_2);
-- -a^6*y^2 -2*a^5*x*y -a^4*x^2 +2*a^3*x*y +2*a^2*x^2 -x^2 +z

Subst(phi(g_2),x,0);
-- -a^6*y^2 +z

Q::=QQ[a,b,y,w];
eta:=PolyAlgebraHom(P,Q,"a,b,0,y,a^6*y^2,w");
ker(eta);
-- ideal(x, a^6*y^2 -z)
-- hence P/I is isomorphic to Q
\end{lstlisting}
\end{example}

\newpage
\begin{example}\label{coex-maxdeg-essential}
\par\noindent
\begin{lstlisting}[style=cocoaStyle]
Use P::=QQ[x,y];
OO:=[one(P),y,x,y^2,x*y,y^3]; OO;
-- [1,  y,  x,  y^2,  x*y,  y^3]
BO:=BBBorder(OO); BO;
-- [x^2,  x*y^2,  x^2*y,  y^4,  x*y^3]
BBS:=BBRing(OO); Use BBS; NumIndets(BBS);
-- 30
IBO:=NatIdealOfBBS(BBS,OO);
Ge:=Gens(IBO); len(Ge);
-- 24

ZeroTotArrDeg(BBS,OO);
-- [[c[4,1], 0], [c[5,1], 0], [c[6,2], 0], [c[6,3], 0]]

NegTotArrDeg(BBS,OO);
-- [[c[6,1], -1]]

BB:=BOhom(BBS,OO);
-- one component of the record is the homomorphism BBS --> BOhom
theta:=BB.homomorphism;
[theta(X) | X in Ge];
-- [0, 0, 0, 0, 0, -c[5,1]*c[6,2] -c[4,1] +c[6,3], 0, 0, 0, 0, 0, 0, 0, 0, 0, 0, 0, 0, 0, 0, 0, 0, 0, 0]

f:=-c[5,1]*c[6,2] -c[4,1] +c[6,3];
TotalArrowDegree(BBS,OO,f);
-- 0

[X in Ge | TotalArrowDegree(BBS,OO,X)=0];
-- [-c[5,1]*c[6,2] -c[6,1]*c[6,4] -c[4,1] +c[6,3]]
\end{lstlisting}
\end{example}

\newpage
\begin{example}\label{coex-L-shape}
We illustrate the re-embedding procedure and the subsequent use of the
Unimodular Matrix Problem (UMP) on a simple $L$-shape order ideal.

Let
$P=\QQ[x,y]$ and $\OO=\{1,y,x,y^2,x^2\}$.
The border of $\OO$ is
$\partial\OO=\{xy,y^3,xy^2,x^2y,x^3\}$.

The following Cocoa computation constructs the corresponding border basis
scheme, computes a maximal degree grading, performs a non-negative
re-embedding, and finally applies the UMP to obtain a further reduction.
\begin{lstlisting}[style=cocoaStyle]
Use P::=QQ[x,y];

OO:=[one(P), y, x, y^2, x^2]; OO;
-- [1, y, x, y^2, x^2]
BO:=BBBorder(OO); BO;
-- [x*y, y^3, x*y^2, x^2*y, x^3]

BBS:=BBRing(OO);
Use BBS;
NumIndets(BBS);
-- 25
IBO:=IdealOfBBS(BBS,OO);

A:=MDGrBBS(OO);  -- total arrow degree grading of the border basis scheme
MDBBS:=A.ring;
GGe:=A.image;

len(GGe);
-- 20
Use MDBBS;
I:=ideal(GGe);
W:=[wdeg(X)[1] | X in indets(MDBBS)]; W;
-- [2,3,3,3,3,1,2,2,2,2,1,2,2,2,2,0,1,1,1,1,0,1,1,1,1]

MG:=MinSubsetOfGens(Ideal(GGe)); -- use of MinsubSetOfGens

All:=AllBestCohSep(MDBBS,MG);  -- Search for the best coherent separating tuple

CC:=CartesianProduct(All[1],All[2],All[3]);

l:=len(CC);
MM:=NewList(l);
S:=NewList(l);
alpha:=NewList(l);
NGe:=NewList(l);
A:=NewList(l);
Su:=NewList(l);
PS:=NewList(l);

for i:=1 to l do
  MM[i]:=NonNegReEmbed(MDBBS,MG,CC[i]);
  S[i]:=MM[i].ring;
  alpha[i]:=MM[i].homomorphism;
  NGe[i]:=MM[i].image;
  A[i]:=reversed(MinSubsetOfGens(ideal(NGe[i])));
  Su[i]:=[len(support(X)) | X in A[i]];
  PS[i]:=sum(Su[i]);
endfor;

Su[12];Su[27]; 
-- [13,  58]
-- [13,  58]
-- these are the smallest
-- We choose CC[12]

Use MDBBS;
MM12:=NonNegReEmbed(MDBBS,GGe,CC[12]);
S:=MM12.ring;
alpha:=MM12.homomorphism;
AA:=[len(support(alpha(X))) | X in indets(MDBBS)];
NGe:=MM12.image;
J:=ideal(NGe);
MG:=reversed(MinSubsetOfGens(J));
--------------------------------------------------
-- Unimodular Matrix Problem
--------------------------------------------------

Use S;
[wdeg(X)[1]|X in indets(S)];
MG:=BringIn(S,MG);
LinZero(S,MG[1]);
LinZero(S,MG[2]);
[X in indets(S) | wdeg(X)[1]=0];
-- [c[4,1],c[5,1]]
[X in indets(S) | wdeg(X)[1]=1];
-- [c[2,1],c[4,2],c[4,3],c[4,5],
--  c[5,2],c[5,4],c[5,5]]
[X in indets(S) | wdeg(X)[1]=2];
-- [c[2,2],c[3,3],c[3,5]]
MA:=LinZeroMat(S,MG);

indent(MA);
-- [matrix( /*RingWithID(389, "QQ[c[2,1],c[2,2],c[3,3],c[3,5],c[4,1],c[4,2],c[4,3],c[4,5],c[5,1],c[5,2],c[5,4],c[5,5]]")*/
--  [[-c[4,1]^2*c[5,1]^2 -c[4,1]*c[5,1] -1, -c[4,1]^2*c[5,1] -c[4,1], -c[4,1]^2*c[5,1]^2 +1, -c[4,1]*c[5,1]^3 -c[5,1]^2, c[4,1]^3, c[4,1]^2*c[5,1] -c[4,1], c[4,1]*c[5,1]]]),
-- matrix( /*RingWithID(389, "QQ[c[2,1],c[2,2],c[3,3],c[3,5],c[4,1],c[4,2],c[4,3],c[4,5],c[5,1],c[5,2],c[5,4],c[5,5]]")*/
--  [[c[4,1]*c[5,1], -2*c[4,1]*c[5,1] +1, -c[5,1]^2]])]

A:=flatten(GetRows(MA[1]));
GBasis(ideal(A));

B:=flatten(GetRows(MA[2]));
GBasis(ideal(B));
-- [1]
-- [1]    -- The two matrices are unimodular

--------------------------------------------------
-- Quillen--Suslin reduction
--------------------------------------------------

QS1:=qsAlgorithm(MA[1],S);
QS2:=qsAlgorithm(MA[2],S);
NewMA:=[QS1,QS2];
QL:=qsList(S,MG,NewMA);
phi:=PolyAlgebraHom(S,S,QL);  -- construction of the automorphism
F1:=phi(MG[1]);
f:=Subst(phi(MG[2]),c[2,1],0);
-- The pair (F1,f) is coherently separating
-- with Z=[c[2,1],c[2,2]]

--------------------------------------------------
-- Final elimination of the separating indeterminates
--------------------------------------------------
A1:=0;
A2:=-f+c[2,2];
LL:=[A1,A2, c[3,3],c[3,5], c[4,1],c[4,2],c[4,3],c[4,5], c[5,1],c[5,2],c[5,4],c[5,5]];
Shat::=QQ[c[3,3],c[3,5],c[4,1],c[4,2],c[4,3],c[4,5],c[5,1],c[5,2],c[5,4],c[5,5]];
LL:=BringIn(Shat,LL);
theta:=PolyAlgebraHom(S,Shat,LL);
theta(F1); theta(f);
-- [0]
-- [0]

rho:=theta(alpha); -- Composition of the two homomorphisms
[len(support(rho(X))) | X in indets(MDBBS)];
-- [34,  298,  372,  342,  254,  0,  84,  89,  97,  104,  8,  89,  1,  36,  1,  1,  1,  1,  2,  1,  1,  1,  9,  1,  1]
\end{lstlisting}
\end{example}

\newpage
\begin{example}\label{coex-LocGor121}
\par\noindent
\begin{lstlisting}[style=cocoaStyle]
Use P::=QQ[x,y];
OO:=[one(P), y, x, x*y]; OO;
BO:=BBBorder(OO); BO;
-- [y^2,  x^2,  x*y^2,  x^2*y]
BBS:=BBRing(OO); NumIndets(BBS);
-- 16
A:=MDGrBBS(OO);  -- We introduce the non-negative grading
MDBBS:=A.ring;  Use MDBBS;
[wdeg(X)[1] | X in indets(MDBBS)];
-- [2,  2,  3,  3,  1,  1,  2,  2,  1,  1,  2,  2,  0,  0,  1,  1]
GGe:=A.image; --indent(GGe);
IBO:=ideal(GGe);
A:=GenMultMat(MDBBS,OO);
Mx:=A[1]; My:=A[2];
TMx:=transposed(Mx); TMy:=transposed(My);
TMxy:=transposed(My*Mx);
-- We extend QQ[C] adding the indeterminates  z_1,z_2,z_3,z_4
W:=mat([[2,  2,  3,  3,  1,  1,  2,  2,  1,  1,  2,  2,  0,  0,  1,  1,    1, 1, 1, 1]]);
O:=MakeTermOrdMat(W);
S:=NewPolyRing(QQ, "c[1,1],c[1,2],c[1,3],c[1,4],c[2,1],c[2,2],c[2,3],c[2,4],c[3,1],c[3,2],c[3,3],c[3,4],c[4,1],c[4,2],c[4,3],c[4,4], z_1,z_2,z_3,z_4", O, 1);
Use S;
STMx:=BringIn(S,TMx); STMy:=BringIn(S,TMy); STMxy:=BringIn(S,TMxy);
-- we bring the matrices to the new ring
Z:=ColMat([z_1,z_2,z_3,z_4]);
DZ:=BlockMat([[Z,STMx*Z,STMy*Z,STMxy*Z]]);
-- we construct DZ
D:=det(DZ);
Co:=CoefficientsWRT(D,[z_1,z_2,z_3,z_4]);
-- the coefficients of $D$ with respect to z_1,z_2,z_3,z_4
L:=flatten([X.coeff | X in Co],2); len(L);
-- a list of the coefficients
CoJ:=ideal(L);
-- the ideal in QQ[C] generated by L
GenCoJ:=Gens(CoJ);
GCoJ:=BringIn(MDBBS,GenCoJ);
-- the generators brought back to the graded ring of BBS
H:=ideal(GGe)+ideal(GCoJ);
t_0:=CpuTime(); dim(MDBBS/H); TimeFrom(t_0);
-- we compute the dimension
-- 4   OK
-- 0.061
--------------------------------------------------------------
Use P;
I:=intersection(ideal(x+1,y+1),ideal(x^2,x*y,y^2));
GB:=GBasis(I); GB;
QuotientBasis(I);
-- [x*y-y^2, x^2-y^2, y^3+y^2]
-- [1, y, y^2, x]
NF(one(P),I); NF(x,I); NF(y,I); NF(x*y,I);
-- we use the normal forms to get border basis
U:=ideal(y^2-x*y,x^2-x*y,x*y^2+x*y,x^2*y+x*y);
I=U;
-- true
Use MDBBS;
-- we prepare the substitution of the coordinate of the point Gamma which represents I
L:=[[c[1,1],0],[c[2,1],0],[c[3,1],0],[c[4,1],1],
   [c[1,2],0],[c[2,2],0],[c[3,2],0],[c[4,2],1],
   [c[1,3],0],[c[2,3],0],[c[3,3],0],[c[4,3],-1],
   [c[1,4],0],[c[2,4],0],[c[3,4],0],[c[4,4],-1]];
Ge:=Gens(H);
[Subst(X,L)|X in Ge];
-- [0,  0,  0,  0,  0,  0,  0,  0,  0,  0,  0,  0,  0,  0,  0,  0,  0,  0,  0,  0,  0,  0,  0,  0,  0,  0,  0,  0,  0,  0,  0,  0,  0,  0,  0,  0,  0,  0,  0,  0,  0,  0,  0,  0,  0,  0,  0]

Last(Gens(H));
-- c[4,1]*c[4,2]-1
\end{lstlisting}
\end{example}

\newpage
\addcontentsline{toc}{section}{References}
\bibliographystyle{amsplain}

\end{document}

%% file: pic64a.tex
\makebox[6 true cm]{
\beginpicture
\setcoordinatesystem units <10mm,10mm> point at 0  -3
\setplotarea x from -2.5  to 3, y from -2.5 to 3
\axis bottom  shiftedto y=0 / 
\axis left  shiftedto x=0 /  

\arrow <2mm> [.2,.67] from  2.5 0  to 3 0
\arrow <2mm> [.2,.67] from  0 2.5  to 0 3

\put {$\scriptstyle x$} [lt] <0.5mm,0.8mm> at 3.1 0
\put {$\scriptstyle y$} [rb] <1.7mm,0.7mm> at 0 3.1
\put {$\bullet$} at -0.8944  0.8944
\put {$\bullet$} at -0.8944  -0.8944
\put {$\bullet$} at  0.8944   0.8944
\put {$\bullet$} at  0.8944   -0.8944
\setquadratic
\setplotsymbol({\fiverm .})
\plot 
 0.0  1.0    %
 0.1  0.9987 %
 0.2  0.9945 %
 0.3  0.9887 %
 0.4  0.9798 %
 0.5  0.9682 %
 0.6  0.9539 %
 0.7  0.9367 %
 0.8  0.9165 %
 0.9  0.8930 %
 1.0  0.8660 %
 1.1  0.8352 %
 1.2  0.8    %
 1.3  0.7599 %
 1.4  0.7141 %
 1.5  0.6614 %
 1.6  0.6    %
 1.7  0.5268 %
 1.75 0.4841 %
 1.8  0.4359 %
 1.85 0.38   %
 1.9  0.3122 %
 1.93 0.2622 %
 1.95 0.2222 %
 1.98 0.1411 %
 1.99 0.1    %
 2.0  0.0 %
/
\plot 
 0.0  -1.0    %
 0.1  -0.9987 %
 0.2  -0.9945 %
 0.3  -0.9887 %
 0.4  -0.9798 %
 0.5  -0.9682 %
 0.6  -0.9539 %
 0.7  -0.9367 %
 0.8  -0.9165 %
 0.9  -0.8930 %
 1.0  -0.8660 %
 1.1  -0.8352 %
 1.2  -0.8    %
 1.3  -0.7599 %
 1.4  -0.7141 %
 1.5  -0.6614 %
 1.6  -0.6    %
 1.7  -0.5268 %
 1.75 -0.4841 %
 1.8  -0.4359 %
 1.85 -0.38   %
 1.9  -0.3122 %
 1.93 -0.2622 %
 1.95 -0.2222 %
 1.98 -0.1411 %
 1.99 -0.1    %
 2.0  0.0 %
/
\plot 
 0.0  -1.0    %
 -0.1  -0.9987 %
 -0.2  -0.9945 %
 -0.3  -0.9887 %
 -0.4  -0.9798 %
 -0.5  -0.9682 %
 -0.6  -0.9539 %
 -0.7  -0.9367 %
 -0.8  -0.9165 %
 -0.9  -0.8930 %
 -1.0  -0.8660 %
 -1.1  -0.8352 %
 -1.2  -0.8    %
 -1.3  -0.7599 %
 -1.4  -0.7141 %
 -1.5  -0.6614 %
 -1.6  -0.6    %
 -1.7  -0.5268 %
 -1.75 -0.4841 %
 -1.8  -0.4359 %
 -1.85 -0.38   %
 -1.9  -0.3122 %
 -1.93 -0.2622 %
 -1.95 -0.2222 %
 -1.98 -0.1411 %
 -1.99 -0.1    %
 -2.0  0.0 %
/
\plot 
 -0.0  1.0    %
 -0.1  0.9987 %
 -0.2  0.9945 %
 -0.3  0.9887 %
 -0.4  0.9798 %
 -0.5  0.9682 %
 -0.6  0.9539 %
 -0.7  0.9367 %
 -0.8  0.9165 %
 -0.9  0.8930 %
 -1.0  0.8660 %
 -1.1  0.8352 %
 -1.2  0.8    %
 -1.3  0.7599 %
 -1.4  0.7141 %
 -1.5  0.6614 %
 -1.6  0.6    %
 -1.7  0.5268 %
 -1.75 0.4841 %
 -1.8  0.4359 %
 -1.85 0.38   %
 -1.9  0.3122 %
 -1.93 0.2622 %
 -1.95 0.2222 %
 -1.98 0.1411 %
 -1.99 0.1    %
 -2.0  0.0 %
/
\plot 
1.0 0.0   %
0.9987 0.1 %
0.9945 0.2 %
0.9887 0.3 %
0.9798 0.4 %
0.9682 0.5 %
0.9539 0.6 %
0.9367 0.7 %
0.9165 0.8 %
0.8930 0.9 %
0.8660 1.0 %
0.8352 1.1 %
0.8    1.2 %
0.7599 1.3 %
0.7141 1.4 %
0.6614 1.5 %
0.6    1.6 %
0.5268 1.7 %
0.4841 1.75 %
0.4359 1.8 %
0.38   1.85 %
0.3122 1.9 %
0.2622 1.93 %
0.2222 1.95 %
0.1411 1.98 %
0.1    1.99 %
0.0 2.0 %
/
\plot 
1.0    0.0   %
0.9987 -0.1 %
0.9945 -0.2 %
0.9887 -0.3 %
0.9798 -0.4 %
0.9682 -0.5 %
0.9539 -0.6 %
0.9367 -0.7 %
0.9165 -0.8 %
0.8930 -0.9 %
0.8660 -1.0 %
0.8352 -1.1 %
0.8    -1.2 %
0.7599 -1.3 %
0.7141 -1.4 %
0.6614 -1.5 %
0.6    -1.6 %
0.5268 -1.7 %
0.4841 -1.75 %
0.4359 -1.8 %
0.38   -1.85 %
0.3122 -1.9 %
0.2622 -1.93 %
0.2222 -1.95 %
0.1411 -1.98 %
0.1    -1.99 %
0.0    -2.0 %
/
\plot 
-1.0 0.0   %
-0.9987 0.1 %
-0.9945 0.2 %
-0.9887 0.3 %
-0.9798 0.4 %
-0.9682 0.5 %
-0.9539 0.6 %
-0.9367 0.7 %
-0.9165 0.8 %
-0.8930 0.9 %
-0.8660 1.0 %
-0.8352 1.1 %
-0.8    1.2 %
-0.7599 1.3 %
-0.7141 1.4 %
-0.6614 1.5 %
-0.6    1.6 %
-0.5268 1.7 %
-0.4841 1.75 %
-0.4359 1.8 %
-0.38   1.85 %
-0.3122 1.9 %
-0.2622 1.93 %
-0.2222 1.95 %
-0.1411 1.98 %
-0.1    1.99 %
0.0 2.0 %
/
\plot 
-1.0    0.0   %
-0.9987 -0.1 %
-0.9945 -0.2 %
-0.9887 -0.3 %
-0.9798 -0.4 %
-0.9682 -0.5 %
-0.9539 -0.6 %
-0.9367 -0.7 %
-0.9165 -0.8 %
-0.8930 -0.9 %
-0.8660 -1.0 %
-0.8352 -1.1 %
-0.8    -1.2 %
-0.7599 -1.3 %
-0.7141 -1.4 %
-0.6614 -1.5 %
-0.6    -1.6 %
-0.5268 -1.7 %
-0.4841 -1.75 %
-0.4359 -1.8 %
-0.38   -1.85 %
-0.3122 -1.9 %
-0.2622 -1.93 %
-0.2222 -1.95 %
-0.1411 -1.98 %
-0.1    -1.99 %
0.0     -2.0 %
/
\endpicture
} 

%% file: pic64b.tex
\makebox[6 true cm]{

\beginpicture
\setcoordinatesystem units <10mm,10mm> point at 0  -3
\setplotarea x from -2.5  to 3, y from -2.5 to 3
\axis bottom  shiftedto y=0 / 
\axis left  shiftedto x=0 /  

\arrow <2mm> [.2,.67] from  2.5 0  to 3 0
\arrow <2mm> [.2,.67] from  0 2.5  to 0 3

\put {$\scriptstyle x$} [lt] <0.5mm,0.8mm> at 3.1 0
\put {$\scriptstyle y$} [rb] <1.7mm,0.7mm> at 0 3.1
\put {$\scriptstyle\bullet$} at -0.8944  0.8944
\put {$\scriptstyle\bullet$} at -0.8944  -0.8944
\put {$\scriptstyle\bullet$} at  0.8944   0.8944
\put {$\scriptstyle\bullet$} at  0.8944   -0.8944
\put {$\scriptstyle\bullet$} at -0.9759  0.9759
\put {$\scriptstyle\bullet$} at -0.8305  -0.8305
\put {$\scriptstyle\bullet$} at  0.8305   0.8305
\put {$\scriptstyle\bullet$} at  0.9759   -0.9759
\setquadratic
\setplotsymbol({\fiverm .})
\plot 
 0.0  1.0    %
 0.1  0.9987 %
 0.2  0.9945 %
 0.3  0.9887 %
 0.4  0.9798 %
 0.5  0.9682 %
 0.6  0.9539 %
 0.7  0.9367 %
 0.8  0.9165 %
 0.9  0.8930 %
 1.0  0.8660 %
 1.1  0.8352 %
 1.2  0.8    %
 1.3  0.7599 %
 1.4  0.7141 %
 1.5  0.6614 %
 1.6  0.6    %
 1.7  0.5268 %
 1.75 0.4841 %
 1.8  0.4359 %
 1.85 0.38   %
 1.9  0.3122 %
 1.93 0.2622 %
 1.95 0.2222 %
 1.98 0.1411 %
 1.99 0.1    %
 2.0  0.0 %
/
\plot 
 0.0  -1.0    %
 0.1  -0.9987 %
 0.2  -0.9945 %
 0.3  -0.9887 %
 0.4  -0.9798 %
 0.5  -0.9682 %
 0.6  -0.9539 %
 0.7  -0.9367 %
 0.8  -0.9165 %
 0.9  -0.8930 %
 1.0  -0.8660 %
 1.1  -0.8352 %
 1.2  -0.8    %
 1.3  -0.7599 %
 1.4  -0.7141 %
 1.5  -0.6614 %
 1.6  -0.6    %
 1.7  -0.5268 %
 1.75 -0.4841 %
 1.8  -0.4359 %
 1.85 -0.38   %
 1.9  -0.3122 %
 1.93 -0.2622 %
 1.95 -0.2222 %
 1.98 -0.1411 %
 1.99 -0.1    %
 2.0  0.0 %
/
\plot 
 0.0  -1.0    %
 -0.1  -0.9987 %
 -0.2  -0.9945 %
 -0.3  -0.9887 %
 -0.4  -0.9798 %
 -0.5  -0.9682 %
 -0.6  -0.9539 %
 -0.7  -0.9367 %
 -0.8  -0.9165 %
 -0.9  -0.8930 %
 -1.0  -0.8660 %
 -1.1  -0.8352 %
 -1.2  -0.8    %
 -1.3  -0.7599 %
 -1.4  -0.7141 %
 -1.5  -0.6614 %
 -1.6  -0.6    %
 -1.7  -0.5268 %
 -1.75 -0.4841 %
 -1.8  -0.4359 %
 -1.85 -0.38   %
 -1.9  -0.3122 %
 -1.93 -0.2622 %
 -1.95 -0.2222 %
 -1.98 -0.1411 %
 -1.99 -0.1    %
 -2.0  0.0 %
/
\plot 
 -0.0  1.0    %
 -0.1  0.9987 %
 -0.2  0.9945 %
 -0.3  0.9887 %
 -0.4  0.9798 %
 -0.5  0.9682 %
 -0.6  0.9539 %
 -0.7  0.9367 %
 -0.8  0.9165 %
 -0.9  0.8930 %
 -1.0  0.8660 %
 -1.1  0.8352 %
 -1.2  0.8    %
 -1.3  0.7599 %
 -1.4  0.7141 %
 -1.5  0.6614 %
 -1.6  0.6    %
 -1.7  0.5268 %
 -1.75 0.4841 %
 -1.8  0.4359 %
 -1.85 0.38   %
 -1.9  0.3122 %
 -1.93 0.2622 %
 -1.95 0.2222 %
 -1.98 0.1411 %
 -1.99 0.1    %
 -2.0  0.0 %
/
\plot 
1.0 0.0   %
0.9987 0.1 %
0.9945 0.2 %
0.9887 0.3 %
0.9798 0.4 %
0.9682 0.5 %
0.9539 0.6 %
0.9367 0.7 %
0.9165 0.8 %
0.8930 0.9 %
0.8660 1.0 %
0.8352 1.1 %
0.8    1.2 %
0.7599 1.3 %
0.7141 1.4 %
0.6614 1.5 %
0.6    1.6 %
0.5268 1.7 %
0.4841 1.75 %
0.4359 1.8 %
0.38   1.85 %
0.3122 1.9 %
0.2622 1.93 %
0.2222 1.95 %
0.1411 1.98 %
0.1    1.99 %
0.0 2.0 %
/
\plot 
1.0    0.0   %
0.9987 -0.1 %
0.9945 -0.2 %
0.9887 -0.3 %
0.9798 -0.4 %
0.9682 -0.5 %
0.9539 -0.6 %
0.9367 -0.7 %
0.9165 -0.8 %
0.8930 -0.9 %
0.8660 -1.0 %
0.8352 -1.1 %
0.8    -1.2 %
0.7599 -1.3 %
0.7141 -1.4 %
0.6614 -1.5 %
0.6    -1.6 %
0.5268 -1.7 %
0.4841 -1.75 %
0.4359 -1.8 %
0.38   -1.85 %
0.3122 -1.9 %
0.2622 -1.93 %
0.2222 -1.95 %
0.1411 -1.98 %
0.1    -1.99 %
0.0    -2.0 %
/
\plot 
-1.0 0.0   %
-0.9987 0.1 %
-0.9945 0.2 %
-0.9887 0.3 %
-0.9798 0.4 %
-0.9682 0.5 %
-0.9539 0.6 %
-0.9367 0.7 %
-0.9165 0.8 %
-0.8930 0.9 %
-0.8660 1.0 %
-0.8352 1.1 %
-0.8    1.2 %
-0.7599 1.3 %
-0.7141 1.4 %
-0.6614 1.5 %
-0.6    1.6 %
-0.5268 1.7 %
-0.4841 1.75 %
-0.4359 1.8 %
-0.38   1.85 %
-0.3122 1.9 %
-0.2622 1.93 %
-0.2222 1.95 %
-0.1411 1.98 %
-0.1    1.99 %
0.0 2.0 %
/
\plot 
-1.0    0.0   %
-0.9987 -0.1 %
-0.9945 -0.2 %
-0.9887 -0.3 %
-0.9798 -0.4 %
-0.9682 -0.5 %
-0.9539 -0.6 %
-0.9367 -0.7 %
-0.9165 -0.8 %
-0.8930 -0.9 %
-0.8660 -1.0 %
-0.8352 -1.1 %
-0.8    -1.2 %
-0.7599 -1.3 %
-0.7141 -1.4 %
-0.6614 -1.5 %
-0.6    -1.6 %
-0.5268 -1.7 %
-0.4841 -1.75 %
-0.4359 -1.8 %
-0.38   -1.85 %
-0.3122 -1.9 %
-0.2622 -1.93 %
-0.2222 -1.95 %
-0.1411 -1.98 %
-0.1    -1.99 %
0.0     -2.0 %
/
\plot
0.0  1.0000 %
0.1  0.9900 %
0.2  0.9752 %
0.3  0.9591 %
0.4  0.9406 %
0.5  0.9195 %
0.6  0.8958 %
0.7  0.8694 %
0.8  0.8400 %
0.9  0.8076 %
1.0  0.7718 %
1.1  0.7324 %
1.2  0.6889 %
1.3  0.6410 %
1.4  0.5877 %
1.5  0.5282 %
1.6  0.4610 %
1.7  0.3835 %
1.75 0.3398 %
1.8  0.2916 %
1.85 0.2376 %
1.9  0.1755 %
1.93 0.1326 %
1.95 0.1066 %
1.98 0.0450 %
1.99 0.0236 %
2.0  0.00   %
/
\plot
0.0  -1.0000 %
0.1  -1.0100 %
0.2  -1.0152 %
0.3  -1.0191 %
0.4  -1.0206 %
0.5  -1.0195 %
0.6  -1.0158 %
0.7  -1.0094 %
0.8  -1.0000 %
0.9  -0.9876 %
1.0  -0.9718 %
1.1  -0.9524 %
1.2  -0.9289 %
1.3  -0.9010 %
1.4  -0.8677 %
1.5  -0.8282 %
1.6  -0.7810 %
1.7  -0.7235 %
1.75 -0.6898 %
1.8  -0.6516 %
1.85 -0.6076 %
1.9  -0.5555 %
1.93 -0.5186 %
1.95 -0.4966 %
1.98 -0.4410 %
1.99 -0.4216 %
2.0  -0.4000 %
2.01 -0.3753 %
2.02 -0.3459 %
2.03 -0.3078 %
2.04 -0.2389 %
2.04124  -0.20412 %
2.04 -0.1691 %
2.03 -0.0982 %
2.02 -0.0581 %
2.01 -0.0267 %
2.0  0.00 %
/
\plot
-0.0  -1.0000 %
-0.1  -0.9900 %
-0.2  -0.9752 %
-0.3  -0.9591 %
-0.4  -0.9406 %
-0.5  -0.9195 %
-0.6  -0.8958 %
-0.7  -0.8694 %
-0.8  -0.8400 %
-0.9  -0.8076 %
-1.0  -0.7718 %
-1.1  -0.7324 %
-1.2  -0.6889 %
-1.3  -0.6410 %
-1.4  -0.5877 %
-1.5  -0.5282 %
-1.6  -0.4610 %
-1.7  -0.3835 %
-1.75 -0.3398 %
-1.8  -0.2916 %
-1.85 -0.2376 %
-1.9  -0.1755 %
-1.93 -0.1326 %
-1.95 -0.1066 %
-1.98 -0.0450 %
-1.99 -0.0236 %
-2.0  -0.00   %
/
\plot
-0.0  1.0000 %
-0.1  1.0100 %
-0.2  1.0152 %
-0.3  1.0191 %
-0.4  1.0206 %
-0.5  1.0195 %
-0.6  1.0158 %
-0.7  1.0094 %
-0.8  1.0000 %
-0.9  0.9876 %
-1.0  0.9718 %
-1.1  0.9524 %
-1.2  0.9289 %
-1.3  0.9010 %
-1.4  0.8677 %
-1.5  0.8282 %
-1.6  0.7810 %
-1.7  0.7235 %
-1.75 0.6898 %
-1.8  0.6516 %
-1.85 0.6076 %
-1.9  0.5555 %
-1.93 0.5186 %
-1.95 0.4966 %
-1.98 0.4410 %
-1.99 0.4216 %
-2.0  0.4000 %
-2.01 0.3753 %
-2.02 0.3459 %
-2.03 0.3078 %
-2.04 0.2389 %
-2.04124  0.20412 %
-2.04 0.1691 %
-2.03 0.0982 %
-2.02 0.0581 %
-2.01 0.0267 %
-2.0  0.00 %
/
\plot
1.0000 0.0 %
0.9900 0.1 %
0.9752 0.2 %
0.9591 0.3 %
0.9406 0.4 %
0.9195 0.5 %
0.8958 0.6 %
0.8694 0.7 %
0.8400 0.8 %
0.8076 0.9 %
0.7718 1.0 %
0.7324 1.1 %
0.6889 1.2 %
0.6410 1.3 %
0.5877 1.4 %
0.5282 1.5 %
0.4610 1.6 %
0.3835 1.7 %
0.3398 1.75 %
0.2916 1.8 %
0.2376 1.85 %
0.1755 1.9 %
0.1326 1.93 %
0.1066 1.95 %
0.0450 1.98 %
0.0236 1.99 %
0.00   2.0 %
/
\plot
-1.0000 0.0 %
-1.0100 0.1 %
-1.0152 0.2 %
-1.0191 0.3 %
-1.0206 0.4 %
-1.0195 0.5 %
-1.0158 0.6 %
-1.0094 0.7 %
-1.0000 0.8 %
-0.9876 0.9 %
-0.9718 1.0 %
-0.9524 1.1 %
-0.9289 1.2 %
-0.9010 1.3 %
-0.8677 1.4 %
-0.8282 1.5 %
-0.7810 1.6 %
-0.7235 1.7 %
-0.6898 1.75 %
-0.6516 1.8 %
-0.6076 1.85 %
-0.5555 1.9 %
-0.5186 1.93 %
-0.4966 1.95 %
-0.4410 1.98 %
-0.4216 1.99 %
-0.4000 2.0 %
-0.3753 2.01 %
-0.3459 2.02 %
-0.3078 2.03 %
-0.2389 2.04 %
-0.20412 2.04124 %
-0.1691 2.04 %
-0.0982 2.03 %
-0.0581 2.02 %
-0.0267 2.01 %
0.00 2.0 %
/
\plot
-1.0000 -0.0 %
-0.9900 -0.1 %
-0.9752 -0.2 %
-0.9591 -0.3 %
-0.9406 -0.4 %
-0.9195 -0.5 %
-0.8958 -0.6 %
-0.8694 -0.7 %
-0.8400 -0.8 %
-0.8076 -0.9 %
-0.7718 -1.0 %
-0.7324 -1.1 %
-0.6889 -1.2 %
-0.6410 -1.3 %
-0.5877 -1.4 %
-0.5282 -1.5 %
-0.4610 -1.6 %
-0.3835 -1.7 %
-0.3398 -1.75 %
-0.2916 -1.8 %
-0.2376 -1.85 %
-0.1755 -1.9 %
-0.1326 -1.93 %
-0.1066 -1.95 %
-0.0450 -1.98 %
-0.0236 -1.99 %
-0.00   -2.0 %
/
\plot
1.0000 -0.0 %
1.0100 -0.1 %
1.0152 -0.2 %
1.0191 -0.3 %
1.0206 -0.4 %
1.0195 -0.5 %
1.0158 -0.6 %
1.0094 -0.7 %
1.0000 -0.8 %
0.9876 -0.9 %
0.9718 -1.0 %
0.9524 -1.1 %
0.9289 -1.2 %
0.9010 -1.3 %
0.8677 -1.4 %
0.8282 -1.5 %
0.7810 -1.6 %
0.7235 -1.7 %
0.6898 -1.75 %
0.6516 -1.8 %
0.6076 -1.85 %
0.5555 -1.9 %
0.5186 -1.93 %
0.4966 -1.95 %
0.4410 -1.98 %
0.4216 -1.99 %
0.4000 -2.0 %
0.3753 -2.01 %
0.3459 -2.02 %
0.3078 -2.03 %
0.2389 -2.04 %
0.20412 -2.04124 %
0.1691 -2.04 %
0.0982 -2.03 %
0.0581 -2.02 %
0.0267 -2.01 %
0.00   -2.0 %
/

\endpicture
} 